\documentclass[11pt,reqno]{amsart}

\usepackage{amsmath, amsfonts, amsthm, amssymb, color}
\usepackage{amsmath,amsthm,amsfonts,amssymb,mathrsfs}
\usepackage{graphicx,xcolor}
\usepackage[mathscr]{euscript}
\usepackage{hyperref}
\allowdisplaybreaks

 \theoremstyle{definition}

 \numberwithin{equation}{section}

\def\Dv{\mathrm{div}}

\def\Trucan{\Phi_R^{\mathbf{u}_n, Q_n}}
\def\Truca{\Phi_R^{\mathbf{u}, Q}}
\def\Trucaa{\Phi_R^{\mathbf{u}_1, Q_1}}
\def\Trucab{\Phi_R^{\mathbf{u}_2, Q_2}}

\allowdisplaybreaks

\def\XXint#1#2#3{{\setbox0=\hbox{$#1{#2#3}{\int}$ }
\vcenter{\hbox{$#2#3$ }}\kern-.6\wd0}}

\newtheorem{theorem}{Theorem}[section]
\newtheorem{lemma}[theorem]{Lemma}

\newtheorem{proposition}[theorem]{Proposition}

\theoremstyle{definition}
\newtheorem{definition}[theorem]{Definition}

\theoremstyle{remark}
\newtheorem{remark}[theorem]{Remark}

\author[Z. Qiu]{Zhaoyang Qiu}
\address{School of Mathematics and Statistics, Huazhong University of Science and Technology, Wuhan, 430074, China.}
\email{zhqmath@163.com}

\author[Y. Wang]{Yixuan Wang}
\address{Department of Mathematics, University of Pittsburgh, Pittsburgh, 15260, USA.}
\email{YIW119@pitt.edu}

\title[Strong solution for compressible liquid crystal system with random force]
{Strong solution for compressible liquid crystal system with random force}


\date{\today}

\begin{document}
\begin{abstract}
We study the  three-dimensional compressible Navier-Stokes equations coupled with the $Q$-tensor equation perturbed by a multiplicative stochastic force, which describes the motion of nematic liquid crystal flows. The local existence and uniqueness of strong pathwise solution up to a positive stopping time is established where ``strong" is in both PDE and probability sense. The proof relies on the Galerkin approximation scheme, stochastic compactness, identification of the limit, uniqueness and a cutting-off argument. In the stochastic setting, we  develop an extra layer approximation to overcome the difficulty arising from the stochastic integral while constructing the approximate solution. Due to the complex structure of the coupled system, the estimates of the high-order items are also the challenging part in the article.
\end{abstract}

\maketitle

\section{{\bf Introduction}}

 Liquid crystal is a kind of material whose mechanical properties and symmetry properties are intermediate between those of a liquid and those of a crystal. The complex structure of liquid crystals made it the ideal material for the study of topological defects. As a result, several mathematical models have been brought out to describe the dynamics of a liquid crystal. For example, in \cite{DGPG}, the Ericksen-Leslie-Parodi system has been used to model the flow of   liquid crystals,
based on  the fact that a nematic flow is very similar to a conventional liquid with molecules of similar size. The challenge is, the flow would disturb the alignment, thus a new flow in the nematic is induced. In order to analyze the coupling between orientation and flow, a macroscopic approach has been used, and a direction field $\mathbf{d}$ with unit length is adopted to describe the local state of alignment. 
However, the model is restricted to an uniaxial order parameter field of constant magnitude.

 In an effort to describe the motion of biaxial liquid crystals, a tensor order parameter $Q$ replacing the unit vector $\mathbf{d}$ was brought up in \cite{BAN, DGPG} 
 to describe the primary and secondary directions of nematic alignments along with variations in the degree of nematic order, which reflects better the properties of   nematic liquid crystals and can be modeled by the Navier-Stokes equations governing the fluid velocity coupled with a parabolic equation of $Q$-tensor;
 see \cite{Ball,Ball1,Maju} for further background discussions. The compressible case we focus on reads as
  \begin{equation}\label{q}
    \begin{cases} d\rho + \Dv_x(\rho \mathbf{u} ) dt = 0,\\
        d(\rho \mathbf{u}) + \Dv_x(\rho \mathbf{u}\otimes \mathbf{u}) dt+\nabla_x pdt\\
        \quad=\mathcal{L}\mathbf{u}dt-\Dv_x(L\nabla_x Q\odot \nabla_x Q-\mathcal{F}(Q){\rm I}_3)dt
        +\Dv_x(Q\mathcal{H}(Q)-\mathcal{H}(Q)Q)dt, \\
        dQ+\mathbf{u}\cdot\nabla_x Qdt-(\Theta Q-Q\Theta) dt=\Gamma\mathcal{H}(Q)dt,
    \end{cases}
    \end{equation}
where $\rho, \mathbf{u} $ denote the density, and the flow velocity,  respectively; $p(\rho)=A\rho^\gamma$ stands for the pressure with the adiabatic exponent $\gamma>1$, $A>0$ is the squared reciprocal of the Mach number. The nematic tensor order parameter $Q$ is a traceless and $3\times 3$ symmetric matrix. Furthermore, $\mathcal{L}$ stands for the Lam\'e operator
  $$\mathcal{L}\mathbf{u}=\upsilon\triangle \mathbf{u}+(\upsilon+\lambda)\nabla\Dv_x\mathbf{u},$$
  where $\upsilon>0, \lambda\geq 0$ are shear viscosity and bulk viscosity coefficient of the fluid, respectively.
The term $\nabla_x Q\odot \nabla_x Q$ stands for the $3\times 3$ matrix with its $(i,j)$-th entry  defined by
  $$(\nabla_x Q\odot \nabla_x Q)_{ij}=\sum_{k,l=1}^3\partial_iQ_{kl}\partial_jQ_{kl},$$
 and ${\rm I}_3$ stands for the $3\times 3$ identity matrix. Define the free energy density of the director field $\mathcal{F}(Q)$
  $$\mathcal{F}(Q)=\frac{L}{2}|\nabla_x Q|^2+\frac{a}{2}{\rm tr}(Q^2)-\frac{b}{3}{\rm tr}(Q^3)+\frac{c}{4}{\rm tr}^2(Q^2),$$
and denote
  $$\Gamma\mathcal{H}(Q)=\Gamma L\triangle Q+\Gamma\left(-aQ+b\left[Q^2-\frac{{\rm I}_3}{3}{\rm tr}(Q^2)\right]-cQ{\rm tr}(Q^2)\right)=:\Gamma L\triangle Q+\mathcal{K}(Q).$$
 The coefficients in the formula are elastic constants: $L>0$, $\Gamma>0$, $a\in\mathbb{R}$, $b>0$ and $c>0$, which are dependent on the material.
 Finally $\Theta=\frac{\nabla_x\mathbf{u}-\nabla_x\mathbf{u}^\top} 2$ 
 is the skew-symmetric part of the rate of strain tensor.  
 From the specific form $\mathcal{K}(Q)$, we remark that
 $$Q\mathcal{H}(Q)-\mathcal{H}(Q)Q= L(Q\triangle Q-\triangle QQ).$$

The PDEs perturbed randomly are considered as a primary tool in the modeling of uncertainty, especially while describing fundamental phenomenon in physics, climate dynamics, communication systems, nanocomposites and gene regulation systems. Hence, the study of the well-posedness and dynamical behaviour of PDEs subject  to the noise which is largely applied to the theoretical and practical areas has drawn a lot of attention. Here, we consider the system (\ref{q}) driven by a multiplicative noise:   
    \begin{equation}\label{qn}
    \begin{cases} d\rho + \Dv_x(\rho \mathbf{u} ) dt = 0,\\
        d(\rho \mathbf{u}) + \Dv_x(\rho \mathbf{u}\otimes \mathbf{u}) dt+A\nabla_x \rho^\gamma dt\\
        \quad=\mathcal{L}\mathbf{u}dt-\Dv_x(L\nabla_x Q\odot \nabla_x Q-\mathcal{F}(Q){\rm I}_3)dt
        +L\Dv_x(Q\triangle Q-\triangle QQ)dt\\ \quad\quad+\mathbb{G}(\rho,\rho \mathbf{u})dW, \\
        dQ+\mathbf{u}\cdot\nabla_x Qdt-(\Theta Q-Q\Theta) dt=\Gamma\mathcal{H}(Q)dt,
    \end{cases}
    \end{equation}
where $W$ is a cylindrical Wiener process which will be introduced later. The system is equipped with the initial data
 \begin{equation}\label{ic}
      \rho(0,x)=\rho_0(x),~~\mathbf{u}(0,x)=\mathbf{u}_0(x),~~Q(0,x)=Q_0(x),
    \end{equation}
and the periodic boundary, where each period is a cube $\mathbb{T}\subset \mathbb{R}^3$ defined as follows
\begin{align}\label{1.4}
 \mathbb{T}=(-\pi,\pi)|_{\{-\pi,\pi\}^3}.
\end{align}

Regarding the incompressible $Q$-tensor liquid crystal framework, Paicu-Zarnescu \cite{PMZAE} have proved the existence of a global weak solution in both 2D and 3D cases,  and also proved the global regularity   and the weak-strong uniqueness in 2D. Then, Paicu-Zarnescu continued his study and got the same results in \cite{PMZAG} for the full system. De Anna \cite{DAF} extended the result \cite{PMZAE} to the low regularity space $W^s$ for $0<s<1$, which filled the gap in \cite{PMZAE}. Wilkinson \cite{WMS} obtained the existence and the regularity property for weak solution in the general d-dimensional case in the presence of a singular potential. The existence of a global in time weak solution for system with thermal effects is proved in \cite{FERS}. The existence and uniqueness of global strong solution for the density-dependent system is established by Li-Wang in \cite{LXWD}. For the compressible model, there are less results due to its complexity. In \cite{WD}, Wang-Xu-Yu established the existence as well as long time dynamics of global weak solutions. In fact, there are more results on the hydrodynamic system for  the three-dimensional flow of nematic liquid crystals. For example, Jiang-Jiang-Wang \cite{JFJS} has proved the existence of a global weak solution to a two-dimensional simplified Ericksen-Leslie system of compressible flow of nematic liquid crystals, and the existence of a weak solution in a bounded domain for both 2D and 3D can be seen in \cite{JFJSW} and \cite{WDYC}. For more researches related to the topic, check \cite{chen, CG} and the reference within.

If we just consider the first two equations in system \eqref{qn}, and make the stochastic forcing term $\mathbb{G}\equiv 0$ in equation \eqref{qn}(2), then the equations would degenerate to the system of compressible Navier-Stokes equations. There have been tremendous studies about the existence of the solution in both deterministic and stochastic cases. For the deterministic case, the pioneering work has been done by Lions in \cite{LP}, the existence of a global weak solution has been proven for adiabatic constant $\gamma>\frac{9}{5}$ by introducing the re-normalized solution to deal with the difficulty in large oscillations. Then, \cite{FENP} extended the result to adiabatic exponent $\gamma>\frac{3}{2}$, which by now is the result that allows the maximum range of $\gamma$. For more results, we  refer the reader to \cite{HD,HXLJ,JS,MAYS,MAVA} and the reference within. For the stochastic case, the existence result of global weak martingale solution was built in \cite {Hofmanova,16, DWang}. In addition, see \cite{17} for the construction of weak martingale solution to non-isentropic, compressible case. Moreover, Breit-Feireisl-Hofmanov\'{a} \cite{breit} proved the existence and uniqueness of local strong pathwise solution to compressible Navier-Stokes equations.

For the stochastic liquid crystal hydrodynamics system, we refer the readers to \cite{brze1,brze2,brze3,wu} for the well-posedness result of incompressible case,
and  Qiu-Wang \cite{QW} obtained the global existence of weak martingale solution to the compressible active liquid crystal system, we remark that the result of this paper could be extended to the active system.

In this paper, we are going to prove the existence and uniqueness of strong pathwise solution to the stochastic system (\ref{qn}), where the ``strong" means the strong existence in both PDE and probability sense. That is, the solution has sufficient space regularity and satisfies the system in the pointwise sense when the probability space is given. Here we mention that even if in the deterministic case, there is no related result for the existence and uniqueness result of strong solution for which the state space lies in $(H^s)^2\times H^{s+1}$ for integer $s>\frac{9}{2}$. We would introduce the symmetric system considering the energy estimate of the strong solution to compressible fluid. Therefore, for the convenience of the symmetrization, we require the density $\rho>0$, which means the vacuum state shall not appear.

To prove the existence of the strong solution, we need to build some compactness result for the approximate solution. In the stochastic case, although we know that  the embedding $X\hookrightarrow Y$ is compact, it is still hard to tell if the embedding $L^p(\Omega; X)\hookrightarrow L^p(\Omega; Y)$ is compact or not. Therefore, we can no longer apply the classical criterion like Aubin-Lions lemma or Arzela-Ascoli theorem directly as in the deterministic case. Invoked by the Yamada-Watanabe argument, first we apply the classical Skorokhod representation theorem to establish a strong martingale solution, then by proving the pathwise uniqueness, we could reveal that the solution is also strong in probability sense.

During the high-order energy estimate of the approximate solution, we would apply Moser-type estimate and get the form  $(\|\mathbf{u}\|_{2,\infty}+\|Q\|_{3,\infty})\cdot (\|\mathbf{u}\|_{s,2}^2+\|Q\|_{s+1,2}^2)$ and $(\|\mathbf{u}\|_{2,\infty}+\|\rho\|_{1,\infty})\cdot \|\rho, \mathbf{u}\|_{s,2}^2$, making it difficult to get the estimate. Inspired by \cite{Kim1}, we could deal with the nonlinear terms by adding a cut-off function. We could get that $\|\rho\|_{1,\infty}$ would be bounded if $\|\rho_0\|_{1,\infty}, \|\mathbf{u}\|_{1,\infty} $ are bounded, then the cut-off function only depends on $\|\mathbf{u}\|_{2,\infty}$ and $\|Q\|_{3,\infty}$ under the assumption that $\|\rho_0\|_{1,\infty}$ is bounded. The benefit is, while building Galerkin approximation system, for every fixed $\mathbf{u}$ we could first solve the mass equation directly which actually is a linear transport equation and solve the ``parabolic-type" $Q$-tensor equation. In turn, we obtain the existence of approximate solution $\mathbf{u}$ in a finite dimensional space. Different from the deterministic case, we will develop a new extra layer approximation to deal with the difficulty arising from the stochastic integral, constructing the Galerkin approximate solution with the spirit of \cite{Hofmanova}. Also, the cut-off function brings downside in proving the uniqueness. We have to restrict our regularity index to integer $s>\frac{9}{2}$ comparing with the martingale solution result which only requires $s>\frac{7}{2}$.

 Note that during the uniform energy estimate and the uniqueness argument, the most challenging term to deal with is the one with the highest order, $\Dv_x(Q\triangle Q-\triangle Q Q)$ in the momentum equation, it is hard to control it directly. Luckily, we are able to cancel this term using $\Theta Q-Q\Theta$ after integration by parts as well as some transformation, the detailed operation can be seen in Lemma \ref{lem2.4} where an artificial scalar function $f(r)$ is added for matching the momentum equation. However, the scalar function $f(r)$ brings extra difficulties in the a priori estimate which needs to be handled technically.

The rest of paper is organized as follows. Section 2 would offer the deterministic and stochastic preliminaries associated with system (\ref{qn}) and the main result. We will transfer system (\ref{qn}) into a symmetric system in Section 3. In Section 4 and Section 5 we will establish the existence of global strong martingale solution and strong pathwise solution to the symmetric system. Finally, in Section 6, we prove the main theorem by applying a cutting-off argument, so that the initial data could be more generalized. Last, we include an Appendix stating the results used frequently in this paper.

\section{{\bf Preliminary and Main Result}}

First, we present some deterministic as well as stochastic preliminaries associated with system (\ref{qn}). For each integer $s\in \mathbb{N}^{+}$, denote $W^{s,2}(\mathbb{T})$ as the Sobolev space containing all the functions having distributional derivatives up to order $s$, and the derivatives are integrable in $L^{2}(\mathbb{T})$, endowed with the norm
\begin{eqnarray*}
\|u\|_{W^{s,2}}^{2}=\sum_{k\in \mathbb{Z}^3}(1+k^{2})^{s}|\hat{u}_k|^{2},
\end{eqnarray*}
where $\hat{u}_k$ is the Fourier coefficients of $u$. $W^{s,2}(\mathbb{T})$ is an Hilbert space, and for any $u$, $v\in W^{s,2}$, the inner product can be denoted as
$$(u,v)_{s,2}=\sum_{|\alpha|\leq s}\int_{\mathbb{T}}\partial_x^{\alpha}u\cdot\partial_x^{\alpha}vdx.$$
For simplicity, we denote the notations $\|\cdot\|$ as the $L^2$-norm, $\|\cdot\|_{\infty}$ as the $L^\infty$-norm, and$\|\cdot\|_{s,p}$ as the $W^{s,p}$-norm for all $1\leq s<\infty, 1\leq p\leq \infty$.

Define the inner product between two $3\times 3$ matrices $\mathrm{M}_1$ and $\mathrm{M}_2$
\begin{eqnarray*}
(\mathrm{M}_1, \mathrm{M}_2)=\int_{\mathbb{T}}\mathrm{M}_1:\mathrm{M}_2 dx=\int_{\mathbb{T}}{\rm tr}(\mathrm{M}_1\mathrm{M}_2)dx,
\end{eqnarray*}
and $S_0^3\subset \mathbb{M}^{3\times 3}$ the space of $Q$-tensor
\begin{eqnarray*}
S_0^3=\left\{Q\in \mathbb{M}^{3\times 3}:~Q_{ij}=Q_{ji},~ {\rm tr}(Q)=0, ~i,j=1,2,3\right\},
\end{eqnarray*}
and the norm of a matrix using the Frobenius norm
\begin{eqnarray*}
|Q|^2={\rm tr}(Q^2)=\sum\limits_{i,j=1}^{3}Q_{ij}Q_{ij}.
\end{eqnarray*}
Set $|\partial_x^\alpha Q|^2=\sum\limits_{i,j=1}^{3}\partial_{x}^\alpha Q_{ij}\partial_{x}^\alpha Q_{ij}$. The Sobolev space of $Q$-tensor is defined by
\begin{eqnarray*}
W^{s,2}(\mathbb{T};S_0^3)=\left\{Q: \mathbb{T}\rightarrow S_0^3, ~{\rm and}~ \sum_{|\alpha|\leq s}\|\partial_x^\alpha Q\|^2<\infty\right\},
\end{eqnarray*}
endowed with the norm
$$\|Q\|_{W^{s,2}(\mathbb{T};S_0^3)}^2:=\|Q\|_{s,2}^2=\sum_{|\alpha|\leq s}\|\partial_x^\alpha Q\|^2.$$

To deal with the estimate of the nonlinear terms in the equations, we present the following lemmas that involves commutator and Moser estimates. The proof of these lemmas can be found in \cite{Kato2,MajdaA}.
\begin{lemma}\label{lem2.1}
For $u,v\in W^{s,2}(\mathbb{T})$, $s>\frac{d}{2}+1$, $d=2,3$ is the dimension of space, it holds
\begin{equation}\label{2.2}
\sum\limits_{0\leq|\alpha|\leq s}\|\partial_x^{\alpha}(u\cdot \nabla_x) v-u\cdot \nabla_x \partial_x^{\alpha}v\|\leq C(\|\nabla_x u\|_{\infty}\|v\|_{s,2}+\|\nabla_x v\|_{\infty}\|u\|_{s,2}),
\end{equation}
and
 \begin{equation}\label{2.3}
\|uv\|_{s,2}\leq C(\|u\|_{\infty}\|v\|_{s,2}+\|v\|_{\infty}\|u\|_{s,2}),
\end{equation}
for some positive constant $C=C(s,\mathbb{T})$ independent of $u$ and $v$.
\end{lemma}

\begin{lemma}\label{lem2.2} Let $f$ be a $s$-order continuously differentiable function on the neighborhood of compact set $G={\rm range}[u]$ and $u\in W^{s,2}(\mathbb{T})\cap C(\mathbb{T})$, it holds
\begin{eqnarray*}
\|\partial_x^\alpha f(u)\|\leq C\|\partial_u f\|_{C^{s-1}(G)}\|u\|_{\infty}^{|\alpha|-1}\|\partial_x^\alpha u\|,
\end{eqnarray*}
for all $\alpha\in \mathbb{N}^N, 1<|\alpha|\leq s$.
\end{lemma}

The following result is crucial to handle the highest-order derivative terms in the momentum and $Q$-tensor equations.
\begin{lemma}\label{lem2.4}
  Assume that $Q$ and $Q'$ are two $3\times 3$ symmetric matrices, and $\Theta=\frac{1}{2}(\nabla_x\mathbf{u}-\nabla_x\mathbf{u}^{{\rm T}})$, as $\nabla_x\mathbf{u}$ is also a $3\times 3$ matrix, and $(\nabla_x\mathbf{u})_{ij}=\partial_i u_j$, $f(r)$ is a scalar function. Then
  $$(f(r)(\Theta Q'-Q'\Theta),\triangle Q)+(f(r)(Q'\triangle Q-\triangle Q Q'),\nabla_x\mathbf{u}^{{\rm T}})=0.$$
\end{lemma}
\begin{proof} In a similar way to \cite[Lemma A.1]{CG}, using the fact that ${\rm tr}(\mathrm{M}_1\mathrm{M}_2)={\rm tr}(\mathrm{M}_1\mathrm{M}_2)$ and $Q',Q, \Theta+\nabla_x\mathbf{u}^{{\rm T}}$ are symmetric, $f(r)$ is scalar function, we get
\begin{eqnarray*}
&&\quad(f(r)(\Theta Q'-Q'\Theta),\triangle Q)+(f(r)(Q'\triangle Q-\triangle Q Q'),\nabla_x\mathbf{u}^{{\rm T}})\\
&&=(f(r)(Q'\triangle Q-\triangle Q Q'),\Theta)+(f(r)(Q'\triangle Q-\triangle Q Q'),\nabla_x\mathbf{u}^{{\rm T}})\\
&&=(f(r)(Q'\triangle Q-\triangle Q Q'),\Theta+\nabla_x\mathbf{u}^{{\rm T}})=0,
\end{eqnarray*}
we finish the proof.
\end{proof}

Next, we introduce the following fractional-order Sobolev space with respect to time $t$, since noise term is only H\"{o}lder's continuous of order strictly less than $\frac{1}{2}$ in time.

For any fixed $p>1$ and $\alpha\in(0,1)$ we define
\begin{equation*}
W^{\alpha,p}(0,T;X)=\left\{v\in L^{p}(0,T;X):\int_{0}^{T}\int_{0}^{T}\frac{\|v(t_{1})-v(t_{2})\|_{X}^{p}}{|t_{1}-t_{2}|^{1+\alpha p}}dt_{1}dt_{2}<\infty\right\},
\end{equation*}
endowed with the norm
\begin{equation*}
\|v\|^p_{W^{\alpha,p}(0,T;X)}:=\int_{0}^{T}\|v(t)\|_{X}^{p}dt+\int_{0}^{T}\int_{0}^{T}\frac{\|v(t_{1})-v(t_{2})\|_{X}^{p}}{|t_{1}-t_{2}|^{1+\alpha p}}dt_{1}dt_{2},
\end{equation*}
for any separable Hilbert space $X$. If we take $\alpha=1$, then
\begin{equation*}
W^{1,p}(0,T;X):=\left\{v\in L^{p}(0,T;X):\frac{dv}{dt}\in L^{p}(0,T;X)\right\},
\end{equation*}
we could see that the space returns to the classical Sobolev space endowed with the usual norm
\begin{equation*}
\|v\|_{W^{1,p}(0,T;X)}^{p}:=\int_{0}^{T}\|v(t)\|_{X}^{p}+\left\|\frac{dv}{dt}(t)\right\|_{X}^{p}dt.
\end{equation*}
Note that for $\alpha\in(0,1)$, $ W^{1,p}(0,T;X)$ is a subspace of $ W^{\alpha,p}(0,T;X)$.

For any $\alpha \leq \beta-\frac{1}{p}$, it holds
\begin{eqnarray}\label{2.31}
W^{\beta,p}(0,T; L^2(\mathbb{T}))\hookrightarrow C^\alpha ([0,T]; L^2(\mathbb{T})).
\end{eqnarray}

 Let $\mathcal{S}:=(\Omega,\mathcal{F},\{\mathcal{F}_{t}\}_{t\geq0},\mathbb{P}, W)$ be a fixed stochastic basis and $(\Omega,\mathcal{F},\mathbb{P})$ be a complete probability space. Let $W$ be a Wiener process defined on an Hilbert space $\mathfrak{U}$, which is adapted to the complete, right continuous filtration $\{\mathcal{F}_{t}\}_{t\geq 0}$. If $\{e_{k}\}_{k\geq 1}$ is a complete orthonormal basis of $\mathfrak{U}$, then $W$ can be written formally as the expansion $W(t,\omega)=\sum_{k\geq 1}e_{k}\beta_{k}(t,\omega)$ where $\{\beta_{k}\}_{k\geq 1}$ is a sequence of independent standard one-dimensional Brownian motions.

Define an auxiliary space $\mathfrak{U}_0\supset \mathfrak{U}$ by
\begin{eqnarray*}
\mathfrak{U}_0=\left\{v=\sum_{k\geq 1}\alpha_ke_k:\sum_{k\geq 1}\frac{\alpha_k^2}{k^2}<\infty\right\},
\end{eqnarray*}
with the norm $\|v\|^2_{\mathfrak{U}_0}=\sum_{k\geq 1}\frac{\alpha_k^2}{k^2}$. Note that the embedding of $ \mathfrak{U}\hookrightarrow \mathfrak{U}_0$ is Hilbert-Schmidt. We also have that $W\in C([0,\infty), \mathfrak{U}_0)$ almost surely, see \cite{Prato}.

Now considering another separable Hilbert space $X$ and let $L_{2}(\mathfrak{U},X)$ be the set of all Hilbert-Schmidt operators $S:\mathfrak{U}\rightarrow X$ with the norm
$\|S\|_{L_{2}(\mathfrak{U},X)}^2=\sum_{k\geq 1}\|Se_k\|_{X}^2$. For a predictable process $G\in L^{2}(\Omega;L^{2}_{loc}([0,\infty),L_{2}(\mathfrak{U},X)))$  by taking $G_{k}=Ge_{k}$, one can define the stochastic integral
\begin{equation*}
\mathcal{M}_{t}:=\int_{0}^{t}GdW=\sum_{k}\int_{0}^{t}
Ge_{k}d\beta_{k}=\sum_{k}\int_{0}^{t}G_{k}d\beta_{k},
\end{equation*}
which is an $X$-valued square integrable martingale, and the Burkholder-Davis-Gundy inequality holds
\begin{equation}\label{2.4}
\mathbb{E}\left(\sup_{0\leq t\leq T}\left\|\int_{0}^{t}GdW\right\|_{X}^{p}\right)\leq c_{p}\mathbb{E}\left(\int_{0}^{T}\|G\|_{L_{2}(\mathfrak{U},X)}^{2}dt\right)^{\frac{p}{2}},
\end{equation}
for any $1\leq p<\infty$, for more details see \cite{Prato}. The notation $\mathbb{E}$ represents the expectation.

We shall present the main result of this paper. First, we define local strong pathwise solution. For this type of solution, "strong" means in PDE and probability sense, "local" means existence in finite time.
\begin{definition}\label{de1}
    (Local strong pathwise solution). Let $(\Omega,\mathcal{F},\{\mathcal{F}_t\}_{t\geq 0},\mathbb{P})$ be a fixed probability space, W be an $\mathcal{F}_t$-cylindrical Wiener process. Then $(\rho,\mathbf{u},Q,\mathfrak{t})$ is a local strong pathwise solution to system \eqref{qn} if the following conditions hold
    \begin{enumerate}
    \item $\mathfrak{t}$ is a strictly positive a.s. $\mathcal{F}_t$-stopping time;

    \item $\rho$, $\mathbf{u}$, $Q$ are $\mathcal{F}_t$-progressively measurable processes, satisfying ~~$\mathbb{P}$ \mbox{a.s.}
    \begin{align*}
    &\qquad\rho(\cdot\wedge\mathfrak{t})>0,~\rho(\cdot\wedge\mathfrak{t})\in C([0,T];W^{s,2}(\mathbb{T})),\\
    &\qquad\mathbf{u}(\cdot\wedge\mathfrak{t})\in L^\infty(0,T;W^{s,2}(\mathbb{T}, \mathbb{R}^3))\cap L^2(0,T;W^{s+1,2}(\mathbb{T},\mathbb{R}^3))\cap C([0,T];W^{s-1,2}(\mathbb{T}, \mathbb{R}^3)),\\
    &\qquad Q(\cdot\wedge\mathfrak{t})\in L^\infty(0,T;W^{s+1,2}(\mathbb{T},S_0^3))\cap L^2(0,T;W^{s+2,2}(\mathbb{T},S_0^3))\cap C([0,T];W^{s,2}(\mathbb{T},S_0^3));
    \end{align*}

    \item for any $t\in[0,T]$, $\mathbb{P}$ a.s.
    \begin{align*}
    &\rho(\mathfrak{t}\wedge t)=\rho_0-\int_{0}^{\mathfrak{t}\wedge t}\Dv_x(\rho \mathbf{u})d\xi,\\
     &(\rho \mathbf{u})(\mathfrak{t}\wedge t)=\rho_0\mathbf{u}_0-\int_{0}^{\mathfrak{t}\wedge t}\Dv_x(\rho \mathbf{u}\otimes \mathbf{u}) d\xi-\int_{0}^{\mathfrak{t}\wedge t}\nabla_x(A\rho^\gamma)d\xi+\int_{0}^{\mathfrak{t}\wedge t}\mathcal{L}\mathbf{u}d\xi\\
     &\qquad\qquad\quad-\int_{0}^{\mathfrak{t}\wedge t}\Dv_x(L\nabla_x Q\odot \nabla_x Q-\mathcal{F}(Q){\rm I}_3)d\xi\\
        &\qquad\qquad\quad+\int_{0}^{\mathfrak{t}\wedge t}L\Dv_x(Q\triangle Q-\triangle QQ)d\xi
        +\int_{0}^{\mathfrak{t}\wedge t}\mathbb{G}(\rho,\rho \mathbf{u})dW, \\
        &Q(\mathfrak{t}\wedge t)=Q_0-\int_{0}^{\mathfrak{t}\wedge t}\mathbf{u}\cdot\nabla_x Qd\xi+\int_{0}^{\mathfrak{t}\wedge t}(\Theta Q-Q\Theta) d\xi+\int_{0}^{\mathfrak{t}\wedge t}\Gamma \mathcal{H}(Q)d\xi.
    \end{align*}
    \end{enumerate}

    We say that the pathwise uniqueness holds: if $(\rho_1, \mathbf{u}_{1}, Q_{1},\mathfrak{t}_{1})$ and $(\rho_2, \mathbf{u}_{2}, Q_{2}, \mathfrak{t}_{2})$ are two local strong pathwise solutions of system (\ref{qn}) with
 \begin{eqnarray*}
\mathbb{P}\{(\rho_1(0), \mathbf{u}_{1}(0), Q_{1}(0))=(\rho_2(0), \mathbf{u}_{2}(0), Q_{2}(0))\}=1,
 \end{eqnarray*}
then
\begin{eqnarray*}
\mathbb{P}\left\{(\rho_1(t,x), \mathbf{u}_{1}(t,x), Q_{1}(t,x))=(\rho_2(t,x), \mathbf{u}_{2}(t,x), Q_{2}(t,x));\forall t\in[0,\mathfrak{t}_{1}\wedge\mathfrak{t}_{2}]\right\}=1.
\end{eqnarray*}
\end{definition}

\begin{definition}\label{de2}
    \rm{(Maximal strong pathwise solution)} A maximal pathwise solution is a quintuple $(\rho, \mathbf{u},Q,\{\tau_{n}\}_{n\geq1}, \mathfrak{t})$ such that each $(\rho, \mathbf{u},Q,\tau_{n})$ is a local pathwise solution in the sense of Definition \ref{de1} and $\{\tau_{n}\}$ is an increasing sequence with
$\lim_{n\rightarrow\infty}\tau_{n}=\mathfrak{t}$ and
\begin{equation*}
\sup\limits_{t\in[0,\tau_{n}]}\|\mathbf{u}(t)\|_{2,\infty}\geq n,~ \sup\limits_{t\in[0,\tau_{n}]}\|Q(t)\|_{3,\infty}\geq n,~~\mbox{on the set} ~~ \{\mathfrak{t}<\infty\}.
\end{equation*}
\end{definition}
From the Definition \ref{de2}, we can see that
$$\sup_{t\in [0,\mathfrak{t})}\|\mathbf{u}(t)\|_{2,\infty}=\infty,~\sup_{t\in [0,\mathfrak{t})}\|Q(t)\|_{3,\infty}=\infty, \mbox{ on the set }\{\mathfrak{t}<\infty\}.$$
This means the existence time for the solution is determined by the explosion time of the $W^{2,\infty}$-norm of the velocity and $W^{3,\infty}$-norm of the $Q$-tensor.

Throughout the paper, we impose the following assumptions on the noise intensity $\mathbb{G}$: there exists a constant $C$ such that for any $s\geq 0, \rho>0$,
\begin{eqnarray}\label{2.5}
\|\rho^{-1}\mathbb{G}(\rho, \mathbf{u})\|_{L_2(\mathfrak{U}; W^{s,2}(\mathbb{T}))}^2\leq C(\|\rho\|^2_{1,\infty}+\|\mathbf{u}\|^2_{2,\infty})\|\rho, \mathbf{u}\|_{s,2}^2,
\end{eqnarray}
and
\begin{eqnarray}\label{2.5*}
&&\quad\|\rho_1^{-1}\mathbb{G}(\rho_1, \mathbf{u}_1)-\rho_2^{-1}\mathbb{G}(\rho_2, \mathbf{u}_2)\|_{L_2(\mathfrak{U}; W^{s,2}(\mathbb{T}))}^2\nonumber \\
&&\leq C(\|\rho_1, \rho_2\|^2_{1,\infty}+\|\mathbf{u}_1, \mathbf{u}_2\|^2_{2,\infty})\|\rho_1-\rho_2, \mathbf{u}_1-\mathbf{u}_2\|_{s,2}^2,
\end{eqnarray}
where the norm $\|u, v\|_{s,2}^2:=\|u\|_{s,2}^2+\|v\|_{s,2}^2$ for $u,v\in W^{s,2}$. Assumption \eqref{2.5} will be used for constructing the a priori estimate, while assumption \eqref{2.5*} will be applied to identify the limit and establish the uniqueness.
\begin{remark}\label{rem2.6} Set $r=\sqrt{\frac{2A\gamma}{\gamma-1}}\rho^{\frac{\gamma-1}{2}}$. If the initial data $r_0$ satisfies some certain assumption, see Theorem \ref{th3.1}, then the assumptions \eqref{2.5},\eqref{2.5*} still hold if we replace $\rho$ by $r$ and $\rho^{-1}\mathbb{G}(\rho, \mathbf{u})$ by $\mathbb{F}(r,\mathbf{u})=\frac{1}{\rho(r)}\mathbb{G}(\rho(r),\rho(r) \mathbf{u})$.
\end{remark}

Our main result of this paper is below.
\begin{theorem}\label{thm2.5}
    Assume $s\in\mathbb{N}$ satisfies $s>\frac{9}{2}$, and the coefficient $\mathbb{G}$ satisfies the assumptions \eqref{2.5},\eqref{2.5*}, and the initial data $(\rho_0,\mathbf{u}_0,Q_0)$ is $\mathcal{F}_0$-measurable random variable, with values in $W^{s,2}(\mathbb{T})\times W^{s,2}(\mathbb{T};\mathbb{R}^3)\times W^{s+1,2}(\mathbb{T};S_0^3)$, also $\rho_0>0$, $\mathbb{P}$ a.s.. Then there exists a unique maximal strong pathwise solution $(\rho,\mathbf{u},Q,\mathfrak{t})$ to system \eqref{qn}-\eqref{1.4} in the sense of Definition \ref{de2}.
\end{theorem}

\section{\bf Construction of Truncated Symmetric System}
Before the construction of the strong solution, we need to assume first that the vacuum state does not appear. By doing so, we are able to rewrite the system \eqref{qn} into the symmetric system following the operation in \cite{breit}. To begin with, applying equation \eqref{qn}(1), then equation \eqref{qn}(2) can be written into the following form
\begin{align*}
&\rho \partial_t\mathbf{u}+\rho \mathbf{u}\cdot \nabla_x \mathbf{u}+A\nabla_x \rho^\gamma\nonumber\\
=&\mathcal{L}\mathbf{u}-\Dv_x(L\nabla_x Q\odot \nabla_x Q-\mathcal{F}(Q){\rm I}_3)
        +L\Dv_x(Q\triangle Q-\triangle QQ)+\mathbb{G}(\rho,\rho \mathbf{u})\frac{dW}{dt},
\end{align*}
as $\rho>0$, divide the above equation by $\rho$ on both sides, we could have
\begin{align}
&\partial_t\mathbf{u}+\mathbf{u}\cdot \nabla_x \mathbf{u}+\frac{A}{\rho}\nabla_x \rho^\gamma\nonumber\\
=&\frac{1}{\rho}\mathcal{L}\mathbf{u}-\frac{1}{\rho}\Dv_x(L\nabla_x Q\odot \nabla_x Q-\mathcal{F}(Q){\rm I}_3)
        +L\frac{1}{\rho}\Dv_x(Q\triangle Q-\triangle QQ)\nonumber\\&+\frac{1}{\rho}\mathbb{G}(\rho,\rho \mathbf{u})\frac{dW}{dt}.
\end{align}
The pressure term can be written into a symmetric form:
\begin{align*}
\frac{A}{\rho}\nabla_x \rho^\gamma=\frac{A}{\gamma-1}\nabla_x\rho^{\gamma-1}=\frac{2A\gamma}{\gamma-1}\rho^\frac{\gamma-1}{2}\nabla_x \rho^\frac{\gamma-1}{2}.
\end{align*}
Considering this, define
\begin{align*}
 r=\sqrt{\frac{2A\gamma}{\gamma-1}}\rho^{\frac{\gamma-1}{2}},
\end{align*}
and
\begin{align*}
D(r)=\frac{1}{\rho(r)}=\left(\frac{\gamma-1}{2A\gamma}\right)^{-\frac{1}{\gamma-1}}r^{-\frac{2}{\gamma-1}}, ~\mathbb{F}(r,\mathbf{u})=\frac{1}{\rho(r)}\mathbb{G}(\rho(r),\rho(r) \mathbf{u}).
\end{align*}
Then, the system \eqref{qn} can be transformed into
\begin{equation}\label{qnt1}
    \begin{cases} dr+\left(\mathbf{u}\cdot\nabla_xr+\frac{\gamma-1}{2}r\Dv_x\mathbf{u}\right)dt= 0,\\
        d\mathbf{u}+(\mathbf{u}\cdot\nabla_x \mathbf{u}+r\nabla_x r)dt\\
    \qquad= D(r)(\mathcal{L}\mathbf{u}-\Dv_x(L\nabla_x Q\odot \nabla_x Q
    -\mathcal{F}(Q){\rm I}_3)+L\Dv_x(Q\triangle Q-\triangle QQ))dt\\
    \quad\qquad+\mathbb{F}(r,\mathbf{u})dW,\\
        dQ+(\mathbf{u}\cdot\nabla_x Q-\Theta Q+Q\Theta) dt
        =\Gamma \mathcal{H}(Q)dt.
    \end{cases}
    \end{equation}

As mentioned in the introduction, we add a cut-off function to render the nonlinear terms, where the cut-off function depends only on $\|\mathbf{u}\|_{2, \infty}, \|Q\|_{3,\infty}$.

Let $\Phi_{R}:[0,\infty)\rightarrow [0,1]$ be a $C^{\infty}$-smooth function defined as follows
\begin{eqnarray*}
 \Phi_{R}(x) =\left\{\begin{array}{ll}
                  1,& \mbox{if} \ 0<x<R,  \\
                  0,& \mbox{if} \ x>2R.  \\
                \end{array}\right.
\end{eqnarray*}
Then, define $\Truca=\Phi_{R}^\mathbf{u}\cdot\Phi_{R}^Q$, where $\Phi_{R}^\mathbf{u}=\Phi_{R}(\|\mathbf{u}\|_{2,\infty}), \Phi_{R}^Q=\Phi_{R}(\|Q\|_{3,\infty})$ and add the cut-off function in front of nonlinear terms of system (\ref{qnt1}), we have
\begin{equation}\label{qnt}
    \begin{cases} dr + \Truca \left(\mathbf{u}\cdot\nabla_xr+\frac{\gamma-1}{2}r\Dv_x\mathbf{u}\right)dt= 0,\\
        d\mathbf{u}+\Truca (\mathbf{u}\cdot\nabla_x \mathbf{u}+r\nabla_x r)dt\\
    \quad=\Truca D(r)(\mathcal{L}\mathbf{u}-\Dv_x(L\nabla_x Q\odot \nabla_x Q
    -\mathcal{F}(Q){\rm I}_3)+L\Dv_x(Q\triangle Q-\triangle QQ))dt\\
    \quad\quad+\Truca \mathbb{F}(r,\mathbf{u})dW,\\
        dQ+\Truca (\mathbf{u}\cdot\nabla_x Q-\Theta Q+Q\Theta) dt
        =\Gamma L\triangle Qdt+\Truca \mathcal{K}(Q)dt.
    \end{cases}
    \end{equation}
\begin{remark} In system (\ref{qnt}), we use the same cut-off function $\Truca$ in front of the all nonlinear terms to simplify the notation. Actually, we can replace $\Truca$ by $\Phi_R^\mathbf{u}$ on the left hand side of equations (\ref{qnt})(1)(2) and in front of the stochastic term, replace $\Truca$ by $\Phi_R^Q$ on the right hand side of equation (\ref{qnt})(3).
\end{remark}
In the following, we mainly discuss the truncated system (\ref{qnt}).

\section{\bf  Existence of Strong Martingale Solution}
In this section, the main aim is that, proving the existence of a strong martingale solution to system (\ref{qnt}) which is strong in PDE sense and weak in probability sense if the initial condition is good enough. To start, we bring in the concept of strong martingale solution.
\begin{definition}\label{def3.1}
(Strong martingale solution) Assume that $\Lambda$ is a Borel probability measure on the space
$W^{s,2}(\mathbb{T})\times W^{s,2}(\mathbb{T},\mathbb{R}^3)\times W^{s+1,2}(\mathbb{T},S_0^3)$ for integer $s>\frac{7}{2}$, then the quintuple $$((\Omega,\mathcal{F},\{\mathcal{F}_t\}_{t\geq 0},\mathbb{P}),r,\mathbf{u},Q,W)$$ is a strong martingale solution to the truncated system \eqref{qnt} equipped with the initial law $\Lambda$ if the following conditions hold
\begin{enumerate}
  \item $(\Omega,\mathcal{F},\{\mathcal{F}_t\}_{t\geq 0},\mathbb{P})$ is a stochastic basis with a complete right-continuous filtration, $W$ is a Wiener process relative to the filtration $\mathcal{F}_t$;

  \item $r$, $\mathbf{u}, Q$ are $\mathcal{F}_t$-progressively measurable processes with values in $W^{s,2}(\mathbb{T}), W^{s,2}(\mathbb{T},\mathbb{R}^3)$, $W^{s+1,2}(\mathbb{T},S_0^3)$, satisfying
      \begin{align*}
        &\qquad r\in L^2(\Omega;C([0,T];W^{s,2}(\mathbb{T}))),~r(t)>0,~ \mathbb{P}\mbox{ a.s.}, {\rm for ~ all} ~t\in [0,T],\\
        &\qquad \mathbf{u}\in L^2(\Omega;L^\infty(0,T;W^{s,2}(\mathbb{T};\mathbb{R}^3))\cap C([0,T];W^{s-1,2}(\mathbb{T}, \mathbb{R}^3))),\\
        &\qquad Q\in L^2(\Omega;L^\infty(0,T;W^{s+1,2}(\mathbb{T};S_0^3))
        \cap L^2(0,T;W^{s+2,2}(\mathbb{T};S_0^3))\cap C([0,T];W^{s,2}(\mathbb{T};S_0^3)));
      \end{align*}

  \item the initial law $\Lambda=\mathbb{P}\circ (r_0, \mathbf{u}_0, Q_0)^{-1}$;

  \item for all $t\in[0,T]$, $\mathbb{P}$ a.s.
    \begin{align*}
     \quad\quad\quad& r(t)=r(0)-\int_{0}^{t} \Truca \left(\mathbf{u}\cdot\nabla_xr+\frac{\gamma-1}{2}r\Dv_x\mathbf{u}\right)d\xi,\\
      &\mathbf{u}(t)=\mathbf{u}(0)-\int_{0}^{t}\Truca (\mathbf{u}\cdot\nabla_x \mathbf{u}+r\nabla_x r)d\xi\\
    &\qquad\quad+\int_{0}^{t}\Truca D(r)(\mathcal{L}\mathbf{u}-\Dv_x(L\nabla_x Q\odot \nabla_x Q
    -\mathcal{F}(Q){\rm I}_3)\\
    &\qquad\quad+L\Dv_x(Q\triangle Q-\triangle QQ))d\xi+\int_{0}^{t}\Truca \mathbb{F}(r,\mathbf{u})dW,\\
    &Q(t)=Q(0)-\int_{0}^{t}\Truca (\mathbf{u}\cdot\nabla_x Q-\Theta Q+Q\Theta) d\xi
        +\int_{0}^{t} \Gamma L\triangle Q+\Truca\mathcal{K}(Q)d\xi.
    \end{align*}
\end{enumerate}
\end{definition}

 We state our main result for this section.
\begin{theorem}\label{th3.1}
Assume the initial data $(r_0,\mathbf{u}_0,Q_0)$ satisfies
$$(r_0,\mathbf{u}_0,Q_0)\in L^p(\Omega;W^{s,2}(\mathbb{T})\times W^{s,2}(\mathbb{T}, \mathbb{R}^3)\times W^{s+1,2}(\mathbb{T},S_0^3)),$$
for any $1\leq p<\infty$, $s>\frac{7}{2}$ be the integer, and in addition
$$\|Q_0\|_{1,2}<R,~ \|r_0\|_{1,\infty}<R,~r_0>\frac{1}{R},~~\mathbb{P}~~ \mbox{a.s}.$$
for constant $R>0$, the coefficient $\mathbb{G}$ satisfies assumptions \eqref{2.5},\eqref{2.5*}, then there exists a strong martingale solution to the system \eqref{qnt} with the initial law $\Lambda=\mathbb{P}\circ (r_0, \mathbf{u}_0, Q_0)^{-1}$ in the sense of Definition \ref{def3.1} and we also have
$$r(t, \cdot)\geq \mathcal{C}(R)>0,~ \mathbb{P}~ \mbox{a.s.}, ~{\rm for ~ all} ~t\in [0,T],$$
where $\mathcal{C}(R)$ is a constant depending on $R$, and
      \begin{align*}
        &\mathbb{E}\left[\sup_{t\in[0,T]}(\|r(t),\mathbf{u}(t)\|_{s,2}^2+\|Q(t)\|_{s+1,2}^2)+\int_{0}^{T}
        \Truca\|\mathbf{u}\|_{s+1,2}^2+\|Q\|_{s+2,2}^2dt\right]^p
         \leq C,
      \end{align*}
for any $T>0$, where $C=C(p,s, R, \mathbb{T},T, L, \Gamma)$ is a constant.
\end{theorem}
\begin{remark}
Here, we assume that $\|Q_0\|_{1,2}<R$, $\mathbb{P}$~ a.s. for establishing the Galerkin approximate solution, which could also be relaxed to general case, see Section 6.
\end{remark}

The following part is devoted to proving Theorem \ref{th3.1} which is divided into three steps. First, we construct the approximate solution in the finite-dimensional space.  Then we get the uniform estimate of the approximate solution, and show the stochastic compactness. Next, the existence of the strong martingale solution can be derived from taking the limit of the approximate system.

{\bf 4.1 Galerkin approximate system}. In this subsection, we construct the Galerkin approximate solution of system (\ref{qnt}). First, for any smooth functions $\mathbf{u}, Q$, the transport equation \eqref{qnt}(1) would admit a classical solution $r=r[\mathbf{u}]$, and the solution is unique if the initial data $r_0$ is given. The solution $r[\mathbf{u}]$ shares the same regularity with the initial data $r_0$. In addition, for certain constant c, we have, see also \cite[Section 3.1]{breit}
    \begin{equation}\label{2.1}
        \begin{aligned}
          &\frac{1}{R}\exp(-cRt)\leq\exp(-cRt)\inf_{x\in \mathbb{T}}r_0\leq r(t,\cdot)
          \leq \exp(cRt)\sup_{x\in \mathbb{T}}r_0\leq R\exp(cRt),\\
          &|\nabla_xr(t,\cdot)|\leq \exp(cRt)|\nabla_xr_0|\leq R\exp(cRt),~~\mbox{for any}~~t\in[0,T].
        \end{aligned}
    \end{equation}
 Using the bound (\ref{2.1}), after a simple calculation, yields
\begin{align}\label{4.2}
 \|D(r)^{-1}\|_{1,\infty}+\|D(r)\|_{1,\infty}\leq C(R)\exp(cRt).
\end{align}
In addition, by the mean value theorem, the bound (\ref{2.1}) and Lemmas \ref{lem2.1}, \ref{lem2.2}, we have for any $s>\frac{d}{2}$
\begin{align}\label{4.3}
\|D(r)\|_{s,2}\leq C(R,T)\|r\|_{s,2},
\end{align}
and
\begin{align}\label{4.4}
 \|D(r_1)-D(r_2)\|_{s,2}\leq C(R,T)\|r_1, r_2\|_{s,2}\|r_1-r_2\|_{s,2}.
\end{align}
Indeed, due to the mean value theorem, there exists some $\theta\in(0,1)$ such that
\begin{align*}
  &\quad\|D(r_1)-D(r_2)\|_{s,2}\\
  &=\left\|\frac{d D}{dr}(\theta r_1+(1-\theta) r_2)\cdot(r_1-r_2)\right\|_{s,2} \\
  &\leq C\left\|\frac{d D}{dr}(\theta r_1+(1-\theta) r_2)\right\|_{\infty}\|r_1-r_2\|_{s,2}
  +C\|r_1-r_2\|_{\infty}\left\|\frac{d D}{dr}(\theta r_1+(1-\theta) r_2)\right\|_{s,2}\\
  &\leq C(R,T)\|r_1, r_2\|_{s,2}\|r_1-r_2\|_{s,2}.
\end{align*}

\begin{lemma}\label{G.1}
    For any smooth function $\mathbf{u}\in C([0,T]; X_N(\mathbb{T}))$ and integer $s>\frac{7}{2}$, there exists a unique solution
    \begin{align*}
    Q\in C([0,T];W^{s+1,2}(\mathbb{T},S_0^3))\cap L^2(0,T;W^{s+2,2}(\mathbb{T},S_0^3))
    \end{align*}
to the initial value problem
    \begin{equation}\label{22.2}
        \begin{cases}
        Q_t+\Truca(\mathbf{u}\cdot\nabla_x Q-\Theta Q+Q\Theta)=\Gamma L\triangle Q+\Truca\mathcal{K}(Q),  \mbox{in } \mathbb{T}\times (0,T) \\
        Q|_{t=0}=Q_0(x)\in W^{s+1,2}(\mathbb{T}, S_0^3).
        \end{cases}
    \end{equation}
Moreover, the mapping
\begin{eqnarray}\label{3.4}
\mathbf{u}\rightarrow Q[\mathbf{u}]: C([0,T];X_N(\mathbb{T}))\rightarrow C([0,T];W^{s+1,2}(\mathbb{T},S_0^3))\cap L^2(0,T;W^{s+2,2}(\mathbb{T},S_0^3))
\end{eqnarray}
is continuous on a bounded set $B\in C([0,T];X_N(\mathbb{T}))$, where $X_N$ is a finite dimensional space spanned by $\{\psi_m\}_{m=1}^N$, see (\ref{4.12}).
\end{lemma}
\begin{proof} Existence: \underline{Step 1}. Since the system \eqref{22.2} is a type of parabolic evolution system, we are able to establish the existence and uniqueness of finite-dimensional local approximate solutions $Q_m$ using the Galerkin method and the fixed point theorem, for further details, see \cite{chen, WD}. Then, we could extend the local solution to global in time using the following uniform a priori estimate.

\underline{Step 2}. Let $\alpha$ be a multi index such that $|\alpha|\leq s$. Taking $\alpha$-order derivative on both sides of the $m$-th order finite-dimensional approximate system of (\ref{22.2}), multiplying by $-\triangle \partial_x^\alpha Q_m$, then the trace and integrating over $\mathbb{T}$, we get
\begin{align}\label{4.8*}
&\quad\frac{1}{2}\frac{d}{dt}\|\partial_x^{\alpha+1}Q_m\|^2
    +\Gamma L\|\triangle\partial_x^\alpha Q_m\|^2\nonumber\\
    &=\Phi_R^{\mathbf{u},Q_m}(\partial_x^{\alpha+1}(\mathbf{u}\cdot \nabla_xQ_m-\Theta Q_m+Q_m\Theta), \partial_x^{\alpha+1}Q_m)\nonumber\\&\quad+\Phi_R^{\mathbf{u},Q_m}(\partial_x^{\alpha+1}\mathcal{K}(Q_m), \partial_x^{\alpha+1}Q_m).
\end{align}
For the first term on the right hand side of \eqref{4.8*}, using the H\"{o}lder inequality and Lemma \ref{lem2.1}, we obtain
\begin{align}\label{4.9*}
&\quad|\Phi_R^{\mathbf{u},Q_m}(\partial_x^{\alpha+1}(\mathbf{u}\cdot \nabla_xQ_m-\Theta Q_m+Q_m\Theta), \partial_x^{\alpha+1}Q_m)|\nonumber\\
&\leq C \Phi_R^{\mathbf{u},Q_m}\|\partial_x^{\alpha+1}Q_m\|\|\partial_x^{\alpha+1}(\mathbf{u}\cdot \nabla_xQ_m-\Theta Q_m+Q_m\Theta)\|\nonumber\\
&\leq C\Phi_R^{\mathbf{u},Q_m}\|\partial_x^{\alpha+1}Q_m\|(\|\mathbf{u}\|_{\infty}\|\partial_x^{\alpha+1}\nabla_x Q_m\|
+C\|\nabla_xQ_m\|_{\infty}\|\partial_x^{\alpha+1}\mathbf{u}\|\nonumber\\
&\qquad\qquad\qquad\qquad\quad+\|\nabla_x\mathbf{u}\|_{\infty}\|\partial_x^{\alpha+1}Q_m\|+\|Q_m\|_{\infty}\|\partial_x^{\alpha+1}\Theta\|)\nonumber\\
&\leq C\|\partial_x^{\alpha+1}Q_m\|^2+\frac{\Gamma L}{4}\|\partial_x^{\alpha+2}Q_m\|^2.
\end{align}
We only deal with the high-order term $Q_m{\rm tr}(Q_m^2)$ in $\mathcal{K}(Q_m)$, the rest of terms are trivial, using the H\"{o}lder inequality and Lemma \ref{lem2.1}, to get
\begin{align}\label{4.10*}
&\quad|\Phi_R^{\mathbf{u},Q_m}(\partial_x^{\alpha+1}(Q_m{\rm tr}(Q_m^2)), \partial_x^{\alpha+1}Q_m)|\nonumber\\
&\leq \Phi_R^{\mathbf{u},Q_m}\|\partial_x^{\alpha+1}Q_m\|\|\partial_x^{\alpha+1}(Q_m{\rm tr}(Q_m^2))\|\nonumber\\
&\leq \Phi_R^{\mathbf{u},Q_m}\|\partial_x^{\alpha+1}Q_m\|(\|\partial_x^{\alpha+1}Q_m\|\|Q_m\|^2_{\infty}+\|{\rm tr}(Q_m^2)\|_{\infty}\|\partial_x^{\alpha+1}Q_m\|)\nonumber\\
&\leq C\|\partial_x^{\alpha+1}Q_m\|^2.
\end{align}
Taking into account of \eqref{4.8*}-\eqref{4.10*}, taking sum of $|\alpha|\leq s$  and using the Gronwall lemma, we have
\begin{align}\label{4.11*}
\sup_{t\in [0,T]}\|Q_m\|_{s+1,2}^2
    +\int_{0}^{T}\Gamma L\| Q_m\|_{s+2,2}^2dt\leq C.
\end{align}
Then, using the estimate \eqref{4.11*}, it is also easily to show that
\begin{align}\label{4.12*}
\|Q_{m}\|_{W^{1,2}(0,T;L^2(\mathbb{T},S_0^3))}\leq C,
\end{align}
where the constant $C$ is independent of $m$.

\underline{Step 3}. Using the a priori estimates \eqref{4.11*}, \eqref{4.12*} and the Aubin-Lions lemma \ref{lem6.1}, we could show the compactness of the sequence of approximate solutions $Q_m$, actually the proof is easier than the argument of Lemma \ref{lem4.5}. Then, we could pass $m\rightarrow\infty$ to identify the limit, the proof is also easier than the argument in Subsection 4.4, here we omit it. This completes the proof of existence.

Uniqueness: The proof of uniqueness is similar to the following continuity argument.

 Next, we focus on showing the continuity of the mapping $\mathbf{u}\rightarrow Q[\mathbf{u}]$. Taking $\{\mathbf{u}_n\}_{n\geq 1}$ is a bounded sequence in
$C([0,T]; X_N(\mathbb{T}))$ with
\begin{equation}\label{22.5}
    \lim_{n\rightarrow \infty}\|\mathbf{u}_n-\mathbf{u}\|_{C([0,T];X_N(\mathbb{T}))}=0.
\end{equation}
Denote $Q_n=Q[\mathbf{u}_n]$, $Q=Q[\mathbf{u}]$, and $\bar{Q}_n=Q_n-Q$, then the continuity result \eqref{3.4} would follow if we could prove
\begin{equation}\label{22.6}
    \|\bar{Q}_n\|_{C([0,T];W^{s+1,2}(\mathbb{T},S_0^3))}^2+\|\bar{Q}_n\|_{L^2(0,T;W^{s+2,2}(\mathbb{T},S_0^3))}^2\leq C\sup_{t\in [0,T]}\|\mathbf{u}_n-\mathbf{u}\|_{X_N}^2.
\end{equation}
From (\ref{22.2}), we can get that $\bar{Q}_n$ satisfies the following system
\begin{equation}\label{22.7}
        \begin{cases}
        \frac{d}{dt}\bar{Q}_n-\Gamma L\triangle\bar{Q}_n\!\!\!\!\!&=\Phi_R^{\mathbf{u},Q}\big[(\mathbf{u}-\mathbf{u}_n)\cdot\nabla_xQ
        -\mathbf{u}_n\cdot\nabla_x\bar{Q}_n\\
        &+\Theta_n \bar{Q}_n-\bar{Q}_n\Theta_n+(\Theta_n-\Theta)Q-Q(\Theta_n-\Theta)\big]\\
        &+\left(\Phi_R^{\mathbf{u}_n,Q_n}-\Phi_R^{\mathbf{u},Q}\right)(\mathbf{u}_n\cdot\nabla_x Q_n-\Theta_n Q_n+Q_n\Theta_n)\\
        &+\Phi_R^{\mathbf{u},Q}(\mathcal{K}(Q_n)-\mathcal{K}(Q))\\
        &+\left(\Phi_R^{\mathbf{u}_n,Q_n}-\Phi_R^{\mathbf{u},Q}\right)\mathcal{K}(Q_n),~~\mbox{in } \mathbb{T}\times (0,T),\\
        \bar{Q}_n(0)=0.
        \end{cases}
    \end{equation}
Taking $\alpha$-order derivative on both sides of (\ref{22.7}) for $|\alpha|\leq s$, multiplying by $-\triangle \partial_x^\alpha\bar{Q}_n$, then the trace and integrating over $\mathbb{T}$, we arrive at
    \begin{align}\label{22.8}
    &\frac{1}{2}\frac{d}{dt}\|\partial_x^{\alpha+1}\bar{Q}_n\|^2
    +\Gamma L\|\triangle\partial_x^\alpha\bar{Q}_n\|^2\nonumber\\
    =&\int_{\mathbb{T}}\Phi_R^{\mathbf{u},Q}\partial_x^\alpha\bigg(\big[(\mathbf{u}-\mathbf{u}_n)\cdot\nabla_xQ
        -\mathbf{u}_n\cdot\nabla_x\bar{Q}_n\nonumber\\
        &+\Theta_n \bar{Q}_n-\bar{Q}_n\Theta_n+(\Theta_n-\Theta)Q-Q(\Theta_n-\Theta)\big]\nonumber\\
        &+\left(\Phi_R^{\mathbf{u}_n,Q_n}-\Phi_R^{\mathbf{u},Q}\right)(\mathbf{u}_n\cdot\nabla_x Q_n-\Theta_n Q_n+Q_n\Theta_n)\nonumber\\
        &+\Phi_R^{\mathbf{u},Q}(\mathcal{K}(Q_n)-\mathcal{K}(Q))\nonumber\\
        &+\left(\Phi_R^{\mathbf{u}_n,Q_n}-\Phi_R^{\mathbf{u},Q}\right)\mathcal{K}(Q_n)\bigg):(-\triangle\partial_x^{\alpha}\bar{Q}_n)dx\nonumber\\
    =&:I_1+ I_2 +I_3+ I_4.
    \end{align}
As $\{Q_n\}_{n\geq 1}$ and $Q$ are uniform bounded in
$$C([0,T];W^{s+1,2}(\mathbb{T},S_0^3))\cap L^2(0,T;W^{s+2,2}(\mathbb{T},S_0^3)).$$
We can estimate $I_1$ by Lemma \ref{lem2.1} and the H\"{o}lder inequality
\begin{align*}
 |I_1|&\leq \left\|\partial_x^\alpha\big[(\mathbf{u}-\mathbf{u}_n)\cdot\nabla_xQ
        -\mathbf{u}_n\cdot\nabla_x\bar{Q}_n+\Theta_n \bar{Q}_n-\bar{Q}_n\Theta_n+(\Theta_n-\Theta)Q-Q(\Theta_n-\Theta)\big]\right\|\\
        &\quad\times\|\triangle\partial_x^{\alpha}\bar{Q}_n\|\\
        &\leq C(\|\mathbf{u}-\mathbf{u}_n\|_{s,2}\|Q\|_{s+1,2}+\|\mathbf{u}_n\|_{s+1,2}\|\partial_x^{\alpha+1}\bar{Q}_n\|+\|\mathbf{u}-\mathbf{u}_n\|_{s+1,2}\|Q\|_{s,2})
        \|\triangle\partial_x^{\alpha}\bar{Q}_n\|\\
        &\leq \frac{\Gamma L}{4}\|\triangle\partial_x^{\alpha}\bar{Q}_n\|^2+C\|\partial_x^{\alpha+1}\bar{Q}_n\|^2+C\|\mathbf{u}-\mathbf{u}_n\|_{s+1,2}^2.
\end{align*}
For $I_2$, by Lemma \ref{lem2.1} and the H\"{o}lder inequality again, we have
\begin{align*}
  |I_2|&\leq C(\|\mathbf{u}-\mathbf{u}_n\|_{2,\infty}+\|\bar{Q}_n\|_{3,\infty})\|\partial_x^\alpha(\mathbf{u}_n\cdot\nabla_x Q_n-\Theta_n Q_n+Q_n\Theta_n)\|\|\triangle\partial_x^{\alpha}\bar{Q}_n\|\\
        &\leq \frac{\Gamma L}{4}\|\triangle\partial_x^{\alpha}\bar{Q}_n\|^2+C\|\partial_x^{\alpha+1}\bar{Q}_n\|^2+C\|\mathbf{u}-\mathbf{u}_n\|_{s+1,2}^2.
\end{align*}
Similarly, for terms $I_3, I_4$
\begin{align*}
  |I_3+I_4|&\leq \frac{\Gamma L}{4}\|\triangle\partial_x^{\alpha}\bar{Q}_n\|^2+C\|\partial_x^{\alpha+1}\bar{Q}_n\|^2+C\|\mathbf{u}-\mathbf{u}_n\|_{s+1,2}^2.
\end{align*}
Summing all the estimates up and taking sum for $|\alpha|\leq s$, we get
$$\frac{d}{dt}\|\bar{Q}_n\|_{{s+1,2}}^2+\frac{\Gamma L}{2}\|\bar{Q}_n\|_{{s+2,2}}^2
\leq C\|\bar{Q}_n\|_{{s+1,2}}^2+C\|\mathbf{u}_n-\mathbf{u}\|_{X_N}^2.$$
Applying the Gronwall lemma, then
$$\|\bar{Q}_n(t)\|_{{s+1,2}}^2+\frac{\Gamma L}{2}\int_0^t\|\bar{Q}_n\|_{{s+2,2}}^2d\xi
\leq Ce^{CT}\sup_{t\in [0,T]}\|\mathbf{u}_n-\mathbf{u}\|_{X_N}^2,$$
for any $t\in[0,T]$. So let $n \rightarrow\infty$, since $\sup_{t\in [0,T]}\|\mathbf{u}_n-\mathbf{u}\|_{X_N}^2\rightarrow 0$, then \eqref{22.6} follows.

Finally, we prove $Q\in S_0^3$, namely ${\rm tr}(Q)=0$ and $Q=Q^{\rm T}$ a.e in $\mathbb{T}\times [0,T]$. If we apply the transpose to the equation \eqref{22.2}(1), using the fact that $\|Q\|_{3,\infty}=\|Q^{{\rm T}}\|_{3,\infty}$, we have
$$(Q^{\rm T})_t+\Phi_{R}^{\mathbf{u}, Q^{{\rm T}}}(\mathbf{u}\cdot \nabla_x Q^{\rm T}-\Theta Q^{\rm T}+Q^{\rm T}\Theta)=\Gamma L \triangle Q^{{\rm T}}+\Phi_{R}^{\mathbf{u}, Q^{{\rm T}}}\mathcal{K}(Q^{\rm T}).$$
So $Q^{\rm T}$ also satisfies the equation. The uniqueness result leads to $Q=Q^{\rm T}$. The proof of ${\rm tr}(Q)=0$, we refer the reader to \cite{chen}.
\end{proof}

For $r_1=r[\mathbf{v}_1]$, $r_2=r[\mathbf{v}_2]$, $r_1-r_2$ satisfies
    \begin{align*}
      d(r_1&-r_2)+\mathbf{v}_1\cdot\nabla_x(r_1-r_2)dt-\frac{\gamma-1}{2}\Dv_x\mathbf{v}_1\cdot(r_1-r_2)dt\\
       &~\qquad=-\nabla_xr_2\cdot(\mathbf{v}_1-\mathbf{v}_2)dt-\frac{\gamma-1}{2}r_2\cdot\Dv_x(\mathbf{v}_1-\mathbf{v}_2)dt,
    \end{align*}
where $\mathbf{v}_1=\Phi_R^{\mathbf{u}_1,Q[\mathbf{u}_1]}\mathbf{u}_1$, $\mathbf{v}_2=\Phi_R^{\mathbf{u}_2,Q[\mathbf{u}_2]}\mathbf{u}_2$. Using the same argument as Breit-Feireisl-Hofmanov\'{a} \cite[Section 3.1]{breit} and the continuity of $Q[\mathbf{u}]$, see \eqref{3.4}, we are able to obtain the continuity of $r[\mathbf{u}]$ with respect to $\mathbf{u}\in C([0,T];X_N(\mathbb{T}))$, that is,
    \begin{equation}\label{22.11}
    \sup_{0\leq t\leq T}\|r[\mathbf{u}_1]-r[\mathbf{u}_2]\|^2\leq TC(N,R,T)\sup_{0\leq t\leq T}
    \|\mathbf{u}_1-\mathbf{u}_2\|_{X_N}^2.
    \end{equation}

We proceed to construct the approximate solution to the momentum equation. Let $\{\psi_m\}_{m=1}^\infty$ be an orthonormal basis of the space $H^1(\mathbb{T}, \mathbb{R}^3)$. Set the space
\begin{eqnarray}\label{4.12}
X_n=\mbox{span}\{\psi_1,\ldots ,\psi_n\}.
\end{eqnarray}
Let $P_{n}$ be an orthogonal projection from $H^1(\mathbb{T}, \mathbb{R}^3)$ into $X_n$.

We now find the approximate velocity field  $\mathbf{u}_n\in L^2(\Omega,C([0,T];X_n))$ to the following momentum equation
\begin{equation}\label{22.9}
        \begin{cases}
          d\langle\mathbf{u}_n,\psi_i\rangle+\Trucan \langle\mathbf{u}_n\nabla_x\mathbf{u}_n
          +r[\mathbf{u}_n]\nabla_xr[\mathbf{u}_n],\psi_i\rangle dt\\
          =\Trucan \langle D(r[\mathbf{u}_n])(\mathcal{L}\mathbf{u}_n-\Dv_x(L\nabla_x Q[\mathbf {u}_n]\odot \nabla_x Q[\mathbf {u}_n]
        -\mathcal{F}(Q[\mathbf {u}_n]){\rm I}_3)\nonumber\\
        \qquad\qquad +L\Dv_x(Q[\mathbf {u}_n]\triangle Q[\mathbf {u}_n]-\triangle Q[\mathbf {u}_n]Q[\mathbf {u}_n])),\psi_i\rangle dt\nonumber \\
          \quad+\Trucan \langle\mathbb{F}(r[\mathbf{u}_n],\mathbf{u}_n),\psi_i\rangle dW,
          i=1,\ldots,n \\
          \mathbf{u}_n(0)=P_n\mathbf{u}_0.
          \end{cases}
        \end{equation}
To handle the nonlinear $Q$-tensor terms and the noise term above, with the spirit of \cite{Hofmanova}, we define another $C^\infty$-smooth cut-off function
\begin{eqnarray*}
\Psi_K(z)=\left\{\begin{array}{ll}
1,~ |z|\leq K,\\
0,~ |z|>2K.
\end{array}\right.
\end{eqnarray*}
For any $\mathbf{v}=\sum_{i=1}^n \mathbf{v}_i\psi_i\in X_n$, define $\mathbf{v}^K=\sum_{i=1}^{n}\Psi_K(\mathbf{v}_i)\mathbf{v}_i\psi_i$, then we have $\|\mathbf{v}^K\|_{C([0,T];X_n)}\leq 2K$.

Define the mapping
    \begin{equation}\label{22.10}
        \begin{aligned}
          \langle\mathcal{T}[\mathbf{u}];\psi_i\rangle&=\langle\mathbf{u}^K_0;\psi_i\rangle- \int_{0}^{\cdot}\Phi_R^{\mathbf{u}^K,Q[\mathbf{u}^K]} \langle\mathbf{u}^K\nabla_x\mathbf{u}^K
          +r[\mathbf{u}^K]\nabla_xr[\mathbf{u}^K];\psi_i\rangle  dt\\
          &+\int_{0}^{\cdot}\Phi_R^{\mathbf{u}^K,Q[\mathbf{u}^K]} \langle D(r[\mathbf{u}^K])
          (\mathcal{L}\mathbf{u}^K-\Dv_x(L\nabla_x Q[\mathbf {u}^K]\odot \nabla_x Q[\mathbf {u}^K]-\mathcal{F}(Q[\mathbf {u}^K]){\rm I}_3)\\
          &\qquad\qquad+L\Dv_x(Q[\mathbf {u}^K]\triangle Q[\mathbf {u}^K]-\triangle Q[\mathbf {u}^K]Q[\mathbf {u}^K]));\psi_i\rangle dt\\
          &+\int_{0}^{\cdot}\Phi_R^{\mathbf{u}^K,Q[\mathbf{u}^K]} \langle\mathbb{F}(r[\mathbf{u}^K],\mathbf{u}^K);\psi_i\rangle dW,
          i=1,\ldots,n.
        \end{aligned}
        \end{equation}
 Next, we show that the mapping $\mathcal{T}$ is a contraction on $\mathcal{B}=L^2(\Omega;C([0,T^*];X_n))$ with fixed $K,n$ for $T^*$ small enough. Denote the right side of \eqref{22.10}  $\mathcal{T}_{det}$ as the deterministic part, and $\mathcal{T}_{sto}$
as the component $\int_{0}^{\cdot}\Phi_R^{\mathbf{u},Q}\langle\mathbb{F}(r[\mathbf{u}],\mathbf{u});\psi_i\rangle dW$ respectively.

Combining the assumption on initial data $Q_0$ and the definition of $\mathbf{u}^K$, we have after a easily calculation
\begin{align}\label{4.13*}
 \|Q[\mathbf{u}^K]\|_{C([0,T];W^{1,2})}\leq C(K,R),~ \mathbb{P} ~\mbox{a.s.}
\end{align}
Together estimates \eqref{2.1}, \eqref{4.4}, \eqref{4.13*}, the continuity results \eqref{22.11}, \eqref{3.4}  with the equivalence of norms on finite dimensional space $X_n$, we can show that the mapping $\mathcal{T}_{det}$ satisfies the estimate
    \begin{equation}\label{22.13}
    \|\mathcal{T}_{det}(\mathbf{u}_1)-\mathcal{T}_{det}(\mathbf{u}_2)\|_{\mathcal{B}}^2
    \leq T^*C(n,R,T,K)\|\mathbf{u}_1-\mathbf{u}_2\|_{\mathcal{B}}^2,
    \end{equation}
see also \cite{chen, WD} and using the Burkholder-Davis-Gundy inequality \eqref{2.5}, the mapping $\mathcal{T}_{sto}$ satisfies the estimate
\begin{align}\label{2.6}
&\quad\|\mathcal{T}_{sto}(\mathbf{u}_1)-\mathcal{T}_{sto}(\mathbf{u}_2)\|_{\mathcal{B}}^2\nonumber\\ &=\mathbb{E}\sup_{t\in [0,T^*]}\left\|\int_{0}^{t}\Phi_R^{{\mathbf{u}_1^K},Q[\mathbf{u}_1^K]} \mathbb{F}(r[\mathbf{u}_1^K],\mathbf{u}_1^K)-\Phi_R^{\mathbf{u}^K_2,Q[\mathbf{u}_2^K]} \mathbb{F}(r[\mathbf{u}_2^K],\mathbf{u}_2^K)dW \right\|^2_{X_n}\nonumber\\
&\leq C\mathbb{E}\int_{0}^{T^*}\left\|\Phi_R^{{\mathbf{u}_1^K},Q[\mathbf{u}_1^K]} \mathbb{F}(r[\mathbf{u}_1^K],\mathbf{u}_1^K)-\Phi_R^{\mathbf{u}^K_2,Q[\mathbf{u}_2^K]} \mathbb{F}(r[\mathbf{u}_2^K],\mathbf{u}_2^K)\right\|^2_{L_{2}(\mathfrak{U};X_n)}dt\nonumber\\
&\leq C\mathbb{E}\int_{0}^{T^*}\left|\Phi_R^{\mathbf{u}^K_1,Q[\mathbf{u}_1^K]}-\Phi_R^{\mathbf{u}^K_2,Q[\mathbf{u}_2^K]}\right|^2 \left\|\mathbb{F}(r[\mathbf{u}_1^K],\mathbf{u}_1^K)\right\|^2_{L_{2}(\mathfrak{U};X_n)}dt\nonumber\\
&\quad+C\mathbb{E}\int_{0}^{T^*}(\Phi_R^{\mathbf{u}^K_2,Q[\mathbf{u}_2^K]})^2
\left\|\mathbb{F}(r[\mathbf{u}_1^K],\mathbf{u}_1^K)-\mathbb{F}(r[\mathbf{u}_2^K],\mathbf{u}_2^K)\right\|^2_{L_{2}(\mathfrak{U};X_n)}dt\nonumber\\
&=:J_1+J_2.
\end{align}
Using the equivalence of norms on finite-dimensional space, assumption \eqref{2.5*} and the continuity result \eqref{22.11}, the bound \eqref{2.1}, we have
\begin{align}\label{4.22*}
J_2&\leq C\mathbb{E}\int_{0}^{T^*}(\Phi_R^{\mathbf{u}^K_2,Q[\mathbf{u}_2^K]})^2
\left\|\mathbb{F}(r[\mathbf{u}_1^K],\mathbf{u}_1^K)-\mathbb{F}(r[\mathbf{u}_2^K],\mathbf{u}_2^K)\right\|^2_{L_{2}(\mathfrak{U};L^2)}dt\nonumber\\
&\leq C\mathbb{E}\int_{0}^{T^*}(\Phi_R^{\mathbf{u}^K_2,Q[\mathbf{u}_2^K]})^2(\|r[\mathbf{u}_1^K], r[\mathbf{u}_2^K]\|^2_{1,\infty}+\|\mathbf{u}_1^K, \mathbf{u}_2^K\|^2_{2,\infty})\|r[\mathbf{u}_1^K]- r[\mathbf{u}_2^K], \mathbf{u}^K_1-\mathbf{u}^K_2\|^2dt\nonumber\\
&\leq T^*C(n,R,K,T)\|\mathbf{u}_1-\mathbf{u}_2\|_{\mathcal{B}}^2.
\end{align}
By the mean value theorem, the equivalence of norms on finite-dimensional space, assumption \eqref{2.5} and continuity result \eqref{3.4}, we also have
\begin{align}\label{4.22}
J_1\leq T^*C(n,K,T)\|\mathbf{u}_1-\mathbf{u}_2\|_{\mathcal{B}}^2.
\end{align}

Combining (\ref{4.13*})-(\ref{4.22}), we infer that there exists approximate solution sequence belonging to $L^{2}(\Omega; C([0,T_{*}]; X_{n}))$ to momentum equation for small time $T^{*}$ by the Banach fixed point theorem. Here we first assume that the estimates (\ref{4.20}),(\ref{4.21}) hold. Then, we could extend the existence time $T^*$ to any $T>0$ for any fixed $n,K$.

Next, we pass $K\rightarrow \infty$ to construct the approximate solution $(r_n, \mathbf{u}_n, Q_n)$ for any fixed $n$. Define the stopping time $\tau_K$
$$\tau_K=\inf\left\{t\in [0,T]; \sup_{\xi\in [0,t]}\left\|\mathbf{u}_n^K(\xi)\right\|_{X_n}\geq K\right\},$$
with the convention $\inf \varnothing=T$. Note that $\tau_{K_{1}}\geq \tau_{K_{2}}$ if $K_1\geq K_2$, due to the uniqueness, we have $(r_n^{K_1}, \mathbf{u}_n^{K_1}, Q_n^{K_1})=(r_n^{K_2}, \mathbf{u}_n^{K_2}, Q_n^{K_2})$ on the interval $[0, \tau_{K_2})$. Therefore, we can define $(r_n, \mathbf{u}_n, Q_n)=(r_n^{K}, \mathbf{u}_n^{K}, Q_n^{K})$ on interval $[0, \tau_{K})$. Note that
\begin{align*}
\mathbb{P}\left\{\sup_{K\in \mathbb{N}^+}\tau_{K}=T\right\}&=1-\mathbb{P}\left\{\left(\sup_{K\in \mathbb{N}^+}\tau_{K}=T\right)^c\right\}=1-\mathbb{P}\left\{\sup_{K\in \mathbb{N}^+}\tau_{K}<T\right\}\\
&\geq 1-\mathbb{P}\left\{\tau_{K}<T\right\}=1-\mathbb{P}\left\{\sup_{t\in [0,T]}\|\mathbf{u}_n^K\|_{X_n}\geq K\right\}.
\end{align*}
From the Chebyshev inequality, estimate (\ref{4.20}) and the equivalence of norms on finite dimensional space, we know
\begin{eqnarray*}
\lim_{K\rightarrow \infty}\mathbb{P}\left\{\sup_{t\in [0,T]}\|\mathbf{u}_n^K\|_{X_n}\geq K\right\}=0,
\end{eqnarray*}
which leads to
\begin{align*}
\mathbb{P}\left\{\sup_{K\in \mathbb{N}^+}\tau_{K}=T\right\}=1.
\end{align*}
As a result, we could extend the existence time interval $[0, \tau_{K})$ to $[0, T]$ for any $T>0$, obtaining the global existence of approximate solution sequence $(r_n, \mathbf{u}_n, Q_n)$.

{\bf 4.2 Uniform estimates}. In this subsection, we derive the a priori estimates that hold uniformly for $n\geq 1$, which allow us to extend the existence interval to any $T>0$ and provide a preliminary for our stochastic compactness argument.

 Taking $\alpha$-order derivative on both sides of system \eqref{qnt} in the x-variable for $|\alpha|\leq s$, then taking inner product with $\partial_x^\alpha r_n$ on both sides of equation (\ref{qnt})(1) and applying the It\^{o} formula to function $\|\partial_x^\alpha \mathbf{u}_n\|^2$, we obtain
        \begin{align}\label{3.1}
        &\frac{1}{2}d\|\partial_x^\alpha r_n\|^2 +\Trucan \left( \mathbf{u}_n\cdot\nabla_x \partial_x^\alpha r_n+ \frac{\gamma-1}{2}r_n\Dv_x\partial_x^\alpha\mathbf{u}_n, \partial_x^\alpha r_n \right) dt\nonumber\\
        =&\Trucan \left(\mathbf{u}_n\cdot  \partial_x^\alpha\nabla_x r_n
        -\partial_x^\alpha(\mathbf{u}_n\cdot\nabla_x r_n),\partial_x^\alpha r_n\right) dt\nonumber\\
        &+\frac{\gamma-1}{2}\Trucan \left( r_n \partial_x^\alpha\Dv_x\mathbf{u}_n
        -\partial_x^\alpha(r_n\Dv_x\mathbf{u}_n),\partial_x^\alpha r_n\right) dt\nonumber\\
        =&:\left( T_1^n dt+T_2^ndt,\partial_x^\alpha r_n\right),
        \end{align}
    and
        \begin{align}\label{3.2}
        &\frac{1}{2}d\|\partial_x^\alpha\mathbf{u}_n\|^2+\Trucan (\mathbf{u}_n\nabla_x\partial_x^\alpha\mathbf{u}_n+r_n\nabla_x\partial_x^\alpha r_n, \partial_x^\alpha \mathbf{u}_n) dt\nonumber\\
        &-\Trucan  (D(r_n)\mathcal{L}(\partial_x^\alpha\mathbf{u}_n),\partial_x^\alpha \mathbf{u}_n) dt\nonumber\\
        &+\Trucan  \left(D(r_n)\Dv_x\partial_x^\alpha\left(L\nabla_x Q_n\odot\nabla_x Q_n-\frac{L}{2}|\nabla_x Q_n|^2{\rm I}_3\right),\partial_x^\alpha \mathbf{u}_n\right) dt\nonumber\\
    &-\Trucan  \left(D(r_n)\Dv_x\partial_x^\alpha\left(\frac{a}{2}{\rm I}_3{\rm tr}(Q_n^2)-\frac{b}{3}{\rm I}_3{\rm tr}(Q_n^3)+\frac{c}{4}{\rm I}_3{\rm tr}^2(Q_n^2)\right), \partial_x^\alpha \mathbf{u}_n\right) dt\nonumber\\
    &-\Trucan (D(r_n)\Dv_x\left(L\partial_x^\alpha(Q_n\triangle Q_n-\triangle Q_n Q_n)\right),\partial_x^\alpha \mathbf{u}_n) dt\nonumber\\
          =&\Trucan \left(\mathbf{u}_n\partial_x^\alpha\nabla_x\mathbf{u}_n
          -\partial_x^\alpha(\mathbf{u}_n\nabla_x\mathbf{u}_n), \partial_x^\alpha \mathbf{u}_n\right) dt\nonumber\\
          &+\Trucan (r_n\partial_x^\alpha\nabla_xr_n-\partial_x^\alpha(r_n\nabla_xr_n), \partial_x^\alpha \mathbf{u}_n) dt\nonumber\\
          &-\Trucan (D(r_n)\partial_x^\alpha\mathcal{L}\mathbf{u}_n
          -\partial_x^\alpha(D(r_n)\mathcal{L}\mathbf{u}_n),\partial_x^\alpha \mathbf{u}_n) dt \nonumber\\
          &+\Trucan \bigg( D(r_n)\partial_x^\alpha\Dv_x\left(L\nabla_x Q_n\odot\nabla_x Q_n-\frac{L}{2}|\nabla_x Q_n|^2{\rm I}_3\right)\nonumber\\
    &\qquad\qquad-\partial_x^\alpha\left(D(r_n)\Dv_x\left(L\nabla_x Q_n\odot\nabla_x Q_n-\frac{L}{2}|\nabla_x Q_n|^2{\rm I}_3\right)\right),\partial_x^\alpha \mathbf{u}_n\bigg) dt\nonumber \\
          &-\Trucan \bigg(D(r_n)\partial_x^\alpha\Dv_x\left(\frac{a}{2}{\rm I}_3 {\rm tr}(Q_n^2)-\frac{b}{3}{\rm I}_3{\rm tr}(Q_n^3)+\frac{c}{4}{\rm I}_3 {\rm tr}^2(Q_n^2)\right)\nonumber\\
          &\qquad\qquad-\partial_x^\alpha\left(D(r_n)\Dv_x\left(\frac{a}{2}{\rm I}_3 {\rm tr}(Q_n^2)-\frac{b}{3}{\rm I}_3 {\rm tr}(Q_n^3)+\frac{c}{4}{\rm I}_3 {\rm tr}^2(Q_n^2)\right)\right),\partial_x^\alpha \mathbf{u}_n\bigg) dt\nonumber\\
          &-\Trucan \bigg(D(r_n)\partial_x^\alpha\Dv_xL(Q_n\triangle Q_n-\triangle Q_n Q_n)\nonumber\\
          &\qquad\qquad-\partial_x^\alpha(D(r_n)\Dv_xL(Q_n\triangle Q_n-\triangle Q_n Q_n)), \partial_x^\alpha \mathbf{u}_n\bigg) dt\nonumber\\
          &+\Trucan (\partial_x^\alpha\mathbb{F}(r_n,\mathbf{u}_n), \partial_x^\alpha \mathbf{u}_n ) dW+\frac{1}{2}(\Trucan)^2 \sum_{k\geq1} \int_{\mathbb{T}}|\partial_x^\alpha\mathbb{F}
        (r_n,\mathbf{u}_n)e_k|^2dxdt\nonumber\\
          =&:\sum_{i=3}^8 ( T_i^n,  \partial_x^\alpha \mathbf{u}_n) dt
          +\Trucan ( \partial_x^\alpha\mathbb{F}(r_n,\mathbf{u}_n), \partial_x^\alpha \mathbf{u}_n ) dW\nonumber\\&+\frac{1}{2}(\Trucan)^2 \sum_{k\geq1} \int_{\mathbb{T}}|\partial_x^\alpha\mathbb{F}
        (r_n,\mathbf{u}_n)e_k|^2dxdt.
        \end{align}
To handle the highest order term ${\rm div}_x(Q_n\triangle Q_n-\triangle Q_nQ_n)$, we multiply $-D(r_n)\triangle\partial_x^\alpha Q_n$ in equation \ref{qnt}(3) instead of $-\triangle\partial_x^\alpha Q_n$, then take the trace and integrate over $\mathbb{T}$, to get
        \begin{align}\label{33.9}
        &\frac{1}{2}d\|\sqrt{D(r_n)}\nabla_x\partial_x^\alpha Q_n\|^2-\frac{1}{2}\int_{\mathbb{T}}D(r_n)_t|\nabla_x\partial_x^\alpha Q_n|^2dxdt\nonumber\\
        &-\int_{\mathbb{T}}\nabla_xD(r_n)(\partial_x^\alpha Q_n)_t:\nabla_x\partial_x^\alpha Q_ndxdt
        +\Gamma L\|\sqrt{D(r_n)}\triangle\partial_x^\alpha Q_n\|^2dt\nonumber\\
        &-\Trucan\int_{\mathbb{T}}D(r_n)(\mathbf{u}_n\cdot\nabla_x \partial_x^\alpha Q_n):
        \triangle\partial_x^\alpha Q_ndxdt\nonumber\\
        &-\Trucan \int_{\mathbb{T}}D(r_n)((\partial_x^\alpha \Theta_n)Q_n-Q_n(\partial_x^\alpha \Theta_n)):
        \triangle\partial_x^\alpha Q_ndxdt\nonumber\\
        &-\Gamma\Trucan \int_{\mathbb{T}}D(r_n)\partial_x^\alpha\left(aQ_n-b\left(Q_n^2-\frac{{\rm I}_3}{3}{\rm tr}(Q_n^2)+cQ_n{\rm tr}(Q_n^2)\right)\right):
        \triangle\partial_x^\alpha Q_ndxdt\nonumber\\
        =&-\int_{\mathbb{T}}D(r_n)\Trucan (\mathbf{u}_n\cdot \partial_x^\alpha \nabla_xQ_n-\partial_x^\alpha(\mathbf{u}_n\cdot\nabla_xQ_n)):\triangle\partial_x^\alpha Q_ndxdt\nonumber\\
        &-\int_{\mathbb{T}}D(r_n)\Trucan ((\partial_x^\alpha \Theta_n)Q_n-Q_n(\partial_x^\alpha \Theta_n)-\partial_x^\alpha (\Theta_nQ_n-Q_n\Theta_n)):\triangle\partial_x^\alpha Q_ndxdt\nonumber\\
        =&:\int_{\mathbb{T}} (T_9+T_{10}):\triangle\partial_x^\alpha Q_ndxdt.
        \end{align}

 We next estimate all the right hand side terms. Using Lemma \ref{lem2.1} and the H\"{o}lder inequality,

    \begin{equation}\label{19}
        \begin{aligned}
        |\langle T_1^n, \partial_x^\alpha r_n\rangle|&\leq C\Trucan(\|\nabla_x\mathbf{u}_n\|_\infty\|\partial_x^\alpha r_n\|+\|\nabla_xr_n\|_\infty\|\partial_x^\alpha\mathbf{u}_n\|)\|\partial_x^\alpha r_n\|\\
        &\leq C(R)(\|\partial_x^\alpha r_n\|^2+\|\partial_x^\alpha\mathbf{u}_n\|^2),\\
        |\langle T_2^n,\partial_x^\alpha r_n\rangle|&\leq C\Trucan(\|\nabla_xr_n\|_\infty\|\partial_x^\alpha\mathbf{u}_n\|+\|\Dv_x\mathbf{u}_n\|_\infty\|\partial_x^\alpha r_n\|)\|\partial_x^\alpha r_n\|\\
        &\leq C(R)(\|\partial_x^\alpha\mathbf{u}_n\|^2+\|\partial_x^\alpha r_n\|^2).
        \end{aligned}
    \end{equation}
Also using Lemma \ref{lem2.1}, estimates (\ref{4.2}), (\ref{4.3})  and the H\"{o}lder inequality, we have the following estimates for $T_3^n$ to $T_{10}^n$
    \begin{align}\label{3.5}
|( T_3^n,\partial_x^\alpha \mathbf{u}_n)|
    &\leq C\Trucan \|\nabla_x\mathbf{u}_n\|_\infty\|\partial_x^\alpha\mathbf{u}_n\|^2\leq C(R)\|\partial_x^\alpha\mathbf{u}_n\|^2,\\
    |( T_4^n,\partial_x^\alpha \mathbf{u}_n)|
    &\leq C\Trucan \|\nabla_xr_n\|_\infty\|\partial_x^\alpha r_n\|\|\partial_x^\alpha\mathbf{u}_n\|\leq C(R)(\|\partial_x^\alpha r_n\|^2+\|\partial_x^\alpha\mathbf{u}_n\|^2),\\
    |( T_5^n,\partial_x^\alpha \mathbf{u}_n)|
    &\leq C\Trucan (\|\nabla_xD(r_n)\|_\infty\|\partial_x^{\alpha-1}\mathcal{L}\mathbf{u}_n\|
    +\|\mathcal{L}\mathbf{u}_n\|_\infty\|\partial_x^\alpha D(r_n)\|)\|\partial_x^\alpha\mathbf{u}_n\|\nonumber\\
    &\leq C(R)\Trucan (\|\partial_x^{\alpha+1}\mathbf{u}_n\|+\|\partial_x^\alpha r_n\|)\|\partial_x^\alpha\mathbf{u}_n\|\nonumber\\
    &\leq \frac{\nu}{8}\Trucan \|\sqrt{D(r_n)}\partial_x^{\alpha+1}\mathbf{u}_n\|^2+C(R)(\|\partial_x^\alpha r_n\|^2+\|\partial_x^\alpha\mathbf{u}_n\|^2),\\
    |( T_6^n,\partial_x^\alpha \mathbf{u}_n)|
    &\leq C\Trucan \left(\|\nabla_xD(r_n)\|_\infty \left\|\partial_x^\alpha\left(L\nabla_x Q_n\odot\nabla_x Q_n-\frac{L}{2}|\nabla_x Q_n|^2{\rm I}_3\right)\right\|\right)\|\partial_x^\alpha\mathbf{u}_n\|\nonumber\\
    &\quad+\left\|\Dv_x\left(L\nabla_x Q_n\odot\nabla_x Q_n-\frac{L}{2}|\nabla_x Q_n|^2{\rm I}_3\right)\right\|_\infty\|\partial_x^\alpha r_n\|\|\partial_x^\alpha\mathbf{u}_n\|\nonumber\\
    &\leq C\Trucan(\|\nabla_xD(r_n)\|_\infty\|\nabla_xQ_n\|_\infty\|\partial_x^{\alpha+1}Q_n\|\nonumber\\&\qquad\qquad\qquad+\|\nabla_xQ_n\|_\infty
    \|Q_n\|_{2,\infty}\|\partial_x^{\alpha}r_n\|)\|\partial_x^\alpha\mathbf{u}_n\|\nonumber\\
    &\leq C(R)(\|\partial_x^{\alpha+1}Q_n\|^2+\|\partial_x^{\alpha}r_n\|^2+\|\partial_x^\alpha\mathbf{u}_n\|^2),\\
    |( T_7^n, \partial_x^\alpha \mathbf{u}_n)|
    &\leq C\Trucan \bigg(\|\nabla_xD(r_n)\|_\infty\left\|\partial_x^{\alpha-1}\Dv_x\left(\frac{a}{2}{\rm I}_3{\rm tr}(Q_n^2)
    -\frac{b}{3}{\rm I}_3{\rm tr}(Q_n^3)+\frac{c}{4}{\rm I}_3{\rm tr}^2(Q_n^2)\right)\right\|\nonumber\\
    &\quad+\left\|\Dv_x\left(\frac{a}{2}{\rm I}_3{\rm tr}(Q_n^2)-\frac{b}{3}{\rm I}_3{\rm tr}(Q_n^3)+\frac{c}{4}{\rm I}_3{\rm tr}^2(Q_n^2)\right)\right\|_\infty\|\partial_x^{\alpha}r_n\|\bigg)
    \|\partial_x^\alpha\mathbf{u}_n\|\nonumber\\
     &\leq C\Trucan(\|Q_n\|_{\infty}^3+\|Q_n\|_{1,\infty})(\|\partial_x^{\alpha}Q_n\|^2+\|\partial_x^{\alpha}r_n\|^2+\|\partial_x^\alpha \mathbf{u}_n\|^2)\nonumber\\
    &\leq C(R)(\|\partial_x^{\alpha}Q_n\|^2
    +\|\partial_x^{\alpha}r_n\|^2+\|\partial_x^\alpha\mathbf{u}_n\|^2),\\
    |( T_8^n, \partial_x^\alpha \mathbf{u}_n)|
    &\leq C\Trucan (\|\nabla_xD(r_n)\|_\infty\|\partial_x^\alpha(Q_n\triangle Q_n-\triangle Q_n Q_n)\|\nonumber\\
    &\qquad\qquad+\|\Dv_x(Q_n\triangle Q_n-\triangle Q_n Q_n)\|_\infty\|\partial_x^\alpha D(r_n)\|)\|\partial_x^\alpha \mathbf{u}_n\|\nonumber\\
    &\leq C\Trucan(\|Q_n\|_\infty\|\triangle\partial_x^{\alpha}Q_n\|+\|Q_n\|_{2,\infty}\|\partial_x^{\alpha}Q_n\|)\|\partial_x^\alpha \mathbf{u}_n\|\nonumber\\
    &\quad+C\Trucan(\|\nabla_xQ_n\|_\infty\|Q_n\|_{2,\infty}+\|Q_n\|_\infty\|Q_n\|_{3,\infty})\|\partial_x^\alpha r_n\|\|\partial_x^\alpha \mathbf{u}_n\|\nonumber\\
    &\leq C(R)(\|\partial_x^{\alpha}Q_n\|^2+\|\partial_x^{\alpha}r_n\|^2+\|\partial_x^\alpha \mathbf{u}_n\|^2)+\frac{\Gamma L}{8}\|\sqrt{D(r_n)}\triangle\partial_x^{\alpha}Q_n\|^2,
    \end{align}
and
    \begin{align}\label{33.8}
    &\quad\left|\int_{\mathbb{T}} (T_9+T_{10}):\triangle\partial_x^\alpha Q_ndx\right|\nonumber\\&\leq C(R)\left\|\Trucan (\mathbf{u}_n\cdot \partial_x^\alpha \nabla_xQ_n-\partial_x^\alpha(\mathbf{u}_n\cdot\nabla_xQ_n))\right\|\|\sqrt{D(r_n)}\triangle\partial_x^\alpha Q_n\|\nonumber\\
    &\quad+C(R)\left\|\Trucan ((\partial_x^\alpha \Theta_n)Q_n-Q_n(\partial_x^\alpha \Theta_n)-\partial_x^\alpha (\Theta_nQ_n-Q_n\Theta_n))\right\|\|\sqrt{D(r_n)}\triangle\partial_x^\alpha Q_n\|\nonumber\\
    &\leq C(R)\Trucan (\|\nabla_x\mathbf{u}_n\|_\infty\|\partial_x^{\alpha}Q_n\|+\|\nabla_xQ_n\|_\infty\|\partial_x^{\alpha}\mathbf{u}_n\|)\|\sqrt{D(r_n)}\triangle\partial_x^\alpha Q_n\|\nonumber\\
    &\quad+ C(R)\Trucan (\|\nabla_xQ_n\|_\infty\|\partial_x^{\alpha}\mathbf{u}_n\|+\|\nabla_x\mathbf{u}_n\|_\infty\|\partial_x^{\alpha}Q_n\|)\|\sqrt{D(r_n)}\triangle\partial_x^\alpha Q_n\|\nonumber\\
    &\leq C(R)(\|\partial_x^{\alpha}Q_n\|^2+\|\partial_x^{\alpha}\mathbf{u}_n\|^2)+\frac{\Gamma L}{8}\|\sqrt{D(r_n)}\triangle\partial_x^{\alpha}Q_n\|^2.
    \end{align}
According to the assumption \eqref{2.5} on $\mathbb{G}$ and the Remark \ref{rem2.6}, we could have the estimate
        \begin{align}\label{33.12}
        &\sum_{k\geq1}\int_{0}^{t}(\Trucan)^2 \int_{\mathbb{T}}|\partial_x^\alpha\mathbb{F}(r_n,\mathbf{u}_n)e_{k}|^2dxd\xi\nonumber\\
        \leq &~C\int_{0}^{t}(\Trucan)^2 \int_{\mathbb{T}}\|r_n,\nabla_x\mathbf{u}_n\|_{1,\infty}^2 (\|\partial_x^\alpha r_n\|^2+\|\partial_x^\alpha\mathbf{u}_n\|^2)dxd\xi\nonumber\\
        \leq &~C(R)\int_{0}^{t}\int_{\mathbb{T}}(\|\partial_x^\alpha r_n\|^2+\|\partial_x^\alpha\mathbf{u}_n\|^2)dxd\xi.
        \end{align}

Next, we proceed to estimate the terms on the left hand side of (\ref{3.1})-(\ref{33.9}). Integration by parts, we get
\begin{align}
  &\left|\Trucan \int_{\mathbb{T}}\mathbf{u}_n\cdot\nabla_x \partial_x^\alpha r_n\partial_x^\alpha r_ndx\right|=
  \left|\frac{1}{2}\Trucan \int_{\mathbb{T}}\mathbf{u}_n\cdot\nabla_x (\partial_x^\alpha r_n)^2dx\right|\nonumber\\
  =&\left|\frac{1}{2}\Trucan \int_{\mathbb{T}}\Dv_x\mathbf{u}_n|\partial_x^\alpha r_n|^2dx\right|\leq C(R)\|\partial_x^\alpha r_n\|^2,
\end{align}
and
\begin{align}\label{4.9}
      &\Trucan\int_{\mathbb{T}}(\mathbf{u}_n\cdot
        \nabla_x\partial_x^\alpha\mathbf{u}_n+r_n\cdot\nabla_x\partial_x^\alpha r_n)\cdot\partial_x^\alpha\mathbf{u}_ndx\nonumber\\
      =&-\frac{1}{2}\Trucan\int_{\mathbb{T}}|\partial_x^\alpha\mathbf{u}_n|^2\Dv_x\mathbf{u}_ndx
      -\Trucan\int_{\mathbb{T}}r_n\Dv_x\partial_x^\alpha\mathbf{u}_n\partial_x^\alpha r_ndx\nonumber\\
      &-\Trucan\int_{\mathbb{T}}\nabla_xr_n\cdot\partial_x^\alpha\mathbf{u}_n\partial_x^\alpha r_ndx\nonumber\\
      \leq &~C(R)(\|\partial_x^\alpha\mathbf{u}_n\|^2+\|\partial_x^\alpha r_n\|^2)-\Trucan\int_{\mathbb{T}}r_n\Dv_x\partial_x^\alpha\mathbf{u}_n\partial_x^\alpha r_ndx,
    \end{align}
as well as we have by estimate (\ref{4.2})
    \begin{align}
      &\Trucan\int_{\mathbb{T}}D(r_n)\mathcal{L}(\partial_x^\alpha\mathbf{u}_n)\cdot \partial_x^\alpha\mathbf{u}_ndx\nonumber\\
      =&~\Trucan\int_{\mathbb{T}}D(r_n)(\upsilon\triangle\partial_x^\alpha\mathbf{u}_n+(\upsilon+\lambda)\nabla_x\Dv_x
      \partial_x^\alpha\mathbf{u}_n)\cdot \partial_x^\alpha\mathbf{u}_ndx\nonumber\\
      =&-\Trucan\int_{\mathbb{T}}D(r_n)(\upsilon|\nabla_x\partial_x^\alpha\mathbf{u}_n|^2+(\upsilon+\lambda)|\Dv_x\partial_x^\alpha\mathbf{u}_n|^2)dx\nonumber\\
      &-\Trucan\int_{\mathbb{T}}\nabla_x D(r_n)(\upsilon\nabla_x\partial_x^\alpha\mathbf{u}_n+(\upsilon+\lambda)\Dv_x
      \partial_x^\alpha\mathbf{u}_n)\partial_x^\alpha\mathbf{u}_ndx\nonumber\\
      \leq &-\Trucan\int_{\mathbb{T}}D(r_n)(\upsilon|\nabla_x\partial_x^\alpha\mathbf{u}_n|^2+(\upsilon+\lambda)|\Dv_x\partial_x^\alpha\mathbf{u}_n|^2)dx\nonumber\\
      &+\Trucan \|\nabla_x D(r_n)\|_{\infty}(\upsilon\|\nabla_x\partial_x^\alpha\mathbf{u}_n\|+(\upsilon+\lambda)\|\Dv_x\partial_x^\alpha\mathbf{u}_n\|)\|\partial_x^\alpha\mathbf{u}_n\|\nonumber\\
      \leq &-\Trucan\int_{\mathbb{T}}D(r_n)(\upsilon|\nabla_x\partial_x^\alpha\mathbf{u}_n|^2+(\upsilon+\lambda)|\Dv_x\partial_x^\alpha\mathbf{u}_n|^2)dx
      +C(R)\|\partial_x^\alpha\mathbf{u}_n\|^2\nonumber\\&+\frac{1}{8}\Trucan (\upsilon\|\sqrt{D(r_n)}\nabla_x\partial_x^\alpha\mathbf{u}_n\|^2+(\upsilon+\lambda)\|\sqrt{D(r_n)}\Dv_x\partial_x^\alpha\mathbf{u}_n\|^2).
    \end{align}
Also integration by parts, estimate (\ref{4.2}), Lemmas \ref{lem2.1},\ref{lem2.4} and the H\"{o}lder inequality give
    \begin{align}\label{4.11}
      &\Trucan\int_{\mathbb{T}}D(r_n)\Dv_x(L\partial_x^\alpha(Q_n\triangle Q_n-\triangle Q_n Q_n))\cdot \partial_x^\alpha\mathbf{u}_ndx\nonumber\\
      =&-L\Trucan\int_{\mathbb{T}}D(r_n)\partial_x^\alpha(Q_n\triangle Q_n-\triangle Q_n Q_n):\partial_x^\alpha\nabla_x\mathbf{u}_n^{{\rm T}}dx\nonumber\\
      &-L\Trucan\int_{\mathbb{T}}\nabla_x D(r_n)\partial_x^\alpha(Q_n\triangle Q_n-\triangle Q_n
      Q_n)\cdot\partial_x^\alpha\mathbf{u}_ndx\nonumber\\
      =&-L\Trucan\int_{\mathbb{T}}D(r_n)(Q_n\triangle \partial_x^\alpha Q_n-\triangle \partial_x^\alpha Q_n Q_n):\partial_x^\alpha\nabla_x\mathbf{u}_n^{{\rm T}}dx\nonumber\\
     & -L\Trucan\int_{\mathbb{T}}\nabla_x D(r_n)\partial_x^\alpha(Q_n\triangle Q_n-\triangle Q_n
      Q_n)\cdot\partial_x^\alpha\mathbf{u}_ndx\nonumber\\
      &-L\Trucan\int_{\mathbb{T}}D(r_n)(\partial_x^\alpha(Q_n\triangle Q_n-\triangle Q_n Q_n)-(Q_n\triangle \partial_x^\alpha Q_n-\triangle \partial_x^\alpha Q_n Q_n)):\partial_x^\alpha\nabla_x\mathbf{u}_n^{{\rm T}}dx\nonumber\\
      \leq&-L\Trucan\int_{\mathbb{T}}D(r_n)((\partial_x^\alpha\Theta_n)Q_n-Q_n(\partial_x^\alpha\Theta_n)):\triangle \partial_x^\alpha Q_ndx\nonumber\\
      &+C(R)\Trucan(\|\partial_x^\alpha(Q_n\triangle Q_n-\triangle Q_nQ_n)\|\|\partial_x^\alpha\mathbf{u}_n\|\nonumber\\
  &\quad\qquad+\|\partial_x^\alpha(Q_n\triangle Q_n-\triangle Q_n Q_n)-(Q_n\triangle \partial_x^\alpha Q_n-\triangle \partial_x^\alpha Q_n Q_n)\|\|\partial_x^{\alpha+1}\mathbf{u}_n\|)\nonumber\\
  \leq& -L\Trucan\int_{\mathbb{T}}D(r_n)((\partial_x^\alpha\Theta_n)Q_n-Q_n(\partial_x^\alpha\Theta_n)):\triangle \partial_x^\alpha Q_ndx\nonumber\\
  &+C(R)\Trucan(\|Q_n\|_\infty\|\triangle\partial_x^{\alpha}Q_n\|+\|Q_n\|_{2,\infty}\|\partial_x^{\alpha}Q_n\|)\|\partial_x^{\alpha}\mathbf{u}_n\|\nonumber\\
  &+ C(R)\Trucan(\|\nabla_xQ_n\|_\infty\|\partial_x^{\alpha+1}Q_n\|+\|Q_n\|_{2,\infty}\|\partial_x^{\alpha}Q_n\|)
  \|\partial_x^{\alpha+1}\mathbf{u}_n\|\nonumber\\
  \leq&-L\Trucan\int_{\mathbb{T}}D(r_n)((\partial_x^\alpha\Theta_n)Q_n-Q_n(\partial_x^\alpha\Theta_n)):\triangle \partial_x^\alpha Q_ndx\nonumber\\
  &+C(R)(\|\triangle\partial_x^{\alpha}Q_n\|+\|\partial_x^{\alpha}Q_n\|)\|\partial_x^{\alpha}\mathbf{u}_n\|
  +C(R)(\|\partial_x^{\alpha+1}Q_n\|+\|\partial_x^{\alpha}Q_n\|)\|\partial_x^{\alpha+1}\mathbf{u}_n\|.
    \end{align}
\begin{remark} Actually, Lemma \ref{lem2.4} requires that the symmetric matrices are $3\times 3$,
here, the $\triangle\partial_x^\alpha Q, \partial_x^\alpha \nabla \mathbf{u}^{\mathrm{T}}$ can be seen as the vector with each component is a $3\times 3$ matrix, therefore, we could apply Lemma \ref{lem2.4} to each component in above argument, adding them together then get the result.
\end{remark}
Also, using the estimate \eqref{4.2}, we have
\begin{align}
  & \left|\Trucan\int_{\mathbb{T}}D(r_n)\Dv_x\partial_x^\alpha\left(L\nabla_x Q_n\odot\nabla_x Q_n-\frac{L}{2}|\nabla_x Q_n|^2{\rm I}_3\right)\cdot \partial_x^\alpha\mathbf{u}_ndx \right|\nonumber\\
  =& ~\bigg|\Trucan\int_{\mathbb{T}}\nabla_x D(r_n)\partial_x^\alpha\left(L\nabla_x Q_n\odot\nabla_x Q_n-\frac{L}{2}|\nabla_x Q_n|^2{\rm I}_3\right)\cdot \partial_x^\alpha\mathbf{u}_ndx\nonumber\\
  &~+\Trucan\int_{\mathbb{T}}D(r_n)\partial_x^\alpha\left(L\nabla_x Q_n\odot\nabla_x Q_n-\frac{L}{2}|\nabla_x Q_n|^2{\rm I}_3\right): \partial_x^\alpha\nabla_x\mathbf{u}_n^{{\rm T}}dx\bigg|\nonumber\\
  \leq& ~C(R)\Trucan\|\partial_x^{\alpha+1}Q_n\|(\|\partial_x^{\alpha}\mathbf{u}_n\|+\|\partial_x^{\alpha+1}\mathbf{u}_n\|)\nonumber\\
  \leq &~C(R)(\|\partial_x^{\alpha+1}Q_n\|^2+\|\partial_x^{\alpha}\mathbf{u}_n\|^2)+\frac{\upsilon}{8}\Trucan\|\sqrt{D(r_n)}\partial_x^{\alpha+1}\mathbf{u}_n\|^2,
\end{align}
as well as
\begin{align}
   &\left|\Trucan\int_{\mathbb{T}}D(r_n)\Dv_x\partial_x^\alpha\left(\frac{a}{2}{\rm I}_3{\rm tr}(Q_n^2)
        -\frac{b}{3}{\rm I}_3{\rm tr}(Q_n^3)+\frac{c}{4}{\rm I}_3{\rm tr}^2(Q_n^2)\right)\cdot \partial_x^\alpha\mathbf{u}_ndx\right|\nonumber \\
  =&~ \bigg|\Trucan\int_{\mathbb{T}}\nabla_xD(r_n)\partial_x^\alpha\left(\frac{a}{2}{\rm I}_3{\rm tr}(Q_n^2)
        -\frac{b}{3}{\rm I}_3{\rm tr}(Q_n^3)+\frac{c}{4}{\rm I}_3{\rm tr}^2(Q_n^2)\right)\cdot \partial_x^\alpha\mathbf{u}_ndx\nonumber\\
  &~+\Trucan\int_{\mathbb{T}}D(r_n)\partial_x^\alpha\left(\frac{a}{2}{\rm I}_3{\rm tr}(Q_n^2)
        -\frac{b}{3}{\rm I}_3{\rm tr}(Q_n^3)+\frac{c}{4}{\rm I}_3{\rm tr}^2(Q_n^2)\right): \partial_x^\alpha\nabla_x\mathbf{u}_n^{{\rm T}}dx\bigg|\nonumber\\
\leq &~ C\Trucan (\|D(r_n)\|_{1,\infty}+\|D(r_n)\|_{\infty})(1+\|Q_n\|_\infty^3)\|\partial_x^{\alpha}Q_n\|(\|\partial_x^{\alpha}\mathbf{u}_n\|+\|\partial_x^{\alpha+1}\mathbf{u}_n\|)\nonumber\\
\leq &~ C(R)(\|\partial_x^{\alpha}Q_n\|^2+\|\partial_x^{\alpha}\mathbf{u}_n\|^2)+\frac{\upsilon}{8}\Trucan\| \sqrt{D(r_n)}\partial_x^{\alpha+1}\mathbf{u}_n\|^2.
\end{align}

 According to equation \eqref{qnt}(1), we have the estimate of $D(r_n)_t$
\begin{align}\label{4.30}
  \|D(r_n)_t\|_\infty&=\left\|\frac{\rho[r_n]_t}{\rho[r_n]^2}\right\|_\infty\leq C(R)\Trucan\|\Dv_x(\rho[r_n] \mathbf{u}_n)\|_\infty \nonumber\\
  &\leq C(R)\Trucan(\|r_n\|_\infty\|\Dv_x\mathbf{u}_n\|_\infty+\|\nabla_xr_n\|_\infty\|\mathbf{u}_n\|_\infty)\leq C(R).
\end{align}
Considering equation \eqref{22.2},  by Lemma \ref{lem2.1}, we can get the estimate of $(\partial_x^\alpha Q_n)_t$ as follows
\begin{align}\label{4.31}
\|(\partial_x^\alpha Q_n)_t\|&=\bigg \|\Gamma L\triangle \partial_x^\alpha Q_n+\Trucan\bigg[-\partial_x^\alpha(\mathbf{u}_n\cdot\nabla_xQ_n)
+\partial_x^\alpha (\Theta_nQ_n-Q_n\Theta_n)\nonumber\\
&\quad-\Gamma\partial_x^\alpha\left(aQ_n-b\left(Q_n^2-\frac{{\rm I}_3}{3}{\rm tr}(Q_n^2)\right)+cQ_n{\rm tr}(Q_n^2)\right)\bigg]\bigg\|\nonumber\\
&\leq C(R)(\|\partial_x^{\alpha}Q_n\|+\|\partial_x^{\alpha+1}Q_n\|+\|\partial_x^{\alpha}\mathbf{u}_n\|+\|\triangle\partial_x^{\alpha}Q_n\|).
\end{align}
The above two estimates combine with (\ref{4.2}), yielding
\begin{align}
   &\quad \left|\frac{1}{2}\int_{\mathbb{T}}D(r_n)_t|\nabla_x\partial_x^\alpha Q_n|^2dx
        +\int_{\mathbb{T}}\nabla_xD(r_n)(\partial_x^\alpha Q_n)_t:\nabla_x\partial_x^\alpha Q_ndx \right|\nonumber\\
   &\leq C\|D(r_n)_t\|_\infty\|\partial_x^{\alpha+1}Q_n\|^2
   +C\|\nabla_xD(r_n)\|_\infty\|(\partial_x^\alpha Q_n)_t\|\|\partial_x^{\alpha+1}Q_n\|\nonumber\\
   &\leq C(R)\|\partial_x^{\alpha+1}Q_n\|^2
   +C(R)(\|\partial_x^{\alpha+1}Q_n\|+\|\partial_x^{\alpha}\mathbf{u}_n\|+\|\triangle\partial_x^{\alpha}Q_n\|)\|\partial_x^{\alpha+1}Q_n\|\nonumber\\
   &\leq C(R)(\|\partial_x^{\alpha+1}Q_n\|^2+\|\partial_x^{\alpha}\mathbf{u}_n\|^2)
   +\frac{\Gamma L}{8}\|\sqrt{D(r_n)}\triangle\partial_x^{\alpha}Q_n\|^2.
\end{align}
In addition, we also have
\begin{align}\label{4.15*}
    &\Trucan\left|\int_{\mathbb{T}}D(r_n)(\mathbf{u}_n\cdot\nabla_x \partial_x^\alpha Q_n):
    \triangle\partial_x^\alpha Q_ndx\right|\nonumber\\
    \leq & ~\Trucan\|D(r_n)\|_\infty\|\mathbf{u}_n\|_\infty\|\partial_x^{\alpha+1}Q_n\|\|\triangle\partial_x^{\alpha}Q_n\|\nonumber\\
    \leq &~C(R)\|\partial_x^{\alpha+1}Q_n\|^2+\frac{\Gamma L}{8}\|\sqrt{D(r_n)}\triangle\partial_x^{\alpha}Q_n\|^2,
\end{align}
and
\begin{align}\label{4.15}
     &\Trucan \left|\int_{\mathbb{T}}D(r_n)\partial_x^\alpha\left(aQ_n-b\left(Q_n^2-\frac{{\rm I}_3}{3}{\rm tr}(Q_n^2)\right)+cQ_n{\rm tr}(Q_n^2)\right):
    \triangle\partial_x^\alpha Q_ndx\right|\nonumber\\
    \leq &~ \Trucan \|D(r_n)\|_{\infty}\|\triangle\partial_x^\alpha Q_n\|\left\|\partial_x^\alpha\left(aQ_n-b\left(Q_n^2-\frac{{\rm I}_3}{3}{\rm tr}(Q_n^2)\right)+cQ_n{\rm tr}(Q_n^2)\right)\right\|\nonumber\\
    \leq &~ C(R)\Trucan (\|Q_n\|_\infty+\|Q_n\|_\infty^2)\|\partial_x^\alpha Q_n\|\|\triangle\partial_x^{\alpha}Q_n\|\nonumber\\
    \leq &~C(R)\|\partial_x^{\alpha+1}Q_n\|^2+\frac{\Gamma L}{8}\|\sqrt{D(r_n)}\triangle\partial_x^{\alpha}Q_n\|^2,
\end{align}
in the last step, we also use the estimate \eqref{4.2}. Summing all the estimates (\ref{19})-(\ref{4.15}), note that the first term in (\ref{4.11}) was cancelled with the forth integral on the left hand side of (\ref{33.9}), also the second term in (\ref{4.9}) was cancelled with the second term on the left hand side of \eqref{3.1} after matching the constant, we conclude
        \begin{align}\label{33.11}
         &d(\|\partial_x^\alpha r_n(t)\|^2+\|\partial_x^\alpha\mathbf{u}_n(t)\|^2+\|\sqrt{D(r_n)}\nabla_x\partial_x^\alpha Q_n(t)\|^2)\nonumber\\
         &+\Trucan \int_{\mathbb{T}}D(r_n)(\upsilon|\nabla_x\partial_x^\alpha\mathbf{u}_n|^2+(\upsilon+\lambda)
        |\Dv_x\partial_x^\alpha\mathbf{u}_n|^2)dxdt\nonumber\\
        &+\Gamma L\|\sqrt{D(r_n)}\triangle\partial_x^\alpha Q_n\|^2dt\nonumber\\
        \leq&~C(R)(\|\partial_x^{\alpha}r_n\|^2+\|\partial_x^{\alpha}\mathbf{u}_n\|^2+\|\sqrt{D(r_n)}\partial_x^{\alpha+1}Q_n\|^2)
        dt\nonumber\\
        &+\Trucan \int_{\mathbb{T}}\partial_x^\alpha\mathbb{F}(r_n,\mathbf{u}_n)
        \cdot\partial_x^\alpha\mathbf{u}_ndx dW.
        \end{align}

Define the stopping time $\tau_{M}$
\begin{eqnarray*}
\tau_{M}=\inf\left\{t\geq 0; \sup_{\xi\in [0,t]}\| r_n(\xi), \mathbf{u}_n(\xi)\|_{s,2}^2\geq M\right\},
\end{eqnarray*}
if the set is empty, choosing $\tau_M=T$. Then, taking sum for $|\alpha|\leq s$, taking integral with respect to time and sumpremum on interval $[0, t\wedge \tau_M]$, power $p$, finally expectation on both sides of (\ref{33.11}), we arrive at
        \begin{align}\label{4.39}
         &\mathbb{E}\left[\sup_{\xi\in [0, t\wedge \tau_M]}(\|\partial_x^s r_n\|^2+\|\partial_x^s\mathbf{u}_n\|^2+\|\sqrt{D(r_n)}\nabla_x\partial_x^s Q_n\|^2)\right]^p\nonumber\\
         &+\mathbb{E}\left(\int_{0}^{t\wedge \tau_M}\Trucan \int_{\mathbb{T}}D(r_n)(\upsilon|\nabla_x\partial_x^s\mathbf{u}_n|^2+(\upsilon+\lambda)
        |\Dv_x\partial_x^s\mathbf{u}_n|^2)dxd\xi\right)^p\nonumber\\
        &+\mathbb{E}\left(\int_{0}^{t\wedge \tau_M}\Gamma L\|\sqrt{D(r_n)}\triangle\partial_x^s Q_n\|^2d\xi\right)^p\nonumber\\
        \leq&~C\mathbb{E}(\|r_0,\mathbf{u}_0\|_{s,2}^{2}+\|Q_0\|_{s+1,2}^{2})^p\nonumber\\
        &+C\mathbb{E}\left(\int_{0}^{t\wedge \tau_M}\|\partial_x^{s}r_n\|^2+\|\partial_x^{s}\mathbf{u}_n\|^2+\|\sqrt{D(r_n)}\partial_x^{s+1}Q_n\|^2
        d\xi\right)^p\nonumber\\
        &+C\mathbb{E}\left[\sup_{\xi\in [0, t\wedge \tau_M]}\left|\int_{0}^{\xi}\Trucan \int_{\mathbb{T}}\partial_x^s\mathbb{F}(r_n,\mathbf{u}_n)
        \cdot\partial_x^s\mathbf{u}_ndx dW\right|\right]^p.
        \end{align}
Regarding the stochastic integral term, we could apply the Burkholder-Davis-Gundy inequality (\ref{2.4}) and assumption \eqref{2.5}(Remark \ref{rem2.6}), for any $1\leq p <\infty$
        \begin{align}\label{33.13}
        &\quad\mathbb{E}\left[\sup_{\xi\in [0, t\wedge \tau_M]}\left|\int_{0}^{\xi}\Trucan \int_{\mathbb{T}}\partial_x^s
        \mathbb{F}(r_n,\mathbf{u}_n)\cdot\partial_x^s\mathbf{u}_n dx dW\right|\right]^p\nonumber\\
        &\leq C(p)\mathbb{E}\left[\int_{0}^{ t\wedge \tau_M}(\Trucan )^2\|
        \mathbb{F}(r_n,\mathbf{u}_n)\|_{L_2(\mathfrak{U}; W^{s,2}(\mathbb{T}))}^2 \|\mathbf{u}_n\|_{s,2}^2d\xi\right]^{\frac{p}{2}}\nonumber\\
        &\leq C(p)\mathbb{E}\left[\int_{0}^{t\wedge \tau_M}(\Trucan )^2\|r_n,\nabla\mathbf{u}_n\|_{1,\infty}^2 \|r_n,\mathbf{u}_n\|_{s,2}^4d\xi\right]^{\frac{p}{2}}\nonumber\\
        &\leq C(p,R)\mathbb{E}\left[\sup_{\xi\in [0, t\wedge \tau_M]}\|r_n,\mathbf{u}_n\|_{s,2}^2\int_{0}^{t\wedge \tau_M}
        \|r_n,\mathbf{u}_n\|_{s,2}^2d\xi\right]^{\frac{p}{2}}\nonumber\\
        &\leq \frac{1}{2}\mathbb{E}\left[\sup_{\xi\in [0, t\wedge \tau_M]}\|r_n,\mathbf{u}_n\|_{s,2}^{2p}\right]
        +C(p,R)\mathbb{E}\left[\int_{0}^{t\wedge \tau_M}\|r_n,\mathbf{u}_n\|_{s,2}^2d\xi\right]^{p}.
        \end{align}
Combining (\ref{4.39})-(\ref{33.13}), the Gronwall lemma gives
\begin{align}\label{4.19}
&\mathbb{E}\left[\sup_{\xi\in [0, t\wedge \tau_M]}(\|\partial_x^s r_n\|^2+\|\partial_x^s\mathbf{u}_n\|^2+\|\sqrt{D(r_n)}\nabla_x\partial_x^s Q_n\|^2)\right]^p\nonumber\\
         &+\mathbb{E}\left(\int_{0}^{t\wedge \tau_M}\Trucan \int_{\mathbb{T}}D(r_n)(\upsilon|\nabla_x\partial_x^s\mathbf{u}_n|^2+
        (\upsilon+\lambda)|\Dv_x\partial_x^s\mathbf{u}_n|^2)dxd\xi\right)^p\nonumber\\
        &+\mathbb{E}\left(\int_{0}^{t\wedge \tau_M}\Gamma L\|\sqrt{D(r_n)}\triangle\partial_x^s Q_n\|^2d\xi\right)^p\leq C,
\end{align}
where the constant $C$ is independent of $n$, but depends on $(s, p, R, \mathbb{T},T)$ and the initial data. Taking $M\rightarrow \infty$ in (\ref{4.19}), using the fact that $\frac{1}{C(R)}\leq D(r_n)\leq C(R)$ and the monotone convergence theorem, we establish the a priori estimates
\begin{eqnarray}
&&r_n\in L^p(\Omega; L^\infty(0,T;W^{s,2}(\mathbb{T}))),~~ \mathbf{u}_n\in L^p(\Omega; L^\infty(0,T;W^{s,2}(\mathbb{T},\mathbb{R}^3))),\label{4.20}\\
&& Q_n\in L^p(\Omega; L^\infty(0,T;W^{s+1,2}(\mathbb{T},S_0^3))\cap L^2(0,T;W^{s+2,2}(\mathbb{T},S_0^3))),\label{4.21}
\end{eqnarray}
for all $ 1\leq p<\infty$, integer $s> \frac{7}{2}$.

{\bf 4.3 Compactness argument}. Let $\{r_{n}, \mathbf{u}_{n}, Q_n\}_{n\geq 1}$ be the sequence of approximate solution to system (\ref{qnt}) relative to the fixed stochastic basis $(\Omega, \mathcal{F},\{\mathcal{F}_{t}\}_{t\geq0}, \mathbb{P}, W)$ and $\mathcal{F}_{0}$-measurable random variable $(r_0, \mathbf{u}_{0}, Q_{0})$. We define the path space
\begin{equation*}
\mathcal{X}=\mathcal{X}_{r}\times \mathcal{X}_{\mathbf{u}}\times \mathcal{X}_{Q}\times \mathcal{X}_{W},
\end{equation*}
where
\begin{eqnarray*}
&&\mathcal{X}_{r}=C([0,T]; W^{s-1,2}(\mathbb{T})),~\mathcal{X}_{\mathbf{u}}=L^\infty(0,T; W^{s-\varepsilon,2}(\mathbb{T},\mathbb{R}^3)),\\
 &&\mathcal{X}_{Q}=C([0,T]; W^{s,2}(\mathbb{T},S_0^3))\cap L^{2}(0,T;W^{s+1,2}(\mathbb{T},S_0^3)),~~\mathcal{X}_{W}=C([0,T];\mathfrak{U}_0),
\end{eqnarray*}
where $\varepsilon$ is small enough such that integer $s-\varepsilon>\frac{3}{2}+2$.

Define the sequence of probability measures
\begin{equation}\label{measure}
\mu^{n}=\mu^{n}_{r}\otimes\mu^{n}_{\mathbf{u}}\otimes \mu^{n}_{Q}\otimes \mu_{W},
\end{equation}
where $\mu^{n}_{r}(\cdot)=\mathbb{P}\{r_{n}\in \cdot\}$, ~$\mu^{n}_{\mathbf{u}}(\cdot)=\mathbb{P}\{\mathbf{u}_{n}\in \cdot\}$,~~$\mu^{n}_{Q}(\cdot)=\mathbb{P}\{Q_{n}\in \cdot\}$,~$\mu_{W}(\cdot)=\mathbb{P}\{W\in \cdot\}$. We show that the set $\{\mu^{n}\}_{n\geq 1}$ is in fact weakly compact.  According to the Prokhorov theorem, it suffices to show that each set $\{\mu^{n}_{(\cdot)}\}_{n\geq 1}$ is tight on the corresponding path space $\mathcal{X}_{(\cdot)}$.
\begin{lemma}\label{lem3.2}
 The set of the sequence of measures $\{\mu^{n}_\mathbf{u}\}_{n\geq 1}$ is tight on path space $\mathcal{X}_\mathbf{u}$.
\end{lemma}
\begin{proof}
First, we show that for any $\alpha\in [0,\frac{1}{2})$
\begin{eqnarray}\label{4.45}
\mathbb{E}\|\mathbf{u}_n\|_{C^\alpha([0,T];L^2(\mathbb{T}, \mathbb{R}^3))}\leq C,
\end{eqnarray}
where $C$ is independent of $n$.

Decompose $\mathbf{u}_n=X_n+Y_n$, where
\begin{eqnarray*}
&&X_n=X_n(0)+\int_{0}^{t}-\Trucan P_n(\mathbf{u}_n\nabla_x\mathbf{u}_n+r_n\nabla_xr_n)
    +\Trucan P_n(D(r_n)(\mathcal{L}\mathbf{u}_n\\&&~\qquad-\Dv_x(L\nabla_x Q_n\odot \nabla_x Q_n
        -\mathcal{F}(Q_n){\rm I}_3)+L\Dv_x(Q_n\triangle Q_n- \triangle Q_nQ_n)))d\xi,\\
&&Y_n=\int_{0}^{t}\Trucan P_n\mathbb{F}(r_n,\mathbf{u}_n)dW.
\end{eqnarray*}
Using the a priori estimates (\ref{4.20}), (\ref{4.21}) and the H\"{o}lder inequality, we have
\begin{eqnarray*}
\mathbb{E}\|X_n\|_{W^{1,2}(0,T;L^2(\mathbb{T},\mathbb{R}^3))}\leq C,
\end{eqnarray*}
where $C$ is independent of $n$. By the embedding (\ref{2.31}), we obtain the estimate
$$\mathbb{E}\|X_n\|_{C^\alpha([0,T];L^2(\mathbb{T}, \mathbb{R}^3))}\leq C.$$

Note that, for a.s. $\omega$, and for any $\delta'>0$, there exists $t_1,t_2\in [0,T]$ such that
$$\sup_{t,t'\in [0, T], t\neq t'}\frac{\left\|\int_{t}^{t'}fdW\right\|_{L^2}}{|t'-t|^\alpha}\leq \frac{\left\|\int_{t_1}^{t_2}fdW\right\|_{L^2}}{|t_2-t_1|^\alpha}+\delta'.$$
Regarding the stochastic term $Y_n$, using the Burkholder-Davis-Gundy inequality (\ref{2.4}) and assumption \eqref{2.5}, we get
\begin{eqnarray*}
&&\quad\mathbb{E}\left\|\int_{0}^{t}\Trucan P_n\mathbb{F}(r_n,\mathbf{u}_n)dW\right\|_{C^\alpha([0,T];L^2(\mathbb{T}))}\\ &&\leq
\mathbb{E}\left[\sup_{t,t'\in [0,T], t\neq t'}\frac{\left\|\int_{t'}^{t}\Trucan P_n\mathbb{F}(r_n,\mathbf{u}_n)dW\right\|_{L^2}}{|t-t'|^\alpha}\right]\\
&&\leq \frac{\mathbb{E}\left\|\int_{t_1}^{t_2}\Trucan P_n\mathbb{F}(r_n,\mathbf{u}_n)dW\right\|_{L^2}}{|t_2-t_1|^\alpha}+\delta'\\
&&\leq \frac{C\mathbb{E}\left(\int_{t_1}^{t_2}\left\|\Trucan \mathbb{F}(r_n,\mathbf{u}_n)\right\|_{L_2(\mathfrak{U};L^2(\mathbb{T}))}^2d\xi\right)^\frac{1}{2}}{|t_2-t_1|^\alpha}+\delta'\\
&&\leq C(R)|t_2-t_1|^{\frac{1}{2}-\alpha}+\delta'\leq C.
\end{eqnarray*}
Thus, we get the estimate (\ref{4.45}). Fix any $\alpha\in (0,\frac{1}{2})$, by the Aubin-Lions lemma \ref{lem6.1}, we have
\begin{eqnarray*}
C^\alpha([0,T];L^2(\mathbb{T}, \mathbb{R}^3))\cap L^\infty(0,T;W^{s,2}(\mathbb{T},\mathbb{R}^3))\hookrightarrow\hookrightarrow L^\infty(0,T;W^{s-\varepsilon,2}(\mathbb{T},\mathbb{R}^3)).
\end{eqnarray*}
Therefore, for any fixed $K> 0$, the set
\begin{eqnarray*}
&&B_{K}:=\bigg\{\mathbf{u}\in C^\alpha([0,T];L^2(\mathbb{T},\mathbb{R}^3))\cap L^\infty(0,T;W^{s,2}(\mathbb{T}, \mathbb{R}^3)):\\ &&\qquad\qquad\qquad\|\mathbf{u}\|_{C^\alpha([0,T];L^2(\mathbb{T}, \mathbb{R}^3))}+\|\mathbf{u}\|_{L^\infty(0,T;W^{s,2}(\mathbb{T}, \mathbb{R}^3))}\leq K\bigg\}
\end{eqnarray*}
is compact in $L^\infty(0,T;W^{s-\varepsilon,2}(\mathbb{T}, \mathbb{R}^3))$. Applying the Chebyshev inequality and the estimates $(\ref{4.20})_2$, (\ref{4.45}), we have
\begin{align*}
 \mu_{\mathbf{u}}^{n}(B_{K}^{c})&=\mathbb{P}\left(\|\mathbf{u}_{n}\|_{L^\infty(0,T;W^{s,2}(\mathbb{T},\mathbb{R}^3))}
+\|\mathbf{u}_{n}\|_{C^\alpha([0,T];L^2(\mathbb{T}, \mathbb{R}^3))}> K\right)\nonumber\\
&\leq \mathbb{P}\left(\|\mathbf{u}_{n}\|_{L^\infty(0,T;W^{s,2}(\mathbb{T}, \mathbb{R}^3))}> \frac{K}{2}\right)+\mathbb{P}\left(\|\mathbf{u}_{n}\|_{C^\alpha([0,T];L^2(\mathbb{T}, \mathbb{R}^3))}> \frac{K}{2}\right)\nonumber\\
&\leq \frac{2}{K}\left(\mathbb{E}\|\mathbf{u}_{n}\|_{L^\infty(0,T;W^{s,2}(\mathbb{T}, \mathbb{R}^3))}+\mathbb{E}\|\mathbf{u}_{n}\|_{C^\alpha([0,T];L^2(\mathbb{T}, \mathbb{R}^3))}\right)\leq \frac{C}{K},
\end{align*}
where the constant $C$ is independent of $n,K$. Thus, we obtain the tightness of the sequence of measures $\{\mu^{n}_\mathbf{u}\}_{n\geq 1}$.
\end{proof}

\begin{lemma}\label{lem4.5}
 The set of the sequence of measures $\{\mu^{n}_Q\}_{n\geq 1}$ is tight on path space $\mathcal{X}_Q$.
\end{lemma}
\begin{proof} We only need to show that the set $\{\mu^{n}_Q\}_{n\geq 1}$ is tight on space $L^2(0,T;W^{s+1,2}(\mathbb{T},S_0^3))$, the proof of tightness on space $C([0,T];W^{s,2}(\mathbb{T},S_0^3))$ is the same as the proof of the set $\{\mu^{n}_\mathbf{u}\}_{n\geq 1}$.

From the equation (\ref{qnt})(3), we can easily show that
\begin{eqnarray}\label{4.42}
\mathbb{E}\|Q_{n}\|_{W^{1,2}(0,T;L^2(\mathbb{T},S_0^3))}\leq C,
\end{eqnarray}
where $C$ is a constant independence of $n$. For any fixed $K>0$, define the set
\begin{eqnarray*}
 &&\overline{B}_{K}:=\bigg\{Q\in L^{2}(0,T;W^{s+2,2}(\mathbb{T},S_0^3))\cap W^{1,2}(0,T;L^2(\mathbb{T},S_0^3)):\\
 &&\qquad\qquad\qquad \|Q\|_{L^{2}(0,T;W^{s+2,2}(\mathbb{T},S_0^3))}+\|Q\|_{W^{1,2}
 (0,T;L^2(\mathbb{T},S_0^3))}\leq K\bigg\},
\end{eqnarray*}
which is thus compact in $L^{2}(0,T;W^{s+1,2}(\mathbb{T},S_0^3))$ as a result of the compactness embedding
\begin{eqnarray*}
L^{2}(0,T;W^{s+2,2}(\mathbb{T},S_0^3))\cap W^{1,2}(0,T;L^2(\mathbb{T},S_0^3))\hookrightarrow L^{2}(0,T;W^{s+1,2}(\mathbb{T},S_0^3)).
\end{eqnarray*}
Applying the Chebyshev inequality and the estimates (\ref{4.21}), (\ref{4.42}), we get
\begin{eqnarray*}
&& \mu_{u}^{n}\left(\overline{B}_{K}^{c}\right)\leq \frac{C}{K},
\end{eqnarray*}
where the constant $C$ is independent of $n,K$.
\end{proof}

Using the same argument as above, we can show the tightness of the sequences of set $\{\mu^{n}_r\}_{n\geq 1}$. Since the sequence $W$ is only one element and thus, the set $\{\mu^{n}_W\}_{n\geq 1}$ is weakly compact. Then, the tightness of measure set $\{\mu^{n}\}_{n\geq 1}$ follows.

With the weakly compact of set $\{\mu^{n}\}_{n\geq 1}$ in hand, using the Skorokhod representation theorem \ref{thm4.1}, we have:
\begin{proposition}\label{pro5.3}
    There exists a subsequence of $\{\mu^n\}_{n\geq 1}$, also denoted as $\{\mu^n\}_{n\geq 1}$,  and a probability space
    $(\tilde{\Omega},\tilde{\mathcal{F}},\tilde{\mathbb{P}})$ as well as a sequence of random variables
    $(\tilde{r}_n,\tilde{\mathbf{u}}_n, \tilde{Q}_n, \tilde{W}_n)$, $(\tilde{r},\tilde{\mathbf{u}}, \tilde{Q}, \tilde{W})$ such that\\
    (a) the joint law of $(\tilde{r}_n,\tilde{\mathbf{u}}_n, \tilde{Q}_n, \tilde{W}_n)$ is $\mu^n$, and the joint law of
    $(\tilde{r},\tilde{\mathbf{u}}, \tilde{Q}, \tilde{W})$ is $\mu$, where $\mu$ is the weak limit of the sequence $\{\mu^n\}_{n\geq 1}$;\\
    (b) $(\tilde{r}_n,\tilde{\mathbf{u}}_n, \tilde{Q}_n, \tilde{W}_n)$ converges to $(\tilde{r},\tilde{\mathbf{u}}, \tilde{Q}, \tilde{W})$, $\tilde{\mathbb{P}}$ a.s. in the topology of $\mathcal{X}$;\\
    (c) the sequence of $\tilde{Q}_n$ and $\tilde{Q}$ belong to $S_0^3$, almost everywhere.\\
    (d) $\tilde{W}_n$ is a cylindrical Wiener process, relative to the filtration $\tilde{\mathcal{F}}_t^n$ given below.
\end{proposition}
\begin{proof} The results (a), (b), (d) are a direct consequence of the Skorokhod representation theorem. The result (c) is a consequence of result (a).
\end{proof}
\begin{proposition}\label{pro4.7} The sequence $(\tilde{r}_n,\tilde{\mathbf{u}}_n, \tilde{Q}_n, \tilde{W}_n)$ still satisfies the $n$-th order Galerkin approximate system  relative to the stochastic basis $\widetilde{S}^n:=(\tilde{\Omega},\tilde{\mathcal{F}},\tilde{\mathbb{P}}, \{\tilde{\mathcal{F}}_t^n\}_{t\geq 0}, \tilde{W}_n )$, where $\tilde{\mathcal{F}}_t^n$ is a canonical filtration defined by $$\sigma\left(\sigma\left(\tilde{r}_n(s),\tilde{\mathbf{u}}_n(s),\tilde{Q}_n(s),
  \tilde{W}_n(s):s\leq t\right)\cup \left\{\Sigma\in \tilde{\mathcal{F}}; \tilde{\mathbb{P}}(\Sigma)=0\right\}\right).$$
\end{proposition}
\begin{proof} The proof is similar to the one in \cite{Glatt-Holtz,DWang}, here we omit the details.
\end{proof}

{\bf 4.4 Identification of limit}. We verify that $(\widetilde{\mathcal{S}}, \tilde{r},\tilde{\mathbf{u}}, \tilde{Q}, \tilde{W})$ is a strong martingale solution to system \eqref{qnt}, where $\widetilde{\mathcal{S}}:=(\tilde{\Omega},\tilde{\mathcal{F}},\tilde{\mathbb{P}}, \{\tilde{\mathcal{F}}_t\}_{t\geq 0})$ and the canonical filtration $\tilde{\mathcal{F}}_t$ was given by
\begin{align*}
  \tilde{\mathcal{F}}_t&=\sigma\left(\sigma\left(\tilde{r}(s),\tilde{\mathbf{u}}(s),\tilde{Q}(s),
  \tilde{W}(s):s\leq t\right)\cup \left\{\Sigma\in \tilde{\mathcal{F}}; \tilde{\mathbb{P}}(\Sigma)=0\right\}\right).
\end{align*}
Define the following functionals
\begin{align*}
  & \mathcal{P}(r,\mathbf{u})_t=r(t)-r(0)+\int_{0}^{t}\Truca \left(\mathbf{u}\cdot\nabla_xr+\frac{\gamma-1}{2}r\Dv_x\mathbf{u}\right)d\xi,\\
  & \mathcal{N}(Q,\mathbf{u})_t=Q(t)-Q(0)-\int_{0}^{t}\Gamma L \triangle Qd\xi+\int_{0}^{t}\Truca (\mathbf{u}\cdot\nabla_x Q-\Theta Q+Q\Theta-\mathcal{K}(Q))d\xi,\\
  &  \mathcal{M}(r,\mathbf{u},Q)_t=\mathbf{u}(t)-\mathbf{u}(0)
  +\int_{0}^{t}\Truca (\mathbf{u}\cdot\nabla_x \mathbf{u}+r\nabla_x r)d\xi\\
  &\quad-\int_{0}^{t}\Truca (D(r)(\mathcal{L}\mathbf{u}-\Dv_x(L\nabla_x Q\odot \nabla_x Q
  -\mathcal{F}(Q){\rm I}_3)+L\Dv_x(Q\triangle Q-\triangle QQ))d\xi.
\end{align*}
First, we show that for any function $\mathbf{h}\in L^2(\mathbb{T})$, almost every $(\omega, t)\in \tilde{\Omega}\times (0,T]$
\begin{eqnarray*}
\langle \mathcal{P}(\tilde{r}_n,\tilde{\mathbf{u}}_n)_t, \mathbf{h}\rangle \rightarrow \langle \mathcal{P}(\tilde{r},\tilde{\mathbf{u}})_t, \mathbf{h}\rangle,~~
\langle \mathcal{N}(\tilde{Q}_n,\tilde{\mathbf{u}}_n)_t,\mathbf{h}\rangle\rightarrow\langle \mathcal{N}(\tilde{Q},\tilde{\mathbf{u}})_t,\mathbf{h}\rangle,
\end{eqnarray*}
as $n\rightarrow \infty$. We only give the argument of high-order term $Q{\rm tr}(Q^2)$ in $\mathcal{K}(Q)$. Note that
\begin{align*}
&\quad\left|\int_{0}^{t}(\Phi_R^{\tilde{\mathbf{u}}_n, \tilde{Q}_n}\tilde{Q}_n{\rm tr}(\tilde{Q}_n^2)-\Phi_R^{\tilde{\mathbf{u}}, \tilde{Q}}\tilde{Q}{\rm tr}(\tilde{Q}^2), \mathbf{h})d\xi\right|\\
&\leq \left|\int_{0}^{t}((\Phi_R^{\tilde{\mathbf{u}}_n, \tilde{Q}_n}-\Phi_R^{\tilde{\mathbf{u}}, \tilde{Q}})\tilde{Q}_n{\rm tr}(\tilde{Q}_n^2), \mathbf{h})d\xi\right|+\left|\int_{0}^{t}\Phi_R^{\tilde{\mathbf{u}}, \tilde{Q}}(\tilde{Q}_n{\rm tr}(\tilde{Q}_n^2)-\tilde{Q}{\rm tr}(\tilde{Q}^2), \mathbf{h})d\xi\right|\\
&\leq \left|\int_{0}^{t}((\Phi_R^{\tilde{\mathbf{u}}_n, \tilde{Q}_n}-\Phi_R^{\tilde{\mathbf{u}}, \tilde{Q}})\tilde{Q}_n{\rm tr}(\tilde{Q}_n^2), \mathbf{h})d\xi\right|+\left|\int_{0}^{t}\Phi_R^{\tilde{\mathbf{u}},\tilde{ Q}}((\tilde{Q}_n-\tilde{Q}){\rm tr}(\tilde{Q}_n^2), \mathbf{h})d\xi\right|\\
&\quad+\left|\int_{0}^{t}\Phi_R^{\tilde{\mathbf{u}}, \tilde{Q}}(\tilde{Q}({\rm tr}(\tilde{Q}_n^2)-{\rm tr}(\tilde{Q}^2)), \mathbf{h})d\xi\right|\\
&=:J_1+J_2+J_3.
\end{align*}
Using the mean value theorem, the H\"{o}lder inequality and Proposition \ref{pro5.3}(b),  we get
\begin{align*}
J_1&\leq C\|\mathbf{h}\|\int_{0}^{t}(\|\tilde{\mathbf{u}}_n-\tilde{\mathbf{u}}\|_{2,\infty}+\|\tilde{Q}_n-\tilde{Q}\|_{3,\infty})\|\tilde{Q}_n{\rm tr}(\tilde{Q}_n^2)\|_{\infty}d\xi\\
&\leq C\|\mathbf{h}\|\int_{0}^{t}(\|\tilde{\mathbf{u}}_n-\tilde{\mathbf{u}}\|_{s-\varepsilon,2}+\|\tilde{Q}_n-\tilde{Q}\|_{s+1,2})\|\tilde{Q}_n{\rm tr}(\tilde{Q}_n^2)\|_{\infty}d\xi\\
&\leq C\|\mathbf{h}\|\sup_{t\in [0,T]}\|\tilde{Q}_n{\rm tr}(\tilde{Q}_n^2)\|_{\infty}\int_{0}^{t}(\|\tilde{\mathbf{u}}_n-\tilde{\mathbf{u}}\|_{s-\varepsilon,2}+\|\tilde{Q}_n-\tilde{Q}\|_{s+1,2})d\xi\\
&\rightarrow 0,~ {\rm as} ~n\rightarrow\infty,~ \tilde{\mathbb{P}} ~\mbox{a.s.}
\end{align*}
We could use the same argument to get $J_2,J_3\rightarrow 0$, $n\rightarrow\infty,~ \tilde{\mathbb{P}} ~\mbox{a.s.}$. Furthermore, by the Vitali convergence theorem \ref{thm6.1}, we infer that $(\tilde{r},\tilde{\mathbf{u}},\tilde{Q})$ solves equations \eqref{qnt}(1)(3).

 It remains to verify that $(\tilde{r},\tilde{\mathbf{u}},\tilde{Q}, \tilde{W})$ solves equation \eqref{qnt}(2) by passing $n\rightarrow \infty$. With the spirit of \cite{ZM}, we are able to obtain the limit $(\tilde{r},\tilde{\mathbf{u}}, \tilde{Q}, \tilde{W})$ satisfies the equation \eqref{qnt}(2) once we show that the process $\mathcal{M}(\tilde{r},\tilde{\mathbf{u}},\tilde{Q})_t$ is a square integral martingale and its quadratic and cross variations satisfy,
\begin{eqnarray}
&& \langle\langle \mathcal{M}(\tilde{r},\tilde{\mathbf{u}},\tilde{Q})_t\rangle\rangle=\int_{0}^{t}(\Phi_R^{\tilde{\mathbf{u}}, \tilde{Q}})^2\|\mathbb{F}(\tilde{r},\tilde{\mathbf{u}})\|_{L_2(\mathfrak{U};L^2(\mathbb{T},\mathbb{R}^3))}^2d\xi,\label{5.6}\\
&&\langle\langle \mathcal{M}(\tilde{r},\tilde{\mathbf{u}},\tilde{Q})_t, \tilde{\beta}_k\rangle\rangle=\int_{0}^{t}\Phi_R^{\tilde{\mathbf{u}}, \tilde{Q}} \|\mathbb{F}(\tilde{r},\tilde{\mathbf{u}})e_k\|d\xi.\label{5.7}
\end{eqnarray}

We clarify that the $\tilde{\mathcal{F}}_t$-Wiener process $\tilde{W}$ can be written in the form of $\tilde{W}=\sum_{k\geq 1}\tilde{\beta}_ke_k$. Since $\tilde{W}_n$ has the same distribution as $W_n$, then clearly its distribution is the same to $W$. That is, for any $n\in\mathbb{N}$, there exists a collection of mutually independent real-valued
$\tilde{\mathcal{F}}_t^n$-Wiener processes $\{\tilde{\beta}_k^n\}_{k\geq 1}$, such that $\tilde{W}_n=\sum_{k\geq 1}\tilde{\beta}_k^ne_k$. Due to the
convergence property of $\tilde{W}_n$, therefore the same thing holds for $\tilde{W}$.

For any function $\mathbf{h}\in L^2(\mathbb{T},\mathbb{R}^3)$, by Proposition \ref{pro4.7}, we have
\begin{align*}
  \tilde{\mathbb{E}}&\left[h(\mathbf{r}_s\tilde{r}_n,\mathbf{r}_s\tilde{\mathbf{u}}_n,\mathbf{r}_s\tilde{Q}_n,\mathbf{r}_s\tilde{W}_n)
  \langle \mathcal{M}(\tilde{r}_n,\tilde{\mathbf{u}}_n,\tilde{Q}_n)_t
  -\mathcal{M}(\tilde{r}_n,\tilde{\mathbf{u}}_n,\tilde{Q}_n)_s,\mathbf{h}\rangle\right]=0, \\
  \tilde{\mathbb{E}}&\bigg[h(\mathbf{r}_s\tilde{r}_n,\mathbf{r}_s\tilde{\mathbf{u}}_n,\mathbf{r}_s\tilde{Q}_n,\mathbf{r}_s\tilde{W}_n)
  \bigg(\langle \mathcal{M}(\tilde{r}_n,\tilde{\mathbf{u}}_n,\tilde{Q}_n)_t,\mathbf{h}\rangle^2
  -\langle \mathcal{M}(\tilde{r}_n,\tilde{\mathbf{u}}_n,\tilde{Q}_n)_s,\mathbf{h}\rangle^2 \\
  &-\int_{s}^{t}(\Phi_R^{\tilde{\mathbf{u}}_n, \tilde{Q}_n})^2\|(P_n\mathbb{F}(\tilde{r}_n,\tilde{\mathbf{u}}_n))^*\mathbf{h}\|_{\mathfrak{U}}^2d\xi\bigg)\bigg]=0,\\
  \tilde{\mathbb{E}}&\bigg[h(\mathbf{r}_s\tilde{r}_n,\mathbf{r}_s\tilde{\mathbf{u}}_n,\mathbf{r}_s\tilde{Q}_n,\mathbf{r}_s\tilde{W}_n)
  \bigg(\tilde{\beta}^n_k(t)\langle \mathcal{M}(\tilde{r}_n,\tilde{\mathbf{u}}_n,\tilde{Q}_n)_t,\mathbf{h}\rangle
  -\tilde{\beta}^n_k(s)\langle \mathcal{M}(\tilde{r}_n,\tilde{\mathbf{u}}_n,\tilde{Q}_n)_s,\mathbf{h}\rangle \\
  &-\int_{s}^{t}\Phi_R^{\tilde{\mathbf{u}}_n,\tilde{Q}_n}\langle P_n\mathbb{F}(\tilde{r}_n,\tilde{\mathbf{u}}_n)e_k,\mathbf{h}\rangle d\xi\bigg)\bigg]=0,
\end{align*}
where $h$ is a continuous function defined by
$$h:\mathcal{X}_r|_{[0,s]}\times\mathcal{X}_\mathbf{u}|_{[0,s]}\times\mathcal{X}_Q|_{[0,s]}\times\mathcal{X}_W|_{[0,s]}\rightarrow [0,1]$$
and  $\mathbf{r}_t$ is an operator as the restriction of the path
spaces $\mathcal{X}_r$, $\mathcal{X}_{\mathbf{u}}$, $\mathcal{X}_Q$ and $\mathcal{X}_W$ to the interval $[0,t]$ for any $t\in[0,T]$.

In order to pass the limit in above equality, we show that for almost every $(\omega, t)\in \tilde{\Omega}\times (0,T]$
\begin{align}\label{4.55}
(\mathcal{M}(\tilde{r}_n,\tilde{\mathbf{u}}_n,\tilde{Q}_n)_t, \mathbf{h})\rightarrow (\mathcal{M}(\tilde{r},\tilde{\mathbf{u}},\tilde{Q})_t, \mathbf{h}).
\end{align}
We only consider the nontrivial term ${\rm div}_x(Q\triangle Q-\triangle QQ)$. Note that
\begin{align*}
&\quad\int_{0}^{t}\bigg(\Phi_R^{\tilde{\mathbf{u}}_n, \tilde{Q}_n} D(\tilde{r}_n)L\Dv_x(\tilde{Q}_n\triangle \tilde{Q}_n-\triangle \tilde{Q}_n\tilde{Q}_n)-\Phi_R^{\tilde{\mathbf{u}}, \tilde{Q}}  D(\tilde{r})L\Dv_x(\tilde{Q}\triangle \tilde{Q}-\triangle \tilde{Q}\tilde{Q}), \mathbf{h}\bigg)d\xi\\
&\leq \int_{0}^{t}\left((\Phi_R^{\tilde{\mathbf{u}}_n, \tilde{Q}_n}-\Phi_R^{\tilde{\mathbf{u}}, \tilde{Q}})D(\tilde{r}_n)L\Dv_x(\tilde{Q}_n\triangle \tilde{Q}_n-\triangle \tilde{Q}_n\tilde{Q}_n), \mathbf{h}\right)d\xi\\
&\quad+\int_{0}^{t}\left(\Phi_R^{\tilde{\mathbf{u}}, \tilde{Q}} (D(\tilde{r}_n)-D(\tilde{r}))L\Dv_x(\tilde{Q}_n\triangle \tilde{Q}_n-\triangle \tilde{Q}_n\tilde{Q}_n), \mathbf{h}\right)d\xi\\
&\quad+\int_{0}^{t}\left(\Phi_R^{\tilde{\mathbf{u}}, \tilde{Q}} D(\tilde{r})L\Dv_x(\tilde{Q}_n\triangle \tilde{Q}_n-\triangle \tilde{Q}_n\tilde{Q}_n-\triangle \tilde{Q}\tilde{Q}+ \tilde{Q}\triangle\tilde{Q}), \mathbf{h}\right)d\xi\\
&=:K_1+K_2+K_3.
\end{align*}
Using the mean value theorem, the H\"{o}lder inequality, \eqref{4.2} and Proposition \ref{pro5.3}(b), we get
\begin{align*}
K_1&\leq C\|\mathbf{h}\|\int_{0}^{t}(\|\tilde{\mathbf{u}}_n-\tilde{\mathbf{u}}\|_{2,\infty}+\|\tilde{Q}_n-\tilde{Q}\|_{3,\infty})\|D(\tilde{r}_n)\|_{\infty}
\|\tilde{Q}_n\|_{1,\infty}\|\tilde{Q}_n\|_{3,\infty}d\xi\\
&\leq C\|\mathbf{h}\|\sup_{t\in [0,T]}\|\tilde{Q}_n\|_{s+1,2}^2\int_{0}^{t}(\|\tilde{\mathbf{u}}_n-\tilde{\mathbf{u}}\|_{2,\infty}+\|\tilde{Q}_n-\tilde{Q}\|_{3,\infty})d\xi\\
&\rightarrow 0,~ {\rm as} ~n\rightarrow\infty,~ \tilde{\mathbb{P}} ~\mbox{a.s.}
\end{align*}
Similarly, using \eqref{4.4}, the H\"{o}lder inequality and Proposition \ref{pro5.3}(b), we get $K_2\rightarrow 0$, $\tilde{\mathbb{P}}$ ~\mbox{a.s.}.
Using the H\"{o}lder inequality, \eqref{4.2} and Proposition \ref{pro5.3}(b), we also get $K_3\rightarrow 0$, $\tilde{\mathbb{P}}$ ~\mbox{a.s.}.

Last, let $n\rightarrow \infty$, by \eqref{4.55} and the Vitali convergence theorem \ref{thm6.1}, we could find
\begin{align*}
  \tilde{\mathbb{E}}&\left[h(\mathbf{r}_s\tilde{r},\mathbf{r}_s\tilde{\mathbf{u}},\mathbf{r}_s\tilde{Q},\mathbf{r}_s\tilde{W})
  \langle \mathcal{M}(\tilde{r},\tilde{\mathbf{u}},\tilde{Q})_t
  -\mathcal{M}(\tilde{r},\tilde{\mathbf{u}},\tilde{Q})_s,\mathbf{h}\rangle\right]=0, \\
  \tilde{\mathbb{E}}&\bigg[h(\mathbf{r}_s\tilde{r},\mathbf{r}_s\tilde{\mathbf{u}},\mathbf{r}_s\tilde{Q},\mathbf{r}_s\tilde{W})
  \bigg(\langle \mathcal{M}(\tilde{r},\tilde{\mathbf{u}},\tilde{Q})_t,\mathbf{h}\rangle^2
  -\langle \mathcal{M}(\tilde{r},\tilde{\mathbf{u}},\tilde{Q})_s,\mathbf{h}\rangle^2 \\
  &-\int_{s}^{t}(\Phi_R^{\tilde{\mathbf{u}}, \tilde{Q}})^2\|(\mathbb{F}(\tilde{r},\tilde{\mathbf{u}}))^*\mathbf{h}\|_{\mathfrak{U}}^2d\xi\bigg)\bigg]=0,\\
  \tilde{\mathbb{E}}&\bigg[h(\mathbf{r}_s\tilde{r},\mathbf{r}_s\tilde{\mathbf{u}},\mathbf{r}_s\tilde{Q},\mathbf{r}_s\tilde{W})
 \bigg(\tilde{\beta}_k(t)\langle \mathcal{M}(\tilde{r},\tilde{\mathbf{u}},\tilde{Q})_t,\mathbf{h}\rangle
  -\tilde{\beta}_k(s)\langle \mathcal{M}(\tilde{r},\tilde{\mathbf{u}},\tilde{Q})_s,\mathbf{h}\rangle \\
  &-\int_{s}^{t}\Phi_R^{\tilde{\mathbf{u}}, \tilde{Q}}\langle \mathbb{F}(\tilde{r},\tilde{\mathbf{u}})e_k,\mathbf{h}\rangle d\xi\bigg)\bigg]=0.
\end{align*}
Thus, we obtain the desired equalities (\ref{5.6}) and (\ref{5.7}), the Definition \ref{def3.1}(4) follows.

From the estimate $\eqref{4.20}_1$ and the mass equation itself, we are able to deduce that the process $r$ is continuous with respect to time $t$ in $W^{s,2}(\mathbb{T})$ using the \cite[Theorem 3.1]{KR}, see also \cite{breit} for the compressible Navier-Stokes equations. Moreover, by the initial data condition and estimate \eqref{2.1}, we infer the process $r$ has the uniform lower bound which depends on $R$, $\tilde{\mathbb{P}}$ a.s..  Since the high-order terms ${\rm div}_x (Q\triangle Q-\triangle QQ)$ and $\Theta Q-Q\Theta$ arise in momentum and $Q$-tensor equations, again by \cite[Theorem 3.1]{KR} and the estimates $\eqref{4.20}, \eqref{4.21}$ and the equations itself, we could only infer that $(\mathbf{u}, Q)$ is continuous with respect to time $t$ in $W^{s-1,2}(\mathbb{T}, \mathbb{R}^3)\times W^{s,2}(\mathbb{T}, S_0^3)$. This completes the proof of Theorem \ref{th3.1}.
\section{\bf Existence and Uniqueness of Strong Pathwise Solution to Truncated System}

In this section, we establish the existence and uniqueness of strong pathwise solution to system (\ref{qnt}) and start with the definition and result.
\begin{definition}\label{def2} (Strong pathwise solution) Let $(\Omega,\mathcal{F},\{\mathcal{F}_t\}_{t\geq 0},\mathbb{P})$ be a fixed stochastic basis and $W$ be a given cylindrical Wiener process. The triple $(r, \mathbf{u}, Q)$ is called a global strong pathwise solution to system (\ref{qnt}) with initial data $(r_0, \mathbf{u}_0, Q_0)$ if the following conditions hold
\begin{enumerate}
  \item $r$, $\mathbf{u}$ are $\mathcal{F}_t$-progressively measurable processes with values in $W^{s,2}(\mathbb{T}), W^{s,2}(\mathbb{T}, \mathbb{R}^3)$, $Q$ is $\mathcal{F}_t$-progressively measurable process with value in $W^{s+1,2}(\mathbb{T},S_0^3)$, satisfying
      \begin{align*}
        &\qquad r\in L^2(\Omega;C([0,T];W^{s,2}(\mathbb{T}))),~r>0,~ \mathbb{P}\mbox{ a.s.}\\
        &\qquad \mathbf{u}\in L^2(\Omega;L^\infty(0,T;W^{s,2}(\mathbb{T};\mathbb{R}^3))\cap C([0,T];W^{s-1,2}(\mathbb{T};\mathbb{R}^3))),\\
        &\qquad Q\in L^2(\Omega;L^\infty(0,T;W^{s+1,2}(\mathbb{T};S_0^3))
        \cap L^2(0,T;W^{s+2,2}(\mathbb{T};S_0^3))\cap C([0,T];W^{s,2}(\mathbb{T};S_0^3)));
      \end{align*}

  \item for all $t\in[0,T]$, $\mathbb{P}$ a.s.
    \begin{align*}
      r(t) &=r_0-\int_{0}^{t} \Truca \left(\mathbf{u}\cdot\nabla_xr+\frac{\gamma-1}{2}r\Dv_x\mathbf{u}\right)d\xi,\\
      \mathbf{u}(t)&=\mathbf{u}_0-\int_{0}^{t}\Truca (\mathbf{u}\cdot\nabla_x \mathbf{u}+r\nabla_x r)d\xi\\
    &\quad+\int_{0}^{t}\Truca D(r)(\mathcal{L}\mathbf{u}-\Dv_x(L\nabla_x Q\odot \nabla_x Q
    -\mathcal{F}(Q){\rm I}_3)\\
    &\quad+L\Dv_x(Q\triangle Q-\triangle QQ))d\xi+\int_{0}^{t}\Truca \mathbb{F}(r,\mathbf{u})dW,\\
    Q(t)&=Q_0-\int_{0}^{t}\Truca (\mathbf{u}\cdot\nabla_x Q-\Theta Q+Q\Theta) d\xi
        +\int_{0}^{t}\Gamma L\triangle Q+\Truca \mathcal{K}(Q)d\xi.
    \end{align*}
\end{enumerate}
\end{definition}

In this section, we shall obtain the following result.
\begin{theorem}\label{th4.2}
Assume the initial data $(r_0,\mathbf{u}_0,Q_0)$ satisfies the same conditions with Theorem \ref{th3.1} and the coefficient $\mathbb{G}$ satisfies the assumptions \eqref{2.5},\eqref{2.5*}. For any integer $s>\frac{9}{2}$, the system (\ref{qnt}) has a unique global strong pathwise solution in the sense of Definition \ref{def2}.
\end{theorem}

Following the Yamada-Watanabe argument, the pathwise uniqueness in probability "1" in turn reveals that the solution is also strong in probability sense, this means the solution is constructed with respect to the fixed probability space in advance. Therefore, we next establish the pathwise uniqueness.
\begin{proposition}\label{pro3.3}{\rm (Uniqueness)} Fix any integer $s>\frac{9}{2}$. Suppose that $\mathbb{G}$ satisfies assumption \eqref{2.5*}, and  $((\mathcal{S},r_1, \mathbf{u}_{1}, Q_1)$, $(\mathcal{S},r_2, \mathbf{u}_{2}, Q_{2}))$ are two martingale solutions of system (\ref{qnt}) with the same stochastic basis $\mathcal{S}:=(\Omega,\mathcal{F},\{\mathcal{F}_{t}\}_{t\geq 0},\mathbb{P}, W)$. Then if
$$\mathbb{P}\{(r_1(0), \mathbf{u}_{1}(0), Q_{1}(0))=(r_2(0), \mathbf{u}_{2}(0),Q_{2}(0))\}=1,$$
then pathwise uniqueness holds in the sense of Definition \ref{de1}.
\end{proposition}
\begin{proof}
Owing to the complexity of constitution and the similarity of argument with the a priori estimate, here we only focus on the estimate of high-order nonlinearity term. Let $\alpha$ be any vector such that $|\alpha|\leq s-1$, taking the difference of $r_1$ and $r_2$, then $\alpha$-order derivative, we have
\begin{align}\label{unr}
     &d \partial_x^{\alpha}(r_1-r_2)\nonumber\\
     =&-\Trucaa \partial_x^{\alpha}\left(\mathbf{u}_1\cdot\nabla_xr_1+\frac{\gamma-1}{2}r_1\Dv_x\mathbf{u}_1\right)dt\nonumber\\
     &+\Trucab \partial_x^{\alpha}\left(\mathbf{u}_2\cdot\nabla_xr_2+\frac{\gamma-1}{2}r_2\Dv_x\mathbf{u}_2\right)dt\nonumber\\
     =&-\left(\Trucaa-\Trucab\right)\partial_x^{\alpha}\left(\mathbf{u}_1\cdot\nabla_xr_1+\frac{\gamma-1}{2}r_1\Dv_x\mathbf{u}_1\right)dt\nonumber\\
      &-\Trucab\partial_x^{\alpha}\bigg(\mathbf{u}_2\cdot\nabla_x(r_1-r_2)+(\mathbf{u}_1-\mathbf{u}_2)\cdot\nabla_xr_1\nonumber\\
      &\qquad\qquad\quad+\frac{\gamma-1}{2}(r_1-r_2)\Dv_x\mathbf{u}_1+\frac{\gamma-1}{2}r_2\Dv_x(\mathbf{u}_1-\mathbf{u}_2)\bigg)dt.
     \end{align}
Multiplying \eqref{unr} by $\partial_x^{\alpha}(r_1-r_2)$ and integrating over $\mathbb{T}$, then the highest order term can be treated as follows
\begin{align*}
  &\quad-\Trucab\int_{\mathbb{T}}\left(\mathbf{u}_2\cdot\nabla_x\partial_x^{\alpha}(r_1-r_2)
  +\frac{\gamma-1}{2}r_2\Dv_x\partial_x^{\alpha}(\mathbf{u}_1-\mathbf{u}_2)\right)\cdot\partial_x^{\alpha}(r_1-r_2)dx\\
  &=\frac{1}{2}\Trucab\int_{\mathbb{T}}\Dv_x\mathbf{u}_2|\partial_x^{\alpha}(r_1-r_2)|^2dx-\frac{\gamma-1}{2}\Trucab
  \int_{\mathbb{T}}r_2\Dv_x\partial_x^{\alpha}(\mathbf{u}_1-\mathbf{u}_2)\cdot\partial_x^{\alpha}(r_1-r_2)dx.
\end{align*}
From the smoothness of $\Phi$, the mean value theorem and the Sobolev embedding, we have for $s>\frac{3}{2}+3$
\begin{align}\label{5.2}
\left|\Trucaa-\Trucab\right|&\leq C(\|\mathbf{u}_1-\mathbf{u}_2\|_{2,\infty}+\|Q_1-Q_2\|_{3,\infty})\nonumber\\ &\leq
C(\|\mathbf{u}_1-\mathbf{u}_2\|_{s-1,2}+\|Q_1-Q_2\|_{s,2}).
\end{align}
Thus we get from above estimates
\begin{align}\label{u1}
  &\quad\frac{1}{2}d\|\partial_x^{\alpha}(r_1-r_2)\|^2\nonumber\\
  &\leq C(R)\left(1+\sum_{j=1}^{2}\|r_j, \mathbf{u}_j\|_{s,2}^2\right)(\|r_1-r_2, \mathbf{u}_1-\mathbf{u}_2\|_{s-1,2}^2+\|Q_1-Q_2\|_{s,2}^2)dt\nonumber\\
  &\quad-\frac{\gamma-1}{2}\Trucab\int_{\mathbb{T}}r_2\Dv_x\partial_x^{\alpha}(\mathbf{u}_1-\mathbf{u}_2)\cdot\partial_x^{\alpha}(r_1-r_2)dx dt.
\end{align}

Similarly, for $\mathbf{u}_1$ and $\mathbf{u}_2$, we have the equation
    \begin{align}\label{unv}
    &d\partial_x^{\alpha}(\mathbf{u}_1-\mathbf{u}_2)\nonumber\\
     =&-\Trucaa \partial_x^{\alpha}\left(\mathbf{u}_1\cdot\nabla_x \mathbf{u}_1+r_1\nabla_x r_1-D(r_1)\mathcal{L}\mathbf{u}_1\right)dt\nonumber\\
     &+\Trucab \partial_x^{\alpha}\left(\mathbf{u}_2\cdot\nabla_x \mathbf{u}_2+r_2\nabla_x r_2-D(r_2)\mathcal{L}\mathbf{u}_2\right)dt\nonumber\\
     &-\Trucaa\partial_x^{\alpha}\left(D(r_1)\Dv_x(L\nabla_x Q_1\odot \nabla_xQ_1-\mathcal{F}(Q_1){\rm I}_3-L(Q_1\triangle Q_1-\triangle Q_1Q_1))\right)dt\nonumber\\
     &+\Trucab\partial_x^{\alpha}\left(D(r_2)\Dv_x(L\nabla_x Q_2\odot \nabla_xQ_2-\mathcal{F}(Q_2){\rm I}_3-L(Q_2\triangle Q_2-\triangle Q_2Q_2))\right)dt\nonumber\\
     &+\Trucaa \partial_x^{\alpha}\mathbb{F}(r_1,\mathbf{u}_1)dW-\Trucab \partial_x^{\alpha}\mathbb{F}(r_2,\mathbf{u}_2)dW\nonumber\\
     =&-\left(\Trucaa-\Trucab\right) \partial_x^{\alpha}\big(\mathbf{u}_1\cdot\nabla_x \mathbf{u}_1+r_1\nabla_x r_1-D(r_1)\mathcal{L}\mathbf{u}_1\big)dt\nonumber\\
     &-\Trucab \partial_x^{\alpha}\big((\mathbf{u}_1-\mathbf{u}_2)\cdot\nabla_x \mathbf{u}_1+\mathbf{u}_1\cdot\nabla_x(\mathbf{u}_1-\mathbf{u}_2)
     +(r_1-r_2)\nabla_x r_1+r_2\nabla_x (r_1-r_2)\nonumber\\
     &\qquad\qquad-(D(r_1)-D(r_2))\mathcal{L}\mathbf{u}_1-D(r_2)\mathcal{L}(\mathbf{u}_1-\mathbf{u}_2)\big)dt\nonumber\\
     &-\left(\Trucaa-\Trucab\right)\partial_x^{\alpha}\bigg(D(r_1)\Dv_x(L\nabla_x Q_1\odot \nabla_xQ_1-\mathcal{F}(Q_1){\rm I}_3\nonumber\\
     &\qquad\qquad\qquad\qquad\quad-L(Q_1\triangle Q_1-\triangle Q_1Q_1))\bigg)dt\nonumber\\
     &-\Trucab\partial_x^{\alpha}\bigg(\left(D(r_1)-D(r_2)\right)\Dv_x(L\nabla_x Q_1\odot \nabla_xQ_1-\mathcal{F}(Q_1){\rm I}_3\nonumber\\
     &\qquad\qquad-L(Q_1\triangle Q_1-\triangle Q_1Q_1))\bigg)dt\nonumber\\
     &-\Trucab\partial_x^{\alpha}\bigg(D(r_2)\Dv_x(L\nabla_x(Q_1-Q_2)\odot\nabla_xQ_1+L\nabla_x Q_2\odot\nabla_x(Q_1-Q_2)\nonumber\\
     &\qquad\qquad-(\mathcal{F}(Q_1)-\mathcal{F}(Q_2)){\rm I}_3)\bigg)dt\nonumber\\
     &+\Trucab\partial_x^{\alpha}\bigg(D(r_2)\Dv_xL(Q_1\triangle(Q_1-Q_2)-\triangle(Q_1-Q_2) Q_1+(Q_1-Q_2)\triangle Q_2\nonumber\\
     &\qquad\qquad+\triangle Q_2 (Q_1-Q_2))\bigg)dt\nonumber\\
     &+\left(\Trucaa\partial_x^{\alpha}\mathbb{F}(r_1,\mathbf{u}_1)-\Trucab \partial_x^{\alpha}\mathbb{F}(r_2,\mathbf{u}_2)\right)dW.
    \end{align}
Applying the It\^o formula to function $\frac{1}{2}\|\partial_x^{\alpha}(\mathbf{u}_1-\mathbf{u}_2)\|^2$, then the high-order term in the formula reads
\begin{align*}
  &\quad-\Trucab \int_{\mathbb{T}}r_2\nabla_x\partial_x^{\alpha}(r_1-r_2)\cdot\partial_x^{\alpha}(\mathbf{u}_1-\mathbf{u}_2)dx \\
  &= \Trucab \int_{\mathbb{T}}r_2\Dv_x\partial_x^{\alpha}(\mathbf{u}_1-\mathbf{u}_2)\cdot\partial_x^{\alpha}(r_1-r_2)dx+
  \Trucab \int_{\mathbb{T}}\nabla_xr_2\partial_x^{\alpha}(\mathbf{u}_1-\mathbf{u}_2)\partial_x^{\alpha}(r_1-r_2)dx\\
  &\leq C(R)\|r_1-r_2, \mathbf{u}_1-\mathbf{u}_2\|_{s-1,2}^2+\Trucab \int_{\mathbb{T}}r_2\Dv_x\partial_x^{\alpha}(\mathbf{u}_1-\mathbf{u}_2)\cdot\partial_x^{\alpha}(r_1-r_2)dx.
\end{align*}
The last integral in above could be cancelled with the last term in \eqref{u1} after matching the constant. What's more, integration by parts and the H\"{o}lder inequality give
\begin{align*}
   &\quad \Trucab \int_{\mathbb{T}}D(r_2)\mathcal{L}(\partial_x^{\alpha}(\mathbf{u}_1-\mathbf{u}_2))
   \cdot\partial_x^{\alpha}(\mathbf{u}_1-\mathbf{u}_2)dx\\
&=-\Trucab\int_{\mathbb{T}}D(r_2)(\upsilon|\nabla_x\partial_x^{\alpha}(\mathbf{u}_1-\mathbf{u}_2)|^2
+(\upsilon+\lambda)|\Dv_x\partial_x^{\alpha}(\mathbf{u}_1-\mathbf{u}_2)|^2)dx\\
 &\quad-\upsilon\Trucab\int_{\mathbb{T}}\nabla_xD(r_2)\nabla_x\partial_x^{\alpha}(\mathbf{u}_1-\mathbf{u}_2)
  \cdot\partial_x^{\alpha}(\mathbf{u}_1-\mathbf{u}_2)dx\\
 &\quad-(\upsilon+\lambda)\Trucab\int_{\mathbb{T}}\nabla_xD(r_2)\Dv_x\partial_x^{\alpha}(\mathbf{u}_1-\mathbf{u}_2)
  \cdot\partial_x^{\alpha}(\mathbf{u}_1-\mathbf{u}_2)dx\\
&\leq C(R)\|\mathbf{u}_1-\mathbf{u}_2\|_{s-1,2}^2\\&\quad-\Trucab\int_{\mathbb{T}}D(r_2)(\upsilon|\nabla_x\partial_x^{\alpha}(\mathbf{u}_1-\mathbf{u}_2)|^2
+(\upsilon+\lambda)|\Dv_x\partial_x^{\alpha}(\mathbf{u}_1-\mathbf{u}_2)|^2)dx\\
&\quad+\frac{1}{4}\Trucab \int_{\mathbb{T}}D(r_2)(\upsilon|\nabla_x\partial_x^{\alpha}(\mathbf{u}_1-\mathbf{u}_2)|^2+(\upsilon+\lambda)|\Dv_x\partial_x^{\alpha}(\mathbf{u}_1-\mathbf{u}_2)|^2)dx,
\end{align*}
as well as using Lemma \ref{lem2.4}
\begin{align*}
  &\Trucab\int_{\mathbb{T}}D(r_2)L\Dv_x(Q_1\triangle \partial_x^{\alpha}(Q_1-Q_2)-\triangle \partial_x^{\alpha}(Q_1-Q_2)Q_1\\
  &\qquad\quad+(Q_1-Q_2)\triangle \partial_x^{\alpha}Q_2-\triangle \partial_x^{\alpha}Q_2 (Q_1-Q_2))
  \cdot \partial_x^{\alpha}(\mathbf{u}_1-\mathbf{u}_2)dx\\
  =&-\Trucab\int_{\mathbb{T}}D(r_2)L(Q_1\triangle \partial_x^{\alpha}(Q_1-Q_2)-\triangle \partial_x^{\alpha}(Q_1-Q_2) Q_1): \partial_x^{\alpha}\nabla_x(\mathbf{u}_1-\mathbf{u}_2)^{{\rm T}}dx\\
  &-\Trucab\int_{\mathbb{T}}\nabla_xD(r_2)L(Q_1\triangle \partial_x^{\alpha}(Q_1-Q_2)-\triangle \partial_x^{\alpha}(Q_1-Q_2) Q_1)\cdot
  \partial_x^{\alpha}(\mathbf{u}_1-\mathbf{u}_2)dx\\
  &+\Trucab\int_{\mathbb{T}}D(r_2)L\Dv_x((Q_1-Q_2)\triangle \partial_x^{\alpha}Q_2-\triangle \partial_x^{\alpha}Q_2 (Q_1-Q_2))\cdot
   \partial_x^{\alpha}(\mathbf{u}_1-\mathbf{u}_2)dx\\
  \leq& -\Trucab\int_{\mathbb{T}}D(r_2)L(\partial_x^{\alpha}(\Theta_1-\Theta_2)Q_1-Q_1\partial_x^{\alpha}(\Theta_1-\Theta_2))
  :\triangle \partial_x^{\alpha}(Q_1-Q_2)dx\\
  &+C(R)\left(\sum_{j=1}^{2}\|Q_j\|_{s+2,2}^2\right)(\|\mathbf{u}_1-\mathbf{u}_2\|_{s-1,2}^2+\|Q_1-Q_2\|_{s,2}^2)\\
  &+\frac{\Gamma L}{8}\int_{\mathbb{T}}D(r_2)|\triangle \partial_x^\alpha(Q_1-Q_2)|^2dx.
\end{align*}
By Lemma \ref{lem2.1}, estimate (\ref{4.4}), we have
\begin{align*}
&\quad\int_{\mathbb{T}}\Trucab\partial_x^{\alpha}\big(\left(D(r_1)-D(r_2)\right)L\Dv_x(Q_1\triangle Q_1-\triangle Q_1Q_1)\big)\cdot \partial_x^{\alpha}(\mathbf{u}_1-\mathbf{u}_2)dx\nonumber\\
&\leq \Trucab \left\|\partial_x^{\alpha}\big(\left(D(r_1)-D(r_2)\right)L\Dv_x(Q_1\triangle Q_1-\triangle Q_1Q_1)\big)\right\|\|\partial_x^{\alpha}(\mathbf{u}_1-\mathbf{u}_2)\|\nonumber\\
&\leq C\Trucab (\|r_1,r_2\|_{s,2}\|Q_1\|_{s,2}^2\|r_1-r_2\|_{s-1,2}+\|Q_1\|_{s,2}\|Q_1\|_{s+2,2}\|r_1,r_2\|_{s,2}\|r_1-r_2\|_{s-1,2})\nonumber\\
&\qquad\qquad\quad\times\|\partial_x^{\alpha}(\mathbf{u}_1-\mathbf{u}_2)\|\nonumber\\
&\leq C(\|r_1,r_2\|_{s,2}\|Q_1\|_{s,2}^2+\|Q_1\|_{s,2}\|Q_1\|_{s+2,2}\|r_1,r_2\|_{s,2})\|r_1-r_2, \mathbf{u}_1-\mathbf{u}_2\|_{s-1,2}^2.
\end{align*}
Finally, by Lemma \ref{lem2.1} and the H\"{o}lder inequality
\begin{align*}
&\quad\int_{\mathbb{T}}\Trucab\partial_x^{\alpha}\big(D(r_2)\Dv_x({\rm tr^2}(Q_1^2){\rm I}_3-{\rm tr^2}(Q_2^2){\rm I}_3)\big)\cdot\partial_x^{\alpha}(\mathbf{u}_1-\mathbf{u}_2)dx\nonumber\\
&\leq \Trucab\|\partial_x^{\alpha}(\mathbf{u}_1-\mathbf{u}_2)\|\left\|\partial_x^{\alpha}\big(D(r_2)\Dv_x({\rm tr^2}(Q_1^2){\rm I}_3-{\rm tr^2}(Q_2^2){\rm I}_3)\big)\right\|\nonumber\\
&\leq \Trucab\|\mathbf{u}_1-\mathbf{u}_2\|_{s-1,2}\nonumber\\
&\quad\times(\|D(r_2)\|_{\infty}\|{\rm tr^2}(Q_1^2)-{\rm tr^2}(Q_2^2)\|_{s,2}
+\|D(r_2)\|_{s,2}\|{\rm tr^2}(Q_1^2)-{\rm tr^2}(Q_2^2)\|_{1,\infty})\nonumber\\
&\leq C(R)(1+\|Q_1\|^3_{1,\infty})\|r_2\|_{s,2}\|\mathbf{u}_1-\mathbf{u}_2\|_{s-1,2}\|Q_1-Q_2\|_{s,2}\nonumber\\
&\quad+C(R)(1+\|Q_1, Q_2\|_{s,2}^3)\|\mathbf{u}_1-\mathbf{u}_2\|_{s-1,2}\|Q_1-Q_2\|_{s,2}.
\end{align*}

Since the order of the rest of nonlinearity terms is lower than above, these terms can be handled using the same way, so we skip the details. In summary, we could get
\begin{align}\label{u2}
  &\quad\frac{1}{2}d\|\partial_x^{\alpha}(\mathbf{u}_1-\mathbf{u}_2)\|^2\nonumber\\ &\quad+\Trucab\int_{\mathbb{T}}D(r_2)(\upsilon|\nabla_x\partial_x^{\alpha}(\mathbf{u}_1-\mathbf{u}_2)|^2
+(\upsilon+\lambda)|\Dv_x\partial_x^{\alpha}(\mathbf{u}_1-\mathbf{u}_2)|^2)dxdt\nonumber\\
  &\leq C(R)\sum_{j=1}^{2}(1+\|r_j, u_j\|_{s,2}^2+\|Q_j\|^3_{s+1,2})(1+\|\mathbf{u}_j\|_{s+1,2}^2+\|Q_j\|_{s+2,2}^2)\nonumber\\
     &\quad\times(\|r_1-r_2, \mathbf{u}_1-\mathbf{u}_2\|_{s-1,2}^2+\|Q_1-Q_2\|_{s,2}^2)dt\nonumber\\
     &\quad+\Trucab \int_{\mathbb{T}}r_2\Dv_x\partial_x^{\alpha}(\mathbf{u}_1-\mathbf{u}_2)\cdot\partial_x^{\alpha}(r_1-r_2)dxdt
     +\frac{\Gamma L}{2}\int_{\mathbb{T}}D(r_2)|\triangle \partial_x^\alpha(Q_1-Q_2)|^2dxdt\nonumber\\
     &\quad-\Trucab\int_{\mathbb{T}}D(r_2)L(\partial_x^{\alpha}(\Theta_1-\Theta_2)Q_1-Q_1\partial_x^{\alpha}(\Theta_1-\Theta_2))
  :\triangle \partial_x^{\alpha}(Q_1-Q_2)dxdt\nonumber\\
     &\quad+\left(\Trucaa\partial_x^{\alpha}\mathbb{F}(r_1,\mathbf{u}_1)-\Trucab \partial_x^{\alpha}\mathbb{F}(r_2,\mathbf{u}_2), \partial_x^{\alpha}(\mathbf{u}_1-\mathbf{u}_2)\right)dW\nonumber\\
     &\quad+\frac{1}{2}\left\|\Trucaa\partial_x^{\alpha}\mathbb{F}(r_1,\mathbf{u}_1)
     -\Trucab\partial_x^{\alpha}\mathbb{F}(r_2,\mathbf{u}_2)\right\|_{L_2(\mathfrak{U};L^2(\mathbb{T},\mathbb{R}^3))}^2dt.
\end{align}

The second and forth terms on the right side of \eqref{u2} could be cancelled later. By assumptions \eqref{2.5},\eqref{2.5*}, we could handle
\begin{eqnarray*}
&&\quad\left\|\Trucaa\partial_x^{\alpha}\mathbb{F}(r_1,\mathbf{u}_1)-\Trucab\partial_x^{\alpha}\mathbb{F}(r_2,\mathbf{u}_2)\right\|_{L_2(\mathfrak{U};L^2(\mathbb{T},\mathbb{R}^3))}^2\\
&&\leq \left\|\left(\Trucaa-\Trucab\right)\partial_x^{\alpha}\mathbb{F}(r_1,\mathbf{u}_1)\right\|_{L_2(\mathfrak{U};L^2(\mathbb{T},\mathbb{R}^3))}^2\\&&\quad+\left\|\Trucab \partial_x^{\alpha}(\mathbb{F}(r_1,\mathbf{u}_1)-\mathbb{F}(r_2,\mathbf{u}_2))\right\|_{L_2(\mathfrak{U};L^2(\mathbb{T},\mathbb{R}^3))}^2\\
&&\leq C(R)\sum_{i=1}^{2}(1+\|r_i,\mathbf{u}_i\|_{s}^2)(\|r_1-r_2,\mathbf{u}_1-\mathbf{u}_2\|_{s-1,2}^2+\|Q_1-Q_2\|_{s,2}^2).
\end{eqnarray*}

For the $Q$-tensor equation, we also get
    \begin{align}\label{unq}
    &d\partial_x^{\alpha}(Q_1-Q_2)-\Gamma L\triangle \partial_x^{\alpha}(Q_1-Q_2)dt\nonumber\\
    =&-\Trucaa \partial_x^{\alpha}(\mathbf{u}_1\cdot\nabla_x Q_1-\Theta_1 Q_1+Q_1\Theta_1-\mathcal{K}(Q_1)) dt\nonumber\\
        &+\Trucab \partial_x^{\alpha}(\mathbf{u}_2\cdot\nabla_x Q_2-\Theta_2 Q_2+Q_2\Theta_2-\mathcal{K}(Q_2)) dt\nonumber\\
    =&-\left(\Trucaa-\Trucab\right) \partial_x^{\alpha}(\mathbf{u}_1\cdot\nabla_x Q_1-\Theta_1 Q_1+Q_1\Theta_1-\mathcal{K}(Q_1)) dt\nonumber\\
    &-\Trucab\partial_x^{\alpha}((\mathbf{u}_1-\mathbf{u}_2)\cdot\nabla_x Q_1+\mathbf{u}_2\cdot\nabla_x (Q_1-Q_2))dt\nonumber\\
    &-\Trucab\partial_x^{\alpha}((\Theta_1-\Theta_2)Q_1-Q_1(\Theta_1-\Theta_2)+\Theta_2(Q_1-Q_2)-(Q_1-Q_2)\Theta_2\nonumber\\
    &\qquad\qquad+\mathcal{K}(Q_1)-\mathcal{K}(Q_2))dt.
    \end{align}
Multiplying \eqref{unq} by $-D(r_2)\partial_x^{\alpha}\triangle(Q_1-Q_2)$ on both sides, taking the trace and integrating over $\mathbb{T}$, as the a priori estimates, we consider the first term
\begin{align*}
  &\quad-\int_{\mathbb{T}}\partial_x^{\alpha}(Q_1-Q_2)_t:D(r_2)\partial_x^{\alpha}\triangle(Q_1-Q_2)dx \\
  &=\frac{1}{2}\partial_t\int_{\mathbb{T}}D(r_2)|\partial_x^{\alpha}\nabla_x(Q_1-Q_2)|^2dx
  -\frac{1}{2}\int_{\mathbb{T}}D(r_2)_t|\partial_x^{\alpha}\nabla_x(Q_1-Q_2)|^2dx\\
  &\quad+\int_{\mathbb{T}}\nabla_xD(r_2)\partial_x^{\alpha}\nabla_x(Q_1-Q_2):\partial_x^{\alpha}(Q_1-Q_2)_tdx.
\end{align*}
Using \eqref{4.30} once more, similar estimate as \eqref{4.31} , estimate \eqref{2.1} and Lemma \ref{lem2.1}, the H\"{o}lder inequality
\begin{align*}
  &\quad\left|\frac{1}{2}\int_{\mathbb{T}}D(r_2)_t|\partial_x^{\alpha}\nabla_x(Q_1-Q_2)|^2dx\right|\leq C(R)\|Q_1-Q_2\|_{s,2}^2,\\
  &\quad\left|\int_{\mathbb{T}}\nabla_xD(r_2)\partial_x^{\alpha}\nabla_x(Q_1-Q_2):\partial_x^{\alpha}(Q_1-Q_2)_tdx\right|\\
  &\leq C(R)\|Q_1-Q_2\|_{s,2}\bigg(\|Q_1-Q_2\|_{s+1,2}+\Trucab\|Q_1\|_{s,2}\|\mathbf{u}_1-\mathbf{u}_2\|_{s,2}\\
  &\qquad\qquad\quad\qquad\quad+\sum_{j=1}^{2}(\|\mathbf{u}_j\|_{s+1,2}+\|Q_j\|_{s+2,2})(\|\mathbf{u}_1-\mathbf{u}_2\|_{s-1,2}+\|Q_1-Q_2\|_{s,2})\bigg)\\
  &\leq \frac{\Gamma L}{8}\int_{\mathbb{T}}D(r_2)|\triangle \partial_x^\alpha(Q_1-Q_2)|^2dx+\frac{\upsilon}{8}\Trucab \int_{\mathbb{T}}D(r_2)|\partial_x^s(\mathbf{u}_1-\mathbf{u}_2)|^2dx\\
  &\quad+C(R)\sum_{j=1}^{2}(1+\|Q_j\|_{s,2}^2+\|\mathbf{u}_j\|_{s+1,2}^2+\|Q_j\|_{s+2,2}^2)(\|\mathbf{u}_1-\mathbf{u}_2\|_{s-1,2}^2+\|Q_1-Q_2\|_{s,2}^2).
\end{align*}
We rewrite the highest-order term in \eqref{unq} as
\begin{align*}
  &\quad-\int_{\mathbb{T}}\Trucab(\partial_x^{\alpha}(\Theta_1-\Theta_2)Q_1-Q_1\partial_x^{\alpha}(\Theta_1-\Theta_2))
  :(-D(r_2)\partial_x^{\alpha}\triangle(Q_1-Q_2))dx\\
  &=\Trucab\int_{\mathbb{T}}D(r_2)(\partial_x^{\alpha}(\Theta_1-\Theta_2)Q_1-Q_1\partial_x^{\alpha}(\Theta_1-\Theta_2))
  :\triangle\partial_x^{\alpha}(Q_1-Q_2)dx,
\end{align*}
which can be cancelled with the forth term on the right hand side of \eqref{u2}.
Again, by Lemma \ref{lem2.1}, (\ref{5.2}) and the H\"{o}lder inequality
\begin{align*}
 &\quad\left(\Trucaa-\Trucab\right)\int_{\mathbb{T}} \partial_x^{\alpha}(\mathbf{u}_1\cdot\nabla_x Q_1-\Theta_1 Q_1+Q_1\Theta_1-\mathcal{K}(Q_1)):D(r_2)\partial_x^{\alpha}\triangle(Q_1-Q_2) dx\\
 &\leq \frac{\Gamma L}{8}\int_{\mathbb{T}}D(r_2)|\triangle \partial_x^\alpha(Q_1-Q_2)|^2dx+C\|\mathbf{u}_1\|_{s,2}^2\|Q_1\|_{s,2}^2(\|\mathbf{u}_1-\mathbf{u}_2\|_{s-1,2}^2+\|Q_1-Q_2\|_{s,2}^2).
\end{align*}
 After all the estimates we could have
\begin{align}\label{u3}
  &\frac{1}{2}d\|\sqrt{D(r_2)}\partial_x^{\alpha+1}(Q_1-Q_2)\|^2+\Gamma L\int_{\mathbb{T}}D(r_2)|\triangle \partial_x^\alpha(Q_1-Q_2)|^2dxdt\nonumber\\
  \leq &~ C\sum_{j=1}^{2}(1+\|\mathbf{u}_j\|_{s,2}^2+\|Q_j\|_{s+1,2}^2)(1+\|\mathbf{u}_j\|_{s+1,2}^2+\|Q_j\|_{s+2,2}^2)\nonumber\\
  &\times(\|\mathbf{u}_1-\mathbf{u}_2\|_{s-1,2}^2+\|Q_1-Q_2\|_{s,2}^2)dt\nonumber\\
  &-\Trucab\int_{\mathbb{T}}D(r_2)(\partial_x^{\alpha}(\Theta_1-\Theta_2)Q_1-Q_1\partial_x^{\alpha}(\Theta_1-\Theta_2))
  :\triangle\partial_x^{\alpha}(Q_1-Q_2)dxdt\nonumber\\
  &+\frac{\Gamma L}{4}\int_{\mathbb{T}}D(r_2)|\triangle \partial_x^\alpha(Q_1-Q_2)|^2dxdt+\frac{\upsilon}{4}\Trucab \int_{\mathbb{T}}D(r_2)|\partial_x^s(\mathbf{u}_1-\mathbf{u}_2)|^2dxdt.
\end{align}

Adding \eqref{u1}, \eqref{u2} and \eqref{u3}, taking sum for $|\alpha|\leq s-1$, also using the fact that $\frac{1}{C(R)}\leq D(r_2)\leq C(R)$, then the following holds
\begin{align*}
  &\quad d(\|r_1-r_2,\mathbf{u}_1-\mathbf{u}_2\|_{s-1,2}^2+\|Q_1-Q_2\|_{s,2}^2)\\
  &\leq  C(R)\sum_{j=1}^{2}(1+\|r_j, \mathbf{u}_j\|_{s,2}^2+\|Q_j\|_{s+1,2}^3)(1+\|\mathbf{u}_j\|_{s+1,2}^2+\|Q_j\|_{s+2,2}^2)\\
  &\quad\times (\|r_1-r_2, \mathbf{u}_1-\mathbf{u}_2\|_{s-1,2}^2+\|Q_1-Q_2\|_{s,2}^2)dt \\
  &\quad+  C\sum_{|\alpha|\leq s-1}\left(\Trucaa\partial_x^{\alpha}\mathbb{F}(r_1,\mathbf{u}_1)-\Trucab \partial_x^{\alpha}\mathbb{F}(r_2,\mathbf{u}_2), \partial_x^{\alpha}(\mathbf{u}_1-\mathbf{u}_2)\right)dW.
\end{align*}
Denote
$$G(t)=C(R)\sum_{j=1}^{2}(1+\|r_j, \mathbf{u}_j\|_{s,2}^2+\|Q_j\|_{s+1,2}^3)(1+\|\mathbf{u}_j\|_{s+1,2}^2+\|Q_j\|_{s+2,2}^2).$$
Then we could apply the It\^o product formula to function
\begin{eqnarray*}
\exp\left(-\int_{0}^{t}G(\tau)d\tau\right)(\|r_1-r_2,\mathbf{u}_1-\mathbf{u}_2\|_{s-1,2}^2+\|Q_1-Q_2\|_{s,2}^2),
\end{eqnarray*}
obtaining
\begin{align*}
  &d\left[\exp\left(-\int_{0}^{t}G(\tau)d\tau\right)(\|r_1-r_2,\mathbf{u}_1-\mathbf{u}_2\|_{s-1,2}^2+\|Q_1-Q_2\|_{s,2}^2)\right]\\
  =&\left[-G(t)\exp\left(-\int_{0}^{t}G(\tau)d\tau\right)(\|r_1-r_2, \mathbf{u}_1-\mathbf{u}_2\|_{s-1,2}^2+\|Q_1-Q_2\|_{s,2}^2)\right]dt\\
  &+\exp\left(-\int_{0}^{t}G(\tau)d\tau\right)d(\|r_1-r_2, \mathbf{u}_1-\mathbf{u}_2\|_{s-1,2}^2+\|Q_1-Q_2\|_{s,2}^2)\\
  \leq  &~C(R) \sum_{|\alpha|\leq s-1}\left(\Trucaa\partial_x^{\alpha}\mathbb{F}(r_1,\mathbf{u}_1)-\Trucab \partial_x^{\alpha}\mathbb{F}(r_2,\mathbf{u}_2), \partial_x^{\alpha}(\mathbf{u}_1-\mathbf{u}_2)\right)dW\nonumber \\
  &\times \exp\left(-\int_{0}^{t}G(\tau)d\tau\right).
\end{align*}
Integrating on $[0,t]$ and then expectation, we have by the Gronwall lemma
$$\mathbb{E}\left[\exp\left(-\int_{0}^{t}G(\tau)d\tau\right)(\|r_1-r_2,\mathbf{u}_1-\mathbf{u}_2\|_{s-1,2}^2+\|Q_1-Q_2\|_{s,2}^2)\right]=0.$$
Here, we use the fact that the stochastic integral term is a square integral martingale which its expectation vanishes. As
\begin{eqnarray*}
\exp\left(-\int_{0}^{t}G(\tau)d\tau\right)>0,\quad  \mathbb{P}~\mbox{a.s.}
\end{eqnarray*}
since
\begin{align*}
&\quad\int_{0}^{t}G(\tau)d\tau \\ &\leq  \sum_{j=1}^{2}\sup_{t\in [0,T]}(1+\|r_j, \mathbf{u}_j\|_{s,2}^2+\|Q_j\|_{s+1,2}^3)\int_{0}^{T}1+\|\mathbf{u}_j\|_{s+1,2}^2+\|Q_j\|_{s+2,2}^2dt\\
&\leq C\sum_{j=1}^{2}\left[\sup_{t\in [0,T]}(1+\|r_j, \mathbf{u}_j\|_{s,2}^2+\|Q_j\|_{s+1,2}^3)^2+\left(\int_{0}^{T}1+\|\mathbf{u}_j\|_{s+1,2}^2+\|Q_j\|_{s+2,2}^2dt\right)^2\right]\\
&<\infty, \quad \mathbb{P}~ \mbox{a.s.}.
\end{align*}
We conclude that for any $t\in [0,T]$
$$\mathbb{E}\left(\|r_1-r_2,\mathbf{u}_1-\mathbf{u}_2\|_{s-1,2}^2+\|Q_1-Q_2\|_{s,2}^2\right)=0,$$
then the pathwise uniqueness holds.
\end{proof}

From the uniqueness, we shall use the following Gy\"{o}ngy-Krylov characterization which can be found in \cite{Krylov} to recover the convergence a.s. of the approximate solution on the original probability space $(\Omega, \mathcal{F}, \mathbb{P})$.

\begin{lemma}\label{lem5.2}
Let $X$ be a complete separable metric space and suppose that $\{Y_{n}\}_{n\geq0}$ is a sequence of $X$-valued random variables on a probability space $(\Omega,\mathcal{F},\mathbb{P})$. Let $\{\mu_{m,n}\}_{m,n\geq1}$ be the set of joint laws of $\{Y_{n}\}_{n\geq1}$, that is
\begin{equation*}
\mu_{m,n}(E):=\mathbb{P}\{(Y_{n},Y_{m})\in E\},~~~E\in\mathcal{B}(X\times X).
\end{equation*}
Then $\{Y_{n}\}_{n\geq1}$ converges in probability if and only if for every subsequence of the joint probability laws $\{\mu_{m_{k},n_{k}}\}_{k\geq1}$, there exists a further subsequence that converges weakly to a probability measure $\mu$ such that
\begin{equation*}
\mu\{(u,v)\in X\times X: u=v\}=1.
\end{equation*}
\end{lemma}

Next, we verify the condition for the above lemma is valid.  Denote by $\mu_{n,m}$ the joint law of
\begin{equation*}
(r_n, \mathbf{u}_{n}, Q_{n};r_m, \mathbf{u}_{m}, Q_{m})~~~~ {\rm on~ the ~path ~space}~\mathcal{X}=\mathcal{X}_{r}\times \mathcal{X}_{\mathbf{u}}\times \mathcal{X}_{Q}\times \mathcal{X}_{r}\times \mathcal{X}_{\mathbf{u}}\times \mathcal{X}_{Q},
\end{equation*}
where $\{r_n, \mathbf{u}_{n}, Q_{n};r_m, \mathbf{u}_{m}, Q_{m}\}_{n,m\geq 1}$ are two sequences of approximate solutions to system \eqref{qnt} relative to the given stochastic basis $\mathcal{S}$, and denote by $\mu_{W}$ the law of $W$ on $\mathcal{X}_{W}$. We introduce the extended phase space
\begin{equation*}
\mathcal{X}^J=\mathcal{X}\times\mathcal{X}_{W},
\end{equation*}
and denote by $\nu_{n,m}$ the joint law of $(r_n, \mathbf{u}_{n}, Q_{n};r_m, \mathbf{u}_{m}, Q_{m}; W)~~~~{\rm on}~~\mathcal{X}^J$. Using a similar argument as the proof of the tightness in subsection 4.3, we obtain the following result.
\begin{proposition}\label{pro5.5}
  The collection of joint laws $\{\nu_{m,n}\}_{n,m\geq 1}$ is tight on $\mathcal{X}^J$.
\end{proposition}

For any subsequence $\{\nu_{n_{k},m_{k}}\}_{k\geq 1}$, there exists a measure $\nu$ such that $\{\nu_{n_{k},m_{k}}\}_{k\geq 1}$ converges to $\nu$. Applying the Skorokhod representation theorem \ref{thm4.1}, we have a new probability space $(\tilde{\Omega},\tilde{\mathcal{F}},\tilde{\mathbb{P}})$ and $\mathcal{X}^J$-valued random variables
\begin{eqnarray*}
(\tilde{r}_{n_k}, \tilde{\mathbf{u}}_{n_k}, \tilde{Q}_{n_k};\tilde{r}_{m_k}, \tilde{\mathbf{u}}_{m_k}, \tilde{Q}_{m_k}; \tilde{W}_{k})~ {\rm and}~(\tilde{r}_1, \tilde{\mathbf{u}}_{1}, \tilde{Q}_{1};\tilde{r}_2, \tilde{\mathbf{u}}_{2}, \tilde{Q}_{2}; \tilde{W})
\end{eqnarray*}
such that
\begin{eqnarray*}
 &&\tilde{\mathbb{P}}\{(\tilde{r}_{n_k}, \tilde{\mathbf{u}}_{n_k}, \tilde{Q}_{n_k};\tilde{r}_{m_k}, \tilde{\mathbf{u}}_{m_k}, \tilde{Q}_{m_k}; \tilde{W}_{k})\in \cdot\}=\nu_{n_{k},m_{k}}(\cdot),\\
 &&\tilde{\mathbb{P}}\{(\tilde{r}_1, \tilde{\mathbf{u}}_{1}, \tilde{Q}_{1};\tilde{r}_2, \tilde{\mathbf{u}}_{2}, \tilde{Q}_{2}; \tilde{W})\in \cdot\}=\nu(\cdot)
\end{eqnarray*}
 and
\begin{eqnarray*}
(\tilde{r}_{n_k}, \tilde{\mathbf{u}}_{n_k}, \tilde{Q}_{n_k};\tilde{r}_{m_k}, \tilde{\mathbf{u}}_{m_k}, \tilde{Q}_{m_k}; \tilde{W}_{k})\rightarrow (\tilde{r}_1, \tilde{\mathbf{u}}_{1}, \tilde{Q}_{1};\tilde{r}_2, \tilde{\mathbf{u}}_{2}, \tilde{Q}_{2}; \tilde{W}),~~\tilde{\mathbb{P}}~\mbox{a.s.}
\end{eqnarray*}
in the topology of $\mathcal{X}^J$. Analogously, this argument can be applied to both
\begin{equation*}
(\tilde{r}_{n_k}, \tilde{\mathbf{u}}_{n_k}, \tilde{Q}_{n_k},\tilde{W}_{k}),
~~(\tilde{r}_1, \tilde{\mathbf{u}}_{1}, \tilde{Q}_{1},\tilde{W}) \hspace{.3cm} \text{and} \hspace{.3cm}
(\tilde{r}_{m_k}, \tilde{\mathbf{u}}_{m_k}, \tilde{Q}_{m_k},\tilde{W}_{k}),~~(\tilde{r}_2,\tilde{ \mathbf{u}}_{2}, \tilde{Q}_{2},\tilde{W})
\end{equation*}
to show that $(\tilde{r}_1, \tilde{\mathbf{u}}_{1}, \tilde{Q}_{1},\tilde{W})$ and $(\tilde{r}_2, \tilde{\mathbf{u}}_{2}, \tilde{Q}_{2},\tilde{W})$ are two martingale solutions relative to the same stochastic basis $\widetilde{\mathcal{S}}:=(\tilde{\Omega},\tilde{\mathcal{F}},\tilde{\mathbb{P}},\{\tilde{\mathcal{F}}_{t}\}_{t\geq 0},\tilde{W})$.

In addition, we have $\mu_{n,m}\rightharpoonup \mu$ where $\mu$ is defined by
$$\mu(\cdot)=\tilde{\mathbb{P}}\{(\tilde{r}_1, \tilde{\mathbf{u}}_{1}, \tilde{Q}_{1};\tilde{r}_2, \tilde{\mathbf{u}}_{2}, \tilde{Q}_{2})\in \cdot\}.$$
Proposition \ref{pro3.3} implies that $\mu\{(r_1, \mathbf{u}_{1}, Q_{1};r_2, \mathbf{u}_{2}, Q_{2})\in \mathcal{X}:(r_1, \mathbf{u}_{1}, Q_{1})=(r_2, \mathbf{u}_{2}, Q_{2})\}=1$. Also since $W^{s,2}\subset W^{s-1,2}$, uniqueness in $W^{s-1,2}$ implies uniqueness in $W^{s,2}$. Therefore, Lemma \ref{lem5.2} can be used to deduce that the sequence $(r_n, \mathbf{u}_{n}, Q_{n})$ defined on the original probability space $(\Omega,\mathcal{F},\mathbb{P})$ converges a.s. in the topology of $\mathcal{X}_r\times \mathcal{X}_{\mathbf{u}}\times \mathcal{X}_{Q}$ to random variable $(r, \mathbf{u}, Q)$.

 Again by the same argument as in subsection 4.4, we get the Theorem \ref{th4.2} in the sense of Definition \ref{def2}.

\section{\bf Proof of Theorem \ref{thm2.5}.}
In the process of proving the Theorem \ref{th4.2}, it's worth noting that due to technical reason, we assume that the initial data is integrable with respect to the random element $\omega$, and that the density is uniformly bounded from below. Next, based on the Theorem \ref{th4.2}, we are able to remove these restrictions on the initial data and discuss the general case, thus the proof of the main Theorem \ref{thm2.5} will be completed.

We start with the proof of the existence of the strong pathwise solution, which is divided into three steps. For the first step, we show the existence of the strong pathwise solution under the assumption that the initial data satisfies
\begin{align}\label{6.1}
\rho_0>\underline{\rho}> 0,~ \|\rho_0\|_{s,2}\leq M,~ \|\mathbf{u}_0\|_{s,2}\leq M, ~\|Q_0\|_{s+1,2}\leq M, ~Q_0\in S_0^3,
\end{align}
for a fixed constant $M>0$ such that $R>\mathcal{C}M$, where $\mathcal{C}$ is a constant satisfying
$$\|\mathbf{u}\|_{2,\infty}<\mathcal{C}\|\mathbf{u}\|_{s-1,2}, \|Q\|_{3,\infty}<\mathcal{C}\|Q\|_{s,2}.$$
Introduce a stopping time $\tau_R=\tau_R^1\wedge \tau_R^2$, where
\begin{eqnarray*}
\tau^1_R=\inf\left\{t\in[0,T];\sup_{\gamma\in [0,t]}\|\mathbf{u}_R\|_{2,\infty}\geq R\right\},~\tau^2_R=\inf\left\{t\in[0,T];\sup_{\gamma\in [0,t]}\|Q_R\|_{3,\infty}\geq R\right\}.
\end{eqnarray*}
If two sets are empty, choosing $\tau^i_R=T,i=1,2$. The fact that $\mathbf{u}, Q$ having continuous trajectories in $W^{s-1,2}(\mathbb{T}, \mathbb{R}^3)$ and in $W^{s,2}(\mathbb{T},S_0^3)$ for integer $s>\frac{9}{2}$ respectively and the Sobolev embedding $W^{s,2}\hookrightarrow W^{\alpha,\infty}$ for $s>\frac{3}{2}+\alpha$, $\mathbb{P}$~\mbox{a.s.} guarantee the well-defined of $\tau_R$ and strictly positive $\mathbb{P}$ a.s..

 Since $r_R(t, \cdot)\geq \mathcal{C}(R)>0,~ \mathbb{P}~ \mbox{a.s.} ~{\rm for ~ all} ~t\in [0,T]$, we could construct a local strong pathwise solution $(\rho_R,\mathbf{u}_R ,Q_R,\tau_R)$ of system \eqref{qn}, based on the existence of unique pathwise solution $(r_R,\mathbf{u}_R,Q_R)$ of the truncated system \eqref{qnt} with initial data conditions (\ref{6.1}), where $\rho_R=\left(\frac{\gamma-1}{2A\gamma}\right)^\frac{1}{\gamma-1}r_R^\frac{2}{\gamma-1}$.

For the second step, we drop the auxiliary boundedness assumption of the initial data following the ideas of \cite{Glatt-Holtz}. For the solution $(r_R,\mathbf{u}_R,Q_R)$ of the system \eqref{qnt}, define the following stopping time
\begin{align*}
    &\tau_M^1=\inf\left\{t\in[0,T];\sup_{\gamma\in [0,t]}\|\mathbf{u}_R\|_{s,2}\geq M\right\}, \\
    &\tau_M^2=\inf\left\{t\in[0,T];\sup_{\gamma\in [0,t]}\|Q_R\|_{s+1,2}\geq M\right\}, \\
    &\tau_M^3=\inf\left\{t\in[0,T];\sup_{\gamma\in [0,t]}\|r_R\|_{s,2}\geq M\right\}, \\
    &\tau_M^4=\inf\left\{t\in[0,T];\inf_{x\in \mathbb{T}}r_R(t)\leq \frac{1}{M}\right\},
\end{align*}
where $M$ relies on $R$ such that $M\rightarrow\infty$ as $R\rightarrow \infty$ and $M\leq {\rm min}\left(\frac{R}{\mathcal{C}}, R\right)$. Then we could define $\tau_M=\tau_M^1\wedge\tau_M^2\wedge\tau_M^3\wedge\tau_M^4$, such that in $[0,\tau_M]$, again using the Sobolev embedding $W^{s,2}\hookrightarrow W^{\alpha,\infty}$ for $s>\frac{3}{2}+\alpha$, $\mathbb{P}$~\mbox{a.s.}, we have
\begin{align*}
&\sup_{t\in [0,\tau_M]}\|r_R(t)\|_{1,\infty}<R,~\sup_{t\in [0,\tau_M]}\|\mathbf{u}_R(t)\|_{2,\infty}<R,\\ &\sup_{t\in [0,\tau_M]}\|Q_R(t)\|_{3,\infty}<R,
~\inf_{t\in [0,\tau_M]}\inf_{\mathbb{T}}r_R(t)>\frac{1}{R}.
\end{align*}
 According to the Theorem \ref{th4.2}, we could construct the solution with respect to the stopping time $\tau_M$  for the general data. Indeed, define
\begin{align*}
  &\Sigma_{M}=\bigg\{(r,\mathbf{u},Q)\in W^{s,2}(\mathbb{T})\times W^{s,2}(\mathbb{T},\mathbb{R}^3)\times W^{s+1,2}(\mathbb{T},S_0^3):\\
&\qquad\qquad\|r(t)\|_{s,2}<M,\|\mathbf{u}(t)\|_{s,2}<M,\|Q(t)\|_{s+1,2}<M,r(t)>\frac{1}{M}\bigg\},
\end{align*}
then, we have there exists a unique solution $(r_M, \mathbf{u}_M, Q_M)$ to system \eqref{qnt1} with the initial data $(r_0, \mathbf{u}_0, Q_0)\mathbf{1}_{(r_0,\mathbf{u}_0,Q_0)\in \Sigma_{M}\backslash \cup_{j=1}^{M-1}\Sigma_{j}}$, which is also a solution to the original system \eqref{qn} with the stopping time $\tau_M$.

Define
$$\tau=\sum_{M=1}^{\infty}\tau_{M}\mathbf{1}_{(r_0,\mathbf{u}_0,Q_0)\in \Sigma_{M}\backslash \cup_{j=1}^{M-1}\Sigma_{j}},$$
$$(r,\mathbf{u},Q)=\sum_{M=1}^{\infty}(r_M,\mathbf{u}_M,Q_M)\mathbf{1}_{(r_0,\mathbf{u}_0,Q_0)\in \Sigma_{M}\backslash \cup_{j=1}^{M-1}\Sigma_{j}}.$$
Using the same argument as \cite[Proposition 4.2]{Glatt-Holtz1}, we infer that the $(r,\mathbf{u},Q, \tau)$ is a solution to system \eqref{qnt1} with the initial condition $(r_0,\mathbf{u}_0,Q_0)$ being $\mathcal{F}_0$-measurable random variable, with values in $W^{s,2}(\mathbb{T})\times W^{s,2}(\mathbb{T},\mathbb{R}^3)\times W^{s+1,2}(\mathbb{T},S_0^3)$ and $r_0>0$, $\mathbb{P}$ ~a.s..

Next, we show $(r,\mathbf{u},Q)$ has continuous trajectory in $W^{s,2}(\mathbb{T})\times W^{s-1,2}(\mathbb{T},\mathbb{R}^3)\times W^{s,2}(\mathbb{T},S_0^3)$, $\mathbb{P}$ ~a.s.. Define
\begin{eqnarray*}
\Omega_M=\left\{\omega\in \Omega:\|r_0(\omega)\|_{s,2}<M,\|\mathbf{u}_0(\omega)\|_{s,2}<M,\|Q_0(\omega)\|_{s+1,2}<M,r_0(\omega)>\frac{1}{M}\right\}.
\end{eqnarray*}
Observe that $\bigcup_{M=1}^\infty\Omega_M=\Omega$. Therefore, for any $\omega\in\Omega$, there exists a set $\Omega_M$ such that $\omega\in\Omega_M$, and by the construction, we have $(r,\mathbf{u},Q)(\omega)=(r_M,\mathbf{u}_M,Q_M)(\omega)$. Since $(r_M,\mathbf{u}_M,Q_M)$ has continuous trajectories in $W^{s,2}(\mathbb{T})\times W^{s-1,2}(\mathbb{T},\mathbb{R}^3)\times W^{s,2}(\mathbb{T}, S_0^3)$ and $r_M(t\wedge \tau_M,\cdot)>\mathcal{C}(M)$, $\mathbb{P}$ a.s. for all $t\in [0,T]$, then we deduce that $(r,\mathbf{u},Q)$ has continuous trajectories in $W^{s,2}(\mathbb{T})\times W^{s-1,2}(\mathbb{T},\mathbb{R}^3)\times W^{s,2}(\mathbb{T},S_0^3)$, $\mathbb{P}$ a.s. and $r(t\wedge \tau,\cdot)>0$, $\mathbb{P}$ a.s. for all $t\in [0,T]$. In addition, for the fixed $\omega$, we have $\Phi_{R}^{\mathbf{u}, Q}=1$ on $[0, \tau_M(\omega)]$, thus $\mathbf{u}_M \mathbf{1}_{t\leq\tau_M}\in L^2(0,T; W^{s+1,2}(\mathbb{T},\mathbb{R}^3))$, then by the construction, we deduce that $\mathbf{u}\mathbf{1}_{t\leq\tau}\in L^2(0,T; W^{s+1,2}(\mathbb{T},\mathbb{R}^3))$, $\mathbb{P}$ a.s..

Finally, since $r(t\wedge \tau,\cdot)>0$, $\mathbb{P}$ a.s. for all $t\in [0,T]$, after a transformation, we summarize that if $(\rho_0,\mathbf{u}_0,Q_0)$ just lies in $W^{s,2}(\mathbb{T})\times W^{s,2}(\mathbb{T},\mathbb{R}^3)\times W^{s+1,2}(\mathbb{T},S_0^3)$ and $\rho_0>0$, $\mathbb{P}$ a.s. this means dropping the integrability with respect to $\omega$ and the positive lower bound of $\rho_0$, we establish the existence of a local strong pathwise solution $(\rho,\mathbf{u},Q)$ to system \eqref{qn} in the sense of Definition \ref{de1}, up to a stopping time $\tau$ which is strictly positive, $\mathbb{P}$ a.s..

The final step would be constructing the maximal strong solutions. That is, extending the strong solution $(\rho,\mathbf{u},Q)$ to a maximal existence time $\mathfrak{t}$. The proof is standard, so we refer the reader to \cite{Seidler,Glatt-Holtz,Rozovskii} for details.

Regarding the proof of uniqueness to Theorem \ref{thm2.5}, first, under the assumption (\ref{6.1}), we could prove the uniqueness result by introducing a stopping time and applying the pathwise uniqueness result derived before. Then, we can remove the extra assumption on the initial data by a same cutting argument as above. This completes the proof of Theorem \ref{thm2.5}.

\section{Appendix}

In the appendix, we present some classical results that could be used in this paper.
\begin{lemma}(The Aubin-Lions Lemma, \!{\rm \cite[Chapter I]{LP1}}\label{lem6.1}) Suppose that $X_{1}\subset X_{0}\subset X_{2}$ are Banach spaces, $X_{1}$ and $X_{2}$ are reflexive,  and the embedding of $X_{1}$ into $X_{0}$ is compact.
Then for any $1<p<\infty,~ 0<\alpha<1$, the embedding
\begin{align*}
L^{p}(0,T;X_{1})\cap W^{\alpha,p}(0,T;X_{2})\hookrightarrow L^{p}(0,T;X_{0}),\\
L^{\infty}(0,T;X_{1})\cap C^{\alpha}([0,T];X_{2})\hookrightarrow L^\infty(0,T;X_{0})
\end{align*}
is compact.
\end{lemma}

\begin{theorem}(The Vitali convergence theorem, \!{\rm \cite[Chapter 3]{kall}}\label{thm6.1}) Let $p\geq 1$, $\{u_n\}_{n\geq 1}\in L^p$ and $u_n\rightarrow u$ in probability. Then, the following are equivalent\\
{\rm (1)}. $u_n\rightarrow u$ in $L^p$;\\
{\rm (2)}. the sequence $|u_n|^p$ is uniformly integrable;\\
{\rm (3)}. $\mathbb{E}|u_n|^p\rightarrow \mathbb{E}|u|^p$.
\end{theorem}
\begin{theorem}(The Skorokhod representation theorem, \!{\rm \cite[Theorem 1]{Sko}}\label{thm4.1}) Let $X$ be a Polish space. If the set of probability measures $\{\nu_n\}_{n\geq 1}$ on $\mathcal{B}(X)$ is tight, then there exists a probability space $(\Omega, \mathcal{F}, \mathbb{P})$ and a sequence of random variables $u_n, u$ such that theirs laws are $\nu_n$, $\nu$ and $u_n\rightarrow u$, $\mathbb{P}$ a.s. as $n\rightarrow \infty$.
\end{theorem}

\section*{Acknowledgments}
Z. Qiu's research was supported by the CSC under grant No.201806160015.
Y. Wang is partially supported by the NSF grant   DMS-1907519.

\end{document}